\documentclass[dvipsnames,10pt]{amsart}

\usepackage[export]{adjustbox}

\usepackage[hmargin=35mm,vmargin=25mm]{geometry}

\usepackage{multirow, hhline}
\usepackage{verbatim}
\usepackage[shortlabels]{enumitem}
\usepackage{tikz-cd}
\usepackage{xcolor}
\usepackage{hyperref}

\usepackage[textsize=small, color=red]{todonotes}

\usepackage{amssymb} \usepackage{amsbsy,
  amsthm, amsmath,amstext, amsopn} \usepackage[all]{xy}
\usepackage{amsfonts} \usepackage{amscd} \usepackage{stmaryrd}
\usepackage{mathabx}

\newtheorem{thm}{Theorem}[section]
\newtheorem{lemma}[thm]{Lemma}

\newtheorem{corollary}[thm]{Corollary}
\newtheorem{cor}[thm]{Corollary}
\newtheorem{prop}[thm]{Proposition}

\theoremstyle{definition}

\newtheorem{defn}[thm]{Definition}

\newtheorem{conj}[thm]{Conjecture}
\newtheorem{example}[thm]{Example}
\newtheorem{exmp}[thm]{Example}

\theoremstyle{remark}

\newtheorem{remark}[thm]{Remark}
\newtheorem{rmk}[thm]{Remark}

\DeclareMathOperator*{\colim}{colim}

\begin{document}

\numberwithin{equation}{section}

\newcommand{\hs}{\mbox{\hspace{.4em}}}
\newcommand{\ds}{\displaystyle}
\newcommand{\bd}{\begin{displaymath}}
\newcommand{\ed}{\end{displaymath}}
\newcommand{\bcd}{\begin{CD}}
\newcommand{\ecd}{\end{CD}}

\newcommand{\on}{\operatorname}
\newcommand{\proj}{\operatorname{Proj}}
\newcommand{\bproj}{\underline{\operatorname{Proj}}}
\newcommand{\spec}{\operatorname{Spec}}
\newcommand{\Spec}{\operatorname{Spec}}
\newcommand{\bspec}{\underline{\operatorname{Spec}}}
\newcommand{\pline}{{\mathbf P} ^1}
\newcommand{\aline}{{\mathbf A} ^1}
\newcommand{\pplane}{{\mathbf P}^2}
\newcommand{\cone}{\operatorname{cone}}
\newcommand{\aplane}{{\mathbf A}^2}
\newcommand{\coker}{{\operatorname{coker}}}
\newcommand{\ldb}{[[}
\newcommand{\rdb}{]]}

\newcommand{\Sym}{\operatorname{Sym}^{\bullet}}
\newcommand{\Symp}{\operatorname{Sym}}
\newcommand{\Pic}{\bf{Pic}}
\newcommand{\Aut}{\operatorname{Aut}}
\newcommand{\codim}{\operatorname{codim}}
\newcommand{\PAut}{\operatorname{PAut}}

\newcommand{\Fqbar}{\overline{\mathbb{F}}_q}
\newcommand{\Fq}{{\mathbb{F}_q}}

\newcommand{\too}{\twoheadrightarrow}
\newcommand{\C}{{\mathbb C}}
\newcommand{\Z}{{\mathbb Z}}
\newcommand{\Q}{{\mathbb Q}}
\newcommand{\R}{{\mathbb R}}
\newcommand{\Cx}{{\mathbb C}^{\times}}
\newcommand{\Cbar}{\overline{\C}}
\newcommand{\Cxbar}{\overline{\Cx}}
\newcommand{\cA}{{\mathcal A}}
\newcommand{\fA}{{\mathfrak A}}
\newcommand{\cS}{{\mathcal S}}
\newcommand{\cV}{{\mathcal V}}
\newcommand{\cM}{{\mathcal M}}
\newcommand{\bA}{{\mathbf A}}
\newcommand{\cB}{{\mathcal B}}
\newcommand{\cC}{{\mathcal C}}
\newcommand{\cD}{{\mathcal D}}
\newcommand{\D}{{\mathcal D}}
\newcommand{\cs}{{\mathbf C} ^*}
\newcommand{\boldc}{{\mathbf C}}
\newcommand{\cE}{{\mathcal E}}
\newcommand{\cF}{{\mathcal F}}
\newcommand{\bF}{{\mathbb F}}
\newcommand{\cG}{{\mathcal G}}
\newcommand{\G}{{\mathbb G}}
\newcommand{\cH}{{\mathcal H}}
\newcommand{\bH}{{\mathbf H}}
\newcommand{\CI}{{\mathcal I}}
\newcommand{\cJ}{{\mathcal J}}
\newcommand{\cK}{{\mathcal K}}
\newcommand{\cL}{{\mathcal L}}
\newcommand{\baL}{{\overline{\mathcal L}}}
\newcommand{\M}{{\mathcal M}}
\newcommand{\Mf}{{\mathfrak M}}
\newcommand{\bM}{{\mathbf M}}
\newcommand{\bm}{{\mathbf m}}
\newcommand{\cN}{{\mathcal N}}
\newcommand{\theo}{\mathcal{?}{O}}
\newcommand{\cP}{{\mathcal P}}
\newcommand{\cR}{{\mathcal R}}
\newcommand{\Pp}{{\mathbb P}}
\newcommand{\boldp}{{\mathbf P}}
\newcommand{\boldq}{{\mathbf Q}}
\newcommand{\bbL}{{\mathbf L}}
\newcommand{\cQ}{{\mathcal Q}}
\newcommand{\cO}{{\mathcal O}}
\newcommand{\cT}{{\mathcal T}}
\newcommand{\Oo}{{\mathcal O}}
\newcommand{\cY}{{\mathcal Y}}
\newcommand{\OX}{{\Oo_X}}
\newcommand{\OY}{{\Oo_Y}}
\newcommand{\all}{{\mathrm{all}}}
\newcommand{\cZ}{{\mathcal Z}}
\newcommand{\bD}{\mathbb{D}}
\newcommand{\DMod}{\mathcal{D}}
\newcommand{\cDMod}{\breve{\mathcal{D}}}
\newcommand{\rnD}{\cDMod}
\newcommand{\otY}{{\underset{\OY}{\ot}}}
\newcommand{\otX}{{\underset{\OX}{\ot}}}
\newcommand{\cU}{{\mathcal U}}\newcommand{\cX}{{\mathcal X}}
\newcommand{\cW}{{\mathcal W}}
\newcommand{\boldz}{{\mathbf Z}}
\newcommand{\Rees}{\operatorname{Rees}}
\newcommand{\IC}{\IndCoh}
\newcommand{\ssupp}{\operatorname{SS}}
\newcommand{\qgr}{\operatorname{q-gr}}
\newcommand{\gr}{\operatorname{gr}}
\newcommand{\rk}{\operatorname{rk}}
\newcommand{\Sh}{\operatorname{Sh}}
\newcommand{\Shv}{\operatorname{Sh}}
\newcommand{\SH}{{\underline{\operatorname{Sh}}}}
\newcommand{\End}{\operatorname{End}}
\newcommand{\uEnd}{\underline{\operatorname{End}}}
\newcommand{\Hom}{\operatorname{Hom}}
\newcommand{\uHom}{\underline{\operatorname{Hom}}}
\newcommand{\Sing}{\operatorname{Sing}}
\newcommand{\uHomY}{\uHom_{\OY}}
\newcommand{\uHomX}{\uHom_{\OX}}
\newcommand{\Ext}{\operatorname{Ext}}
\newcommand{\bExt}{\operatorname{\bf{Ext}}}
\newcommand{\Tor}{\operatorname{Tor}}

\newcommand*\brslash{\vcenter{\hbox{\includegraphics[height=3.6mm]{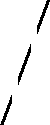}}}}

\newcommand{\inv}{^{-1}}
\newcommand{\airtilde}{\widetilde{\hspace{.5em}}}
\newcommand{\airhat}{\widehat{\hspace{.5em}}}
\newcommand{\nt}{^{\circ}}
\newcommand{\del}{\partial}

\newcommand{\supp}{\operatorname{supp}}
\newcommand{\GK}{\operatorname{GK-dim}}
\newcommand{\hd}{\operatorname{hd}}
\newcommand{\pt}{\operatorname{pt}}
\newcommand{\id}{\operatorname{id}}
\newcommand{\res}{\operatorname{res}}
\newcommand{\lrar}{\leadsto}
\newcommand{\im}{\operatorname{Im}}
\newcommand{\hh}{HH}
\newcommand{\hn}{HN}
\newcommand{\hc}{HC}
\newcommand{\hp}{HP}

\newcommand{\TF}{\operatorname{TF}}
\newcommand{\Bun}{\operatorname{Bun}}

\newcommand{\mix}{{\mathrm{m}}}

\newcommand{\F}{\mathcal{F}}
\newcommand{\Ff}{\mathbb{F}}
\newcommand{\nthord}{^{(n)}}
\newcommand{\Gr}{{\mathfrak{Gr}}}

\newcommand{\BB}{\mathbb{B}}

\newcommand{\Ting}{\mathbb{T}^*_{[\mh 1]}}

\newcommand{\bL}{\mathbb{L}}
\newcommand{\deeq}{{\mathrm{dq}}}
\newcommand{\Fr}{\operatorname{Fr}}
\newcommand{\GL}{\operatorname{GL}}
\newcommand{\Perv}{\operatorname{Perv}}
\newcommand{\gl}{\mathfrak{gl}}
\newcommand{\SL}{\operatorname{SL}}
\newcommand{\KPerf}{\operatorname{KPerf}}
\newcommand{\ff}{\footnote}
\newcommand{\ot}{\otimes}
\def\Ext{\operatorname {Ext}}
\def\Hom{\operatorname {Hom}}
\def\Ind{\operatorname {Ind}}
\newcommand{\MHM}{\operatorname{MHM}}

\def\bbZ{{\mathbb Z}}

\newcommand{\Irr}{\mathrm{Irr}}

\newcommand{\nc}{\newcommand}
\nc{\ol}{\overline} \nc{\cont}{\on{cont}} \nc{\rmod}{\on{mod}}
\nc{\Mtil}{\widetilde{M}} \nc{\wb}{\overline} \nc{\wt}{\widetilde}
\nc{\wh}{\widehat} \nc{\sm}{\setminus} \nc{\mc}{\mathcal}
\nc{\mbb}{\mathbb}  \nc{\K}{{\mc K}} \nc{\Kx}{{\mc K}^{\times}}
\nc{\Ox}{{\mc O}^{\times}} \nc{\unit}{{\bf \on{unit}}}
\nc{\boxt}{\boxtimes} \nc{\xarr}{\stackrel{\rightarrow}{x}}

\newcommand{\bC}{\mathbb{C}}

\nc{\Ga}{\G_a}
 \nc{\PGL}{{\on{PGL}}}
 \nc{\PU}{{\on{PU}}}

\nc{\h}{{\mathfrak h}} \nc{\kk}{{\mathfrak k}}
 \nc{\Gm}{\G_m}
\nc{\Gabar}{\wb{\G}_a} \nc{\Gmbar}{\wb{\G}_m} \nc{\Gv}{G^\vee}
\nc{\Tv}{T^\vee} \nc{\Bv}{B^\vee} 
\nc{\g}{{\mathfrak g}}
\nc{\gv}{{\mathfrak g}^\vee} \nc{\RGv}{\on{Rep}\Gv}
\nc{\RTv}{\on{Rep}T^\vee}
 \nc{\Flv}{{\mathcal B}^\vee}
 \nc{\TFlv}{T^*\Flv}
 \nc{\Fl}{{\mathfrak Fl}}
\nc{\RR}{{\mathcal R}} \nc{\Nv}{{\mathcal{N}}^\vee}
\nc{\St}{{\mathcal St}} \nc{\ST}{{\underline{\mathcal St}}}
\nc{\Hec}{{\bf{\mathcal H}}} \nc{\Hecblock}{{\bf{\mathcal
H_{\alpha,\beta}}}} \nc{\dualHec}{{\bf{\mathcal H^\vee}}}
\nc{\dualHecblock}{{\bf{\mathcal H^\vee_{\alpha,\beta}}}}
\newcommand{\ramBun}{{\bf{Bun}}}
\newcommand{\ramBuno}{\ramBun^{\circ}}

\nc{\Buntheta}{{\bf Bun}_{\theta}} \nc{\Bunthetao}{{\bf
Bun}_{\theta}^{\circ}} \nc{\BunGR}{{\bf Bun}_{G_\R}}
\nc{\BunGRo}{{\bf Bun}_{G_\R}^{\circ}}
\nc{\HC}{{\mathcal{HC}}}
\nc{\risom}{\stackrel{\sim}{\to}} \nc{\Hv}{{H^\vee}}
\nc{\bS}{{\mathbf S}}

\def\Conn{\operatorname {Conn}}

\nc{\Vect}{{\operatorname{Vect}}}
\nc{\Hecke}{{\operatorname{Hecke}}}

\newcommand{\ZZ}{{Z_{\bullet}}}
\nc{\HZ}{{\mc H}\ZZ} \nc{\eps}{\epsilon}

\nc{\CN}{\mathcal N} \nc{\BA}{\mathbb A}
\nc{\XYX}{X\times_Y X}

\nc{\bQ}{\mathbb{Q}}

\nc{\ul}{\underline}

\nc{\bn}{\mathbf n} \nc{\Sets}{{\on{Sets}}} \nc{\Top}{{\on{Top}}}

\nc{\Simp}{{\mathbf \Delta}} \nc{\Simpop}{{\mathbf\Delta^\circ}}

\nc{\Cyc}{{\mathbf \Lambda}} \nc{\Cycop}{{\mathbf\Lambda^\circ}}

\nc{\Mon}{{\mathbf \Lambda^{mon}}}
\nc{\Monop}{{(\mathbf\Lambda^{mon})\circ}}

\nc{\Aff}{{\on{Aff}}} \nc{\Sch}{{\on{Sch}}}

\nc{\bul}{\bullet}
\nc{\module}{{\operatorname{-mod}}}

\nc{\dstack}{{\mathcal D}}

\nc{\BL}{{\mathbb L}}

\nc{\BD}{{\mathbb D}}

\nc{\BR}{{\mathbb R}}

\nc{\BT}{{\mathbb T}}

\nc{\bT}{\mathbb{T}}

\nc{\SCA}{{\mc{SCA}}}
\nc{\DGA}{{\mc DGA}}

\nc{\DSt}{{DSt}}

\nc{\lotimes}{{\otimes}^{\mathbf L}}

\nc{\bs}{\backslash}

\nc{\Lhat}{\widehat{\mc L}}

\newcommand{\Coh}{{\on{Coh}}}

\newcommand{\TT}{\mathbb{T}^{*[\mh 1]}}

\nc{\QC}{\operatorname{QC}}
\nc\Perf{\on{Perf}}
\nc{\Cat}{{\on{Cat}}}
\nc{\dgCat}{{\on{dgCat}}}
\nc{\bLa}{{\mathbf \Lambda}}
\nc{\QCoh}{\QC}
\newcommand{\IndCoh}{\QC^!}

\nc{\RHom}{\mathbf{R}\hspace{-0.15em}\on{Hom}}
\nc{\REnd}{\mathbf{R}\hspace{-0.15em}\on{End}}
\nc{\oo}{\infty}
\nc\Mod{\on{Mod}}

\nc\fh{\mathfrak h}
\nc\al{\alpha}
\nc\la{\alpha}
\nc\BGB{B\bs G/B}
\nc\QCb{QC^\flat}
\nc\qc{\cQ}

\nc{\fg}{\mathfrak g}

\nc{\fn}{\mathfrak n}
\nc{\Map}{\on{Map}} \nc{\fX}{\mathfrak X}

\nc{\Tate}{\operatorname{Tate}}

\nc{\ch}{\check}
\nc{\fb}{\mathfrak b} \nc{\fu}{\mathfrak u} \nc{\st}{{st}}
\nc{\fU}{\mathfrak U}
\nc{\fZ}{\mathfrak Z}

\nc\fk{\mathfrak k} \nc\fp{\mathfrak p}

\nc{\RP}{\mathbf{RP}} \nc{\rigid}{\text{rigid}}
\nc{\glob}{\text{glob}}

\nc{\cI}{\mathcal I}

\nc{\La}{\mathcal L}

\nc{\quot}{/\hspace{-.25em}/}

\nc\aff{\it{aff}}
\nc\BS{\mathbb S}

\nc\Loc{{\mc Loc}}
\nc\Ch{{\mc Ch}}

\nc\git{/\hspace{-0.2em}/}
\nc{\fc}{\mathfrak c}
\nc\BC{\mathbb C}
\nc\BZ{\mathbb Z}
\nc\bZ{\mathbb Z}

\nc\stab{\text{\it st}}
\nc\Stab{\text{\it St}}

\nc\perf{\on{-perf}}

\nc\intHom{\mathcal{H}om}
\nc\intEnd{\mathcal{E}nd}

\nc\gtil{\widetilde\fg}

\newcommand{\shear}{{\mathbin{\mkern-6mu\fatslash}}}
\newcommand{\unshear}{{\mathbin{\mkern-6mu\fatbslash}}}

\def\adjquot{/_{\hspace{-0.2em}ad}\hspace{0.1em}}

\nc\mon{\text{\it mon}}
\nc\bimon{\text{\it bimon}}
\nc\uG{{\underline{G}}}
\nc\uB{{\underline{B}}}
\nc\uN{{\underline{\cN}}}
\nc\uNtil{{\underline{\wt{\cN}}}}
\nc\ugtil{{\underline{\wt{\fg}}}}
\nc\uH{{\underline{H}}}
\nc\uX{{\ul{X}}}
\nc\uY{\ul{Y}}
\nc\upi{\ul{\pi}}
\nc{\uZ}{\ul{Z}}
\nc{\ucZ}{\ul{\cZ}}
\nc{\ucH}{\ul{\cH}}
\nc{\ucS}{\ul{\cS}}
\nc{\ahat}{{\wh{a}}}
\nc{\shat}{{\wh{s}}}
\nc{\DCoh}{{\rm DCoh}}

\newcommand{\Lie}{\operatorname{Lie}}

\newcommand{\cind}{\operatorname{cInd}}
\newcommand{\LL}{\cL}

\newcommand{\Tot}{\operatorname{Tot}}

\newcommand{\mf}{\mathfrak}

\newcommand{\mantodo}[1]{\textbf{\textcolor{red}{todo: #1}}}
\newcommand{\help}[1]{\todo[color=green]{HELP: #1}}
\newcommand{\comm}[1]{\todo[color=green]{HC: #1}}
\newcommand{\chen}[1]{\todo[color=green]{HC: #1}}
\newcommand{\dbz}[1]{\todo[color=orange]{DBZ: #1}}
\newcommand{\helm}[1]{\todo[color=yellow]{DH: #1}}
\newcommand{\nadler}[1]{\todo[color=blue]{DN: #1}}

\newcommand{\gwt}{\operatorname{wt}_{\mathbb{G}_m}}

\newcommand{\cat}{\mathbf}
\newcommand{\OO}{\cO}
\newcommand{\dmod}{\operatorname{-mod}}
\newcommand{\DMOD}{\textbf{-mod}}
\newcommand{\dperf}{\operatorname{-perf}}
\newcommand{\dcmod}{\operatorname{-cmod}}
\newcommand{\dtors}{\operatorname{-tors}}

\newcommand{\dcomod}{\operatorname{-comod}}
\newcommand{\dcoh}{\operatorname{-coh}}
\newcommand{\ICoh}{\operatorname{ICoh}}
\newcommand{\dR}{\operatorname{dR}}
\newcommand{\LLf}{\widehat{\cL}}
\newcommand{\GG}{G_{\mathrm{gr}}}
\newcommand{\ggr}{{\mathrm{gr}}}
\newcommand{\Hgr}{{H_{\mathrm{gr}}}}
\newcommand{\Bgr}{{B_{\mathrm{gr}}}}

\newcommand{\Fun}{\operatorname{Fun}}

\newcommand{\NN}{\widetilde{\mathcal{N}}}

\newcommand{\actson}{\circlearrowright}

\newcommand{\Haff}{\mathcal{H}}
\newcommand{\grH}{\bar{\mathcal{H}}}

\newcommand{\Hcat}{\mathbf{H}}

\newcommand{\ubQ}{\underline{\mathbb{Q}}}

\newcommand{\Rep}{\operatorname{Rep}}

\nc\Tr{{\mathbf{Tr}}} 
\newcommand{\tr}{\mathcal{T}r} 
\newcommand{\chern}{\operatorname{ch}}

\newcommand{\tilN}{\widetilde{\mathcal{N}}}

\newcommand{\un}{{un}}

\newcommand{\TS}{\mathbf{T}}

\newcommand{\DLnil}{\mathcal{V}}

\newcommand{\bQl}{\overline{\mathbb{Q}}_\ell}

\newcommand{\tens}[1]{%
  \mathbin{\mathop{\otimes}\limits_{#1}}%
}

\newcommand{\utimes}[1]{%
  \mathbin{\mathop{\times}\limits_{#1}}%
}

\mathchardef\md="2D
\mathchardef\mh="2D

\newcommand{\BG}{\cC}

\newcommand{\NH}{\mathbf{H}}

\newcommand{\bfI}{{\widecheck{\mathbf{I}}}}
\newcommand{\bfG}{{\widecheck{\mathbf{G}}}}

\newcommand{\LS}{\mathcal{LS}}
\newcommand{\APerf}{\operatorname{APerf}}

\newcommand{\bG}{\mathbb{G}}


\title[Coherent Springer Theory and Categorical Deligne-Langlands]{Coherent Springer Theory and the Categorical Deligne-Langlands Correspondence}
\author{David Ben-Zvi} \address{Department of Mathematics\\University
  of Texas\\Austin, TX 78712-0257} \email{benzvi@math.utexas.edu}
\author{Harrison Chen} \address{Institute of Mathematics\\Academia Sinica\\Taipei 106319, Taiwan} \email{chenhi@gate.sinica.edu.tw}
  
  \author{David Helm} \address{Department of Mathematics\\Imperial College\\  London SW7 2BU, United Kingdom} \email{d.helm@imperial.ac.uk}
\author{David Nadler} \address{Department of Mathematics\\University
  of California\\Berkeley, CA 94720-3840}
\email{nadler@math.berkeley.edu}

\begin{abstract}
Kazhdan and Lusztig identified the affine Hecke algebra $\mathcal{H}$ with an equivariant $K$-group of the Steinberg variety, and applied this to prove the Deligne-Langlands conjecture, i.e., the local Langlands parametrization of irreducible representations of reductive groups over nonarchimedean local fields $F$ with an Iwahori-fixed vector. We apply techniques from derived algebraic geometry to pass from $K$-theory to Hochschild homology and thereby identify $\mathcal{H}$ with the endomorphisms of a coherent sheaf on the stack of unipotent Langlands parameters, the \emph{coherent Springer sheaf}.  As a result the derived category of $\mathcal{H}$-modules is realized as a full subcategory of coherent sheaves on this stack, confirming expectations from strong forms of the local Langlands correspondence (including recent conjectures of Fargues-Scholze, Hellmann and Zhu).  

In the case of the general linear group our result allows us to lift the local Langlands classification of irreducible representations to a categorical statement: we construct a full embedding of the derived category of smooth representations of $\mathrm{GL}_n(F)$ into coherent sheaves on the stack of Langlands parameters. 
\end{abstract}

\maketitle


\tableofcontents

\section{Introduction}\label{introduction}
Our goals in this paper are to provide a spectral description of the category of representations of the affine Hecke algebra and deduce applications to the local Langlands correspondence. We begin with a quick review of Springer theory and then discuss our main results starting in Section~\ref{main results section}.

We will work in the setting of derived algebraic geometry over a field $k$ of characteristic zero, as presented in~\cite{GR}. In particular all operations, sheaves, categories etc will be derived unless otherwise noted.

\medskip

\subsection{Springer theory and Hecke algebras}\label{Springer}
We first review some key points of Springer theory, largely following the perspective of~\cite{CG,ginzburg intro}. Let $G$ denote a complex reductive group with Lie algebra $\fg$ and Borel $B\subset G$. We denote by $\cB\simeq G/B$ the flag variety, $\cN$ the nilpotent cone, $\mu:\tilN=T^{\ast} \cB\to \cN$ the Springer resolution, and $\cZ=\tilN\times_\cN \tilN$ the Steinberg variety. 

The Springer correspondence provides a geometric realization of representations of the Weyl group $W$ of $G$. The Weyl group is in bijection with the Bruhat double cosets $B\bs G/B=G\bs(\cB \times \cB)$, and hence with the conormals to the Schubert varieties, which form the irreducible components of the Steinberg variety $\cZ$. In fact the group algebra of the Weyl group can be identified with the top Borel-Moore homology of $\cZ$ under the convolution product
$$\C W\simeq H_d^{BM}(\cZ; \C),$$
where $d = \dim(\cN) = \dim(\wt{\cN}) = \dim(\cZ)$.  This realization of $W$ can be converted into a sheaf-theoretic statement.
The Springer sheaf $$\bS=\mu_* \C_{\tilN}[d]\in \Perv(\cN/G)$$ is the equivariant perverse sheaf on the nilpotent cone given by the pushforward of the (shifted) constant sheaf on the Springer resolution. Thanks to the definition of $\cZ$ as the self-fiber-product $\cZ=\tilN\times_\cN\tilN$, a simple base-change calculation provides an isomorphism 
$$H_{d}^{BM}(\cZ; \C)\simeq \End_{\cN/G}(\bS)$$
between the endomorphisms of $\bS$ and the top homology of $\cZ$, i.e., the group algebra $\C W$.
By Lusztig's generalized Springer correspondence~\cite[Theorem 6.5]{lusztig 84} the abelian category $\Perv(\cN/G)$ is semisimple, thus all objects are projective and we may interpret this isomorphism as a full embedding of the abelian category of representations of $W$ into equivariant perverse sheaves on the nilpotent cone,
$$\Rep(W)=\C W\module\simeq\langle \bS\rangle \subset \Perv(\cN/G).$$

One important role for this embedding is provided by the representation theory of Chevalley groups.
The universal unipotent principal series representation\footnote{Note that the finite Hecke algebra and hence the category of unipotent principal series representations is insensitive to Langlands duality. From our perspective it is in fact more natural to consider here representations of the Langlands dual Chevalley group $\Gv(\Ff_q)$.}
 $$\C G(\Ff_q)\actson \C[\cB(\Ff_q)]$$ has as endomorphism algebra the finite Hecke algebra
$$\Haff^f=\C[B(\Ff_q) \bs G(\Ff_q) / B(\Ff_q)] = \End_{G(\Ff_q)}(\C[G(\Ff_q)/B(\Ff_q)]),$$
which (after choosing a square root of $q$) may be identified with $\C W$. Thus Springer theory provides a full embedding
$$\{\mbox{unipotent principal series of $G(\Ff_q)$}\}\simeq \Haff^f\module\stackrel{\sim}{\longrightarrow} \langle \bS\rangle \subset \Perv(\cN/G)$$
where we say a representation of $G(\Ff_q)$ is in the unipotent principal series if it is generated by its $B(\Ff_q)$-invariants.
\medskip

\subsection{Affine Hecke algebras}

We now let $G$ be a reductive group, 
Langlands dual to a split group $\Gv(F)$ over a nonarchimedean local field $F$ with ring of integers $O$ and residue field $\Ff_q$.  We write $\GG=G\times \Gm$ as shorthand, which acts on $\cZ$ by $(g, z) \cdot (x, B, B') = (z^{-1}gxg^{-1}, gB, gB')$.

\begin{defn}
Let $G$ be a reductive group with maximal torus $T$.  The (extended) \emph{affine Weyl group} of the dual group $G^\vee$ is the semidirect product $W_a = W \ltimes X_\bullet(T^\vee) = W \ltimes X^\bullet(T)$ of the finite Weyl group with the cocharacter lattice of $T^\vee$.  The \emph{affine Hecke algebra} $\Haff$ is a certain $q$-deformation of the group ring $\C {W}_a$ such that specializing $q$ at a prime power gives the Iwahori-Hecke algebra:
$$\Haff_q = \C_c[I\bs G^\vee(F)/I] = \End_{\Rep(G^\vee(F))}(\C_c[G^\vee(F)/I])$$
where $I \subset G^\vee(F)$ is an Iwahori subgroup.  Explicit presentations of the affine Hecke algebra can be found, for example, in Section 7.1 of \cite{CG}.  Unlike the finite Hecke algebra, $\cH_q \not\simeq \C W_a$.
\end{defn}

Our starting point is the celebrated theorem of Kazhdan-Lusztig~\cite{KL} (as later extended and modified by Ginzburg, see~\cite{CG} and Lusztig~\cite{lusztig bases}), providing a geometric realization of the affine Hecke algebra in terms of the Steinberg variety.  

\medskip 

\begin{thm}~\cite{KL,CG,lusztig bases}\label{KL theorem}
Suppose that $G$ has simply connected derived subgroup.  There is an isomorphism of algebras $\Haff\simeq K_0(\cZ/\GG) \otimes_{\Z} \C$, compatible with the Bernstein isomorphism $Z(\Haff)\simeq \C[\GG]^{\GG} \simeq K_0^{\GG}(\pt) \otimes_{\Z} \C$ between the center of $\Haff$ and the ring of equivariant parameters. 
\end{thm}

\medskip

Kazhdan and Lusztig famously applied Theorem~\ref{KL theorem} to prove the Deligne-Langlands conjecture, as refined by Lusztig. The category of representations of $\Haff_q$ is identified with the ``Iwahori block", the (smooth) representations of $\Gv(F)$ that are generated by their $I$-invariants (i.e., ``appear in the decomposition of $C^\infty_c(\Gv(F)/I; \bQl)$"). Equivalently this is the unramified principal series, the representations of $\Gv(F)$ appearing in the parabolic induction of unramified characters of a split torus (i.e., ``appear in the decomposition of $C^\infty(\Gv(F)/N^\vee(F)T^\vee(O); \bQl)$"). 
The Deligne-Langlands conjecture provides a classification of irreducible representations in the Iwahori block (i.e. with an Iwahori fixed vector), or equivalently irreducible $\Haff_q$ modules, in terms of Langlands parameters:

\medskip

\begin{thm}~\cite{KL, reeder}\label{DL conjecture}
The irreducible representations of $\Haff_q$ are in bijection with $G$-conjugacy classes of $q$-commuting pairs of semisimple and nilpotent elements in $G$
 $$\{s\in G^{ss}, n\in \cN: gng^{-1} =qn\}/G,$$ 
 together with a $G$-equivariant local system on the orbit of $(s,n)$ which appears in the decomposition of a corresponding Springer sheaf. 
\end{thm}

\medskip

For fixed $(s, q)$ the variety $\cN^{(s,q)}$ of $(s,q)$-fixed points on the nilpotent cone can be interpreted as a variety of {\em Langlands parameters}. Representations with a fixed Langlands parameter $(s, n)$ form an {\em L-packet}, and are described in terms of irreducible representations of the component group of the stabilizer. These representations can then be interpreted as equivariant local systems on the orbit of the Langlands parameter. 
Indeed general conjectures going back to work of Lusztig~\cite{lusztig examples}, Zelevinsky~\cite{zelevinsky} and Vogan~\cite{vogan} describe the representation theory of $\Gv(F)$ at a fixed central character with the geometry of equivariant perverse sheaves on suitable spaces of Langlands parameters, generalizing the appearance of $\cN^{(s,q)}$ above. 

\medskip

However, unlike the classical Springer theory story for $\Haff^f_q \simeq \C W$, the realization of $\Haff$ by equivariant $K$-theory in Theorem~\ref{KL theorem} does not immediately lead to a realization of $\Haff$ as endomorphisms of a sheaf, and therefore to a sheaf-theoretic description of the entire category of $\Haff$-modules. Rather, in applications equivariant $K$-theory is used as an intermediate step on the way to equivariant Borel-Moore homology, which leads back to variants of the Springer correspondence. Namely, by fixing a central character for $\Haff$, i.e. a Weyl group orbit of
$(s,q)\in T \times \Gm$, the central completions of equivariant $K$-theory are identified by Lusztig~\cite{cuspidal1,gradedAHA} with graded Hecke algebras, which have a geometric description where we replace the nilpotent cone $\cN$, Springer resolution $\wt{\cN}$ and Steinberg variety $\cZ$ by their $(s,q)$-fixed points.  For example, the Chern character identifies the completion of $\Haff$ at the trivial central character with the $\GG = G \times \G_m$-equivariant homology of the Steinberg variety $\cZ$. This algebra is identified via Theorem 8.11 of \cite{cuspidal2} with the full Ext-algebra of the Springer sheaf in the equivariant derived category
$$\cH^{gr}\simeq H_\bullet^{BM}(\cZ/\GG; \C)=R\Gamma(\cZ/\GG,\omega_{\cZ/\GG})\simeq \Ext^\bullet_{\cN/\GG}(\bS).$$ Moreover, by a theorem of Rider~\cite{rider} this Ext algebra is formal, hence we obtain a full embedding 
\begin{equation}\label{graded Hecke}
\cH^{gr}\module \simeq \langle \bS \rangle \subset \Shv(\cN/\GG)
\end{equation} of representations of $\cH^{gr}$ into the equivariant derived category of the nilpotent cone.
More generally, for $(s,q)\in T \times \Gm$,  we have an identification 
$$\cH^{gr}_{(s,q)}\simeq H_\bullet^{BM}(\cZ^{(s,q)}/\GG^{(s,q)}; \C)\simeq \Ext^\bullet_{\cN^{(s,q)}/\GG^{(s,q)}}(\bS^{(s,q)})$$
of the corresponding graded Hecke algebra in terms of an $(s,q)$-variant of the Springer sheaf. This provides a geometric approach to constructing and studying modules\footnote{Further if one had an $(s,q)$-version of Rider's formality theorem, one could deduce a full embedding of the corresponding module categories into equivariant derived categories of constructible sheaves on $\cN^{(s,q)}$.  See Theorem 3.1 of \cite{kato} for an accounting.} of $\Haff$, see~\cite{CG}.

These developments give satisfying descriptions of the representation theory of $\Haff$ at a fixed central character. However there are numerous motivations to seek a description of {\em families} of representations of varying central character, including classical harmonic analysis (for example in the setting of spherical varieties~\cite{SV}), $K$-theory and the Baum-Connes conjecture~\cite{ABPS}, and modular and integral representation theory~\cite{emerton helm, curtis, HM converse}.

\medskip
\subsection{Coherent Springer Theory}\label{main results section} 

In this paper we apply ideas from derived algebraic geometry to deduce from Theorem~\ref{KL theorem} a different, and in some sense simpler, geometric realization of the affine Hecke algebra, in which we first replace $K$-theory by Hochschild homology, and then derive a description of its entire category of representations as a category of coherent sheaves (without the need for specifying central characters).  For technical reasons, we will need to replace the nilpotent cone $\cN$ with its formal completion $\wh{\cN} \subset \mf{g}$, and likewise the Steinberg variety $\cZ = \wt{\cN} \times_{\mf{g}} \wt{\cN}$ will be defined via a derived fiber product.  For precise definitions of objects in this context, see Section \ref{rep thy defs}.

\medskip
 
\begin{thm}[Theorem \ref{ThmHH}, Corollary \ref{thm no sc}]\label{first main thm intro}
Let $k = \bQl$ or $\C$, and $G$ a reductive algebraic group over $k$.  The trace map from connective $K$-theory to Hochschild homology on $\Coh(\cZ/\GG)$ factors through an isomorphism of $K_0$ and $HH_\bullet$ (which is concentrated in cohomological degree zero):
$$\xymatrix{ K_\bullet(\Coh(\cZ/\GG)) \otimes_{\Z} k \ar[r] \ar[d] & HH_\bullet(\Coh(\cZ/\GG))\ar[d]^-\simeq\\
K_0(\Coh(\cZ/\GG)) \otimes_{\Z} k \ar[r]^-\simeq & HH_0(\Coh(\cZ/\GG)).}$$ 
\end{thm}

\begin{rmk}
Our results also allow for an identification of monodromic variants of the affine Hecke category.  See Remark \ref{hecke variants} for details.
\end{rmk}

\medskip

The Hochschild homology of categories of coherent sheaves admits a description in the derived algebraic geometry of loop spaces.
In particular, we deduce an isomorphism of the affine Hecke algebra with volume forms on the derived loop space to the Steinberg stack,
$$\Haff\simeq R\Gamma(\cL(\cZ/\GG),\omega_{\LL(\cZ/\GG)}).$$ More significantly,  the geometry of derived loop spaces provides a natural home for the entire category of $\Haff$-modules, without fixing central characters. 
\begin{defn}\label{intro springer sheaf} 
Let $\wh{\cN} \subset \mf{g}$ be the formal completion\footnote{Note that for any formal completion $\wh{Z}$ along a closed substack $Z \subset X$, following \cite{GR} we define the category $\Coh(\wh{Z})$ so that it is canonically equivalent to the category $\Coh_{Z}(X)$ of coherent sheaves on the ambient stack set-theoretically supported at $Z$.  Thus the reader unfamiliar with formal completions may replace $\wh{\cN}$ with $\mf{g}$, and impose nilpotent support conditions on all categories of sheaves.} of the nilpotent cone, $\wt{\cN}$ the usual (reduced) Springer resolution and $\mu: \wt{\cN} \rightarrow \cN \hookrightarrow \wh{\cN}$ the composition of the Springer resolution with the inclusion. The \emph{coherent Springer sheaf} $\cS_G \in \Coh(\cL(\wh{\cN}/\GG))$ (or simply $\cS$) is the pushforward of the structure sheaf under the loop map $\cL \mu:\cL(\tilN/\GG)\to \cL(\wh{\cN}/\GG)$: 
$$\cS_G =\cL \mu_*\cO_{\cL(\tilN/\GG)} \in \Coh(\cL(\wh{\cN}/\GG)).$$
Equivalently, $\cS_G$ is given by applying the parabolic induction correspondence
$$\xymatrix{
\LL(\wh{\{0\}}/T)  & \cL(\wh{\mf{n}}/B) = \cL(\wh{\wt{\cN}}/G)\ar[l]\ar[r] & \cL(\wh{\cN}/G)}$$
to the (reduced) structure sheaf of $\cL({\{0\}}/T)$.
\end{defn}

A priori the coherent Springer sheaf is only a complex of sheaves. However we show, using the theory of traces for monoidal categories in higher algebra, that its Ext algebra is concentrated in degree zero, and is identified with the affine Hecke algebra. This provides the following ``coherent Springer correspondence'', realizing the representations of the affine Hecke algebra as coherent sheaves.
\begin{thm}[Theorem \ref{thm endomorph}]\label{second main thm intro}
Let $G$ be a reductive algebraic group over $k = \bQl$ or $\bC$.
\begin{enumerate}
\item There is an isomorphism of algebras $\Haff_G\simeq \End_{\cL(\wh{\cN}/\GG)}(\cS_G)$
and all other self-$\Ext$ groups of $\cS_G$ vanish.
\item There is an embedding of dg derived categories
$$\begin{tikzcd}
D(\Haff_G) \arrow[rr, "\simeq"', "{- \otimes_{\End(\cS)} \cS_G}"] & &  \langle \cS_G \rangle \arrow[r, hook] &  \IndCoh(\cL(\wh{\cN}/\GG)).\end{tikzcd}$$
\item The embedding takes the anti-spherical module to the projection of the dualizing sheaf to the Springer subcategory
$$D(\Haff_G) \ni \Ind_{\cH^f}^{\cH}(\mathrm{sgn}) \longmapsto \mathrm{pr}_{\cS_G}(\omega_{\cL(\wh{\cN}/\GG)}) \in \IndCoh(\cL(\wh{\cN}/\GG)).$$
\item The embedding is compatible with parabolic induction of affine Hecke algebras, i.e. if $P$ is a parabolic subgroup of $G$ with Levi quotient $M$, then there is a commuting diagram
$$\begin{tikzcd}
D(\Haff_{M}) \arrow[d, "{\cH_G \otimes_{\cH_M} -}"']\arrow[r, hook] & \IndCoh(\cL(\wh{\cN}_M/\wt{M}))\arrow[d, "\cL\mu_* \circ \cL\nu^*"]\\
D(\Haff_{G}) \arrow[r, hook] &  \IndCoh(\cL(\wh{\cN}_G/\GG)),
  \end{tikzcd}$$
where $\cL\mu_* \circ \cL\nu^*$ is the pull-push along the correspondence obtained by applying $\cL$ to the usual parabolic induction correspondence
$$\begin{tikzcd}
\cL(\wh{\cN}_M/\wt{M}) & \arrow[l, "\cL\mu"'] \cL(\wh{\cN}_P/\wt{P}) \arrow[r, "\cL\nu"] & \cL(\wh{\cN}_G/\GG).
\end{tikzcd}$$
In particular, $\cL\mu_* \cL\nu^* \cS_M \simeq \cS_G$.
\end{enumerate}
\end{thm}
One consequence of the theorem is an interpretation of the coherent Springer sheaf as a universal family of $\Haff$-modules.

\medskip

We also conjecture (Conjecture \ref{springer heart}) -- and check for $SL_2$ -- that $\cS$ is actually a coherent sheaf (i.e., lives in the heart of the standard t-structure on coherent sheaves). The vanishing of all nonzero Ext groups of $\cS$ suggests the existence of a natural ``exotic" t-structure for which $\cS$ is a compact projective object in the heart. For such a t-structure we would then automatically obtain a full embedding of the abelian category $\Haff\module$ into ``exotic" coherent sheaves, where one could expect a geometric description of simple objects. 

\medskip

In \cite{ihes} we will explain how equivariant localization and Koszul duality patterns in derived algebraic geometry (as developed in~\cite{loops and reps, Ch, koszul}) provide the precise compatibility between this coherent Springer theory and the usual perverse Springer theory, one parameter at a time.

\medskip

\subsection{Applications to the local Langlands correspondence}\label{local Langlands applications}

We will consider a derived stack $\BL_{q,G}^u$ of {\em unipotent Langlands parameters}, which parametrizes the unipotent Weil-Deligne representations for a local field $F$ with residue field $\mathbb{F}_q$, and whose set of $k$-points is a variant of the set of Deligne-Langlands parameters in Theorem~\ref{DL conjecture} (with semisimplicity of $s$ dropped).  Note that the following notions make sense for any $q \in \mathbb{C}$, with applications to local Langlands when $q$ is a prime power, and that, in line with expectations, the stack of \emph{unipotent} Langlands parameters depends only on the order of the residue field of $F$.

\medskip

\begin{defn}\label{q-Springer} \;
Let $q = p^r$ be a prime power.
\begin{enumerate}
\item The \emph{stack of unipotent Langlands parameters} $\BL_{q,G}^u= \cL_q(\wh{\cN}/G)$ (or simply $\BL_q^u$) is the derived fixed point stack of multiplication by $q\in \Gm$ on $\wh{\cN}/G$.  Equivalently, it is the fiber of the loop (or derived inertia) stack of the nilpotent cone over $q\in \Gm$,
$$\xymatrix{ \mathbb{L}_{q,G}^u \ar[r]\ar[d] & \cL(\wh{\cN}/\GG)\ar[d]\\
\{q\} \ar[r] & \cL(\pt/\Gm)=\Gm/\Gm.}$$ 
By Proposition \ref{no derived structure}, the derived inf-stack $\BL_{q,G}^u$  has no derived nor infinitesimal structure, i.e. $\cL_q(\wh{\cN}/G) = \cL_q(\mf{g}/G)$, and by \cite{DHKM} it is reduced, so we may equivalently define $\BL_{q,G}^u$ using the classical fiber product of the reduced nilpotent cone $\cN$, i.e.
$$\BL_{q,G}^u \simeq \{g\in G, n\in \cN: gng^{-1}=qn\}/G.$$
\item The \emph{$q$-coherent Springer sheaf} $\cS_{q,G}\in \Coh(\BL_q^u)$ (or simply $\cS_q$) is the $*$-specialization of $\cS_G$ to the fiber $\BL_q^u$ over $q$. 
Equivalently, $\cS_{q,G}$ is given by applying the parabolic induction correspondence
$$\xymatrix{ \BL_{q,T}^u &\BL_{q,B}^u \ar[l]\ar[r] & \BL_{q,G}^u }$$
to the structure sheaf of $\BL_{q,T}^u \simeq T \times BT$. 
\end{enumerate}
\end{defn}


\medskip

Specializing Theorem~\ref{second main thm intro} to $q\in \Gm$ we obtain the following.  Note that Theorem 2.2, Proposition 2.4 and Corollary 2.5 of \cite{OS} apply in the case where $q$ is specialized away from roots of unity; in particular, $\cH_{q,G}$ has finite cohomological dimension if $q$ is not a root of unity.  Thus in the following statement we implicitly identify the compact objects $D_{perf}(\Haff_G) \subset D(\Haff_G)$ (i.e. the subcategory of perfect complexes) with the bounded derived category of coherent complexes.
\begin{thm}[Theorem \ref{thm endomorph}]\label{third main thm intro} 
Suppose that $q = p^r$ is a prime power (or more generally, $q \in \G_m$ is not a root of unity), and let $G$ be a reductive algebraic group over $k = \bQl$ or $\bC$.
\begin{enumerate}
\item There an isomorphism of algebras $\cH_{q,G} \simeq \End_{\BL_{q,G}}(\cS_{q,G})$ and a full embedding 
$$\begin{tikzcd}
D_{perf}(\Haff_{q,G}) = D_{coh}(\Haff_{q,G}) \arrow[rr, "\simeq"', "{- \otimes_{\End(\cS)} \cS_{q,G}}"] & & \langle \cS_{q,G} \rangle \arrow[r, hook] & \Coh(\BL_{q,G}^u).\end{tikzcd}$$
In particular, this gives a full embedding of the principal block of $\Gv(F)$ into coherent sheaves on the stack of unipotent Langlands parameters.
\item The embedding takes the anti-spherical module to the structure sheaf $\cO_{\BL_{q,G}^u} \in \Coh(\BL_{q,G}^u)$.
\item The embedding is compatible with parabolic induction, i.e. if $P^{\vee} \subset G^\vee$ is a parabolic with quotient Levi $M^{\vee}$, then we have a commutative diagram
$$\begin{tikzcd}
\{\mbox{unramified principal series of  $M^{\vee}(F)$}\} \simeq D_{coh}(\Haff_{q,M}) \arrow[d, "{i_{P^\vee}^{G^\vee}}"']\arrow[r, hook] & \Coh(\BL_{q,M}^u)\arrow[d, "(\mu^q)_* \circ (\nu^q)^*"]\\
  \{\mbox{unramified principal series of $\Gv(F)$}\}\simeq D_{coh}(\Haff_{q,G}) \arrow[r, hook] &  \Coh(\BL_{q,G}^u),
  \end{tikzcd}$$
 where $i_{P^\vee}^{G^\vee}: \Rep^{sm}_{f.g.}(M^\vee(F)) \rightarrow \Rep^{sm}_{f.g.}(G^\vee(F))$ is the parabolic induction functor from smooth finitely-generated\footnote{I.e. the corresponding modules for Hecke algebras are finitely generated.} reprentations of $M^\vee(F)$ to $G^\vee(F)$ restricted to the unramified principal series, and the map $(\mu^q)_* \circ (\nu^q)^*$ is the pull-push along  the correspondence obtained by applying taking derived $q$-invariants of the usual parabolic induction correspondence
$$\begin{tikzcd}
\BL_{q,M}^u & \arrow[l, "\mu^q"'] \BL_{q,P}^u \arrow[r, "\nu^q"] & \BL_{q,G}^u.
\end{tikzcd}$$
In particular, $(\mu^q)_* (\nu^q)^* \cS_{q,M} \simeq \cS_{q,G}$.
\end{enumerate}
\end{thm}
Note that due to Proposition \ref{no derived structure}, in the $q$-specialized setting of the above theorem the stack of parameters has no infinitesimal structure, i.e. $\cL_q(\mf{g}/G) = \cL_q(\wh{\cN}/G)$.  This has two consequences: first, due to Proposition \ref{calabi yau}, which does not apply in the context of Theorem~\ref{second main thm intro}, we may identify the anti-spherical sheaf at specialized $q$ with the structure sheaf, which is equivalent to the dualizing sheaf.  Second, the anti-spherical sheaf at specialized $q$ is a compact object in the category, i.e. a coherent sheaf, whereas the sheaf appearing in Theorem~\ref{second main thm intro} is not.

\medskip

The existence of such an equivalence was conjectured independently by Hellmann in~\cite{hellmann}, whose work we learned of at a late stage in the preparation of his paper.  Indeed, the above result resolves Conjecture 3.2 of ~\cite{hellmann}.  
Hellmann's work also gives an alternative characterization
of the ($q$-specialized) coherent Springer sheaf as the Iwahori invariants of a certain family of admissible representations on $\BL_{q,G}^u$ constructed by Emerton and the third author in \cite{emerton helm}.

A much more general categorical form of the local Langlands correspondence is formulated by Fargues-Scholze~\cite{scholze lecture} and Zhu~\cite{xinwen survey}, as well as compatibility with a categorical global Langlands correspondence. In {\em loc. cit.} a forthcoming proof by Hemo and Zhu~\cite{hemo zhu} of a result closely parallel to ours is also announced.

\begin{rmk}
The local Langlands correspondence depends on a choice of Whittaker normalization; that is, a choice of a pair $(U,\psi)$, where $U$ is the unipotent radical of a Borel subgroup of $\Gv$ and $\psi$ is a generic character
of $U(F)$, up to $\Gv(F)$-conjugacy, and indeed, the conjecture in \cite{hellmann} and the announced result in \cite{hemo zhu} depend on such a choice.  In the formulation of Theorem~\ref{third main thm intro} no such choice
appears explicitly, but instead comes from the integral structure on $G^{\vee}$, which in particular gives us a distinguished hyperspecial subgroup $G^{\vee}(O)$ of $G^{\vee}(F)$.

Indeed, for any unramified group $G^{\vee}$ over $F$
there is a natural bijection between $G^{\vee}(F)$-conjugacy classes of Whittaker data $(U,\psi)$ for $G^{\vee}$ and $G^{\vee}(F)$-conjugacy classes of triples $(K_x, U_x, \psi_x)$, where $K_x$ is a hyperspecial subgroup
of $G^{\vee}(F)$, $U_x$ is the unipotent radical of a Borel subgroup of the reductive quotient $G^{\vee}_x$ of $K_x$, and $\psi_x$ is a generic character of $U_x$.  This bijection has the property that if $(U,\psi)$
corresponds to $(K_x,U_x,\psi_x)$, then the summand of the compact induction $\cind_{U(F)}^{\Gv(F)} \psi$ corresponding to the unipotent principal	 series block is isomorphic to 
$\cind_{K_x}^{\Gv(F)} \operatorname{St}_x$, where $\operatorname{St}_x$ denotes the inflation to $K_x$ of the Steinberg representation of the reductive quotient $G^{\vee}_x$.  In particular the  ``unipotent principal series part'' of
$\cind_{U(F)}^{\Gv(F)} \psi$ depends only on the conjugacy class of hyperspecial subgroup associated to $(U,\psi)$, and not the whole tuple $(K_x, U_x, \psi_x)$.  This means that the
restriction of the local Langlands correspondence to the unramified principal series depends only on a choice of hyperspecial subgroup (which we have fixed).

Note in particular that for any choice of Whittaker datum $(U,\psi)$ compatible with our hyperspecial subgroup $\Gv(O)$, the $\Haff_{q,G}$-module associated to the compact induction $\cind_{U(F)}^{\Gv(F)} \psi$
is precisely the antispherical module, so property (2) of Theorem~\ref{third main thm intro} is consistent with (and indeed, equivalent to) the Whittaker normalization appearing in \cite{hellmann}.
\end{rmk}

\medskip

In the case of the general linear group and its Levi subgroups, one can go much further. Namely, in Section \ref{moduli} we combine the local Langlands classification of irreducible representations due to Harris-Taylor and Henniart with the Bushnell-Kutzko theory of types and the ensuing inductive reduction of all representations to the principal block. The result is a spectral description of the entire category of smooth $\GL_n(F)$ representations. To do so it is imperative to first have a suitable stack of Langlands parameters. These have been studied extensively in
mixed characteristic, for instance in \cite{curtis} in the case of $\GL_n$, or more recently in
\cite{bellovin-gee, booher-patrikis}, and  \cite{DHKM} for more general groups.  
Since in our present context we work over $\C$, the results we need are in general simpler than the
results of the above papers, and have not appeared explicitly in the literature in the form we need.

\medskip

\begin{thm}[\cite{curtis}]\label{Langlands stack} Let $F$ be a local field with residue field $\mathbb{F}_q$.  There is a classical Artin stack locally of finite type $\bL_{F,\GL_n}$, with the following properties:
\begin{enumerate}
\item The $k$-points of $\bL_{F,\GL_n}$ are identified with the groupoid of continuous $n$-dimensional representations of the Weil-Deligne group of $F$.
\item The formal deformation spaces of Weil-Deligne representations are identified with the formal completions of $\bL_{F,\GL_n}$.
\item The stack $\BL_{q,\GL_n}^u$ of unipotent Langlands parameters is a connected component of $\bL_{F,\GL_n}$.
\end{enumerate}
\end{thm}

\medskip

We then deduce a categorical local Langlands correspondence for $\GL_n$ and its Levi subgroups as follows:

\medskip

\begin{thm}[Theorems \ref{derived local Langlands}, \ref{Springer induction}, and \ref{thm:LL parabolic induction}]\label{fourth main thm intro}
For each Levi subgroup $M$ of $\GL_n(F)$, there is a full embedding
$$D(M)\hookrightarrow \IndCoh(\bL_{F,M})$$
of the derived category of smooth $M$-representations into ind-coherent sheaves on the stack of Langlands parameters, uniquely characterized by the following properties.
\begin{enumerate}
\item If $\pi$ is an irreducible cuspidal representation of $M$, then the image of $\pi$ under this embedding is the skyscraper sheaf supported at the Langlands parameter associated to $\pi$.
\item Let $M'$ be a Levi subgroup of $G$, and let $P$ be a parabolic subgroup of $M'$ with Levi subgroup $M$.  There is a commutative diagram of functors:
$$
\begin{tikzcd}
D(M) \ar[d, "i_M^{M'}"']\ar[r, hook] & \IndCoh(\bL_{F,M})\ar[d, "\mu_* \nu^*"]\\
D(M') \ar[r, hook] & \IndCoh(\bL_{F,M'})
\end{tikzcd}
$$
in which $i_M^{M'}$ is the parabolic induction functor and the right-hand map is obtained by applying the
correspondence 
$$\begin{tikzcd}
\bL_{F,M} & \arrow[l, "\mu"'] \arrow[r, "\nu"]  \bL_{F,P} &\bL_{F,M'}.
\end{tikzcd}$$
\end{enumerate}
\end{thm}

Note that the local Langlands correspondence for cuspidal representations of $\GL_n$ and its Levis, is an {\em input} to the above result.  We do not expect the functor to be an equivalence, see Remark \ref{DL not equivalence}.

As with Theorem~\ref{third main thm intro} our results here were independently conjectured by Hellmann (see in particular Conjecture 3.2 of~\cite{hellmann}) for more general groups $G$; these results also fit the general categorical form of the local Langlands correspondence formulated by Fargues-Scholze~\cite{scholze lecture} and Zhu~\cite{xinwen survey}.

\medskip

\subsubsection{Discussion: Categorical Langlands Correspondence}\label{local langlands section}
Theorems~\ref{third main thm intro} and~\ref{fourth main thm intro} match the expectation in the Langlands program that has emerged in the last couple of years for a strong form of the local Langlands correspondence, in which categories of representations of groups over local fields are identified with categories of coherent sheaves on stacks of Langlands parameters. Such a coherent formulation of the real local Langlands correspondence was discovered in~\cite{loops and reps}, while the current paper finds a closely analogous picture in the Deligne-Langlands setting. As this paper was being completed Xinwen Zhu shared the excellent overview~\cite{xinwen survey} on this topic and Laurent Fargues and Peter Scholze completed the manuscript~\cite{scholze lecture}, to which we refer the reader for more details. We only briefly mention three deep recent developments in this general spirit.

The first derives from the work of V. Lafforgue on the global Langlands correspondence over function fields~\cite{lafforgue,lafforgue ICM}. Lafforgue's construction in Drinfeld's interpretation (cf.~\cite[Section 6]{lafforgue zhu}, ~\cite[Remark 8.5]{lafforgue ICM} and~\cite{dennis shtuka}) predicts the existence of a universal quasicoherent sheaf $\fA_X$ on the stack of representations of $\pi_1(X)$ into $G$ corresponding to the cohomology of moduli spaces of shtukas. The theorem of Genestier-Lafforgue~\cite{genestier} implies that the category of smooth $\Gv(F)$ representations sheafifies over a stack of local Langlands parameters, and the local version $\fA$ of the Drinfeld-Lafforgue sheaf is expected~\cite{xinwen survey} to be a universal $\Gv(F)$-module over the stack of local Langlands parameters. In other words, the fibers $\fA_\sigma$ are built out of the $\Gv(F)$-representations in the L-packet labelled by $\sigma$. The expectation is that the coherent Springer sheaf, which by our results is naturally enriched in $\Haff_q$-modules, is identified with the Iwahori invariants of the local Lafforgue sheaf $\cS_q\simeq \fA^I.$

The second is the theory of categorical traces of Frobenius as developed in~\cite{dennis shtuka,xinwen trace,GKRV}. When applied to a suitably formulated local geometric Langlands correspondence, we obtain an expected equivalence between an automorphic and spectral category. The automorphic category is $\Shv(\Gv(F)/^{\Fr} \Gv(F))$, the category of Frobenius-twisted adjoint equivariant sheaves on $\Gv(F)$, with orbits given by the Kottwitz set $B(\Gv)$ of isomorphism classes of $\Gv$-isocrystals. The spectral category is expected to be a variant of a category $\IndCoh(\BL_{F,G})$ of ind-coherent sheaves over the stack $\BL_{F,G}$ of Langlands parameters into $G$. The former category contains the categories of representations of $\Gv(F)$ and its inner forms as full subcategories, hence we expect a spectral realization in the spirit of Theorems~\ref{third main thm intro} and \ref{fourth main thm intro}.

The last of these developments is the program of Fargues-Scholze~\cite{fargues}, \cite{scholze lecture} in the context of $p$-adic groups, which interprets the local Langlands correspondence as a geometric Langlands correspondence. On the automorphic side one considers  sheaves on the stack $\Bun_{\Gv}$ of bundles on the Fargues-Fontaine curve, whose isomorphism classes $|\Bun_{\Gv}|=B(\Gv)$ are given as before by the Kottwitz set of $\Gv$-isocrystals. This category of sheaves admits a semiorthogonal decomposition indexed by $B(\Gv)$, in which the factor corresponding to $b \in B(\Gv)$ is naturally equivalent to the category of smooth representations of the inner form $\Gv_b(F)$ arising from $b$.  On the spectral side of the picture is the same category of ind-coherent sheaves on the moduli stack of Langlands parameter that we study.  Fargues-Scholze construct a spectral action of the category of perfect complexes on this moduli stack on the category of $\ell$-adic sheaves on $\Bun_{\Gv},$ and conjecture that there is an equivalence of this category with the category of ind-coherent sheaves on the moduli stack of Langlands parameters compatible with this spectral action.  Such an equivalence necessarily has the properties given in Theorem~\ref{fourth main thm intro}, although we do not attempt to verify that our construction is compatible with that of Fargues-Scholze.

\medskip
\subsection{Methods}\label{roma section}
We now discuss the techniques underlying the proofs of Theorems~\ref{first main thm intro} and~\ref{second main thm intro} -- namely, Bezrukavnikov's Langlands duality for the affine Hecke category and the theory of traces of monoidal dg categories. 

\medskip

\subsubsection{Bezrukavnikov's theorem} The Kazhdan-Lusztig theorem (Theorem~\ref{KL theorem}) has been famously categorified in the work of Bezrukavnikov~\cite{roma ICM, roma hecke}, with numerous applications in representation theory and the local geometric Langlands correspondence (see Theorem \ref{roma equiv}).

\medskip

\begin{thm} \label{roma thm}
Let $\mathbf{G} := G^\vee(\overline{\bF}_q((t)))$ denote the loop group viewed as an ind-scheme, and $\mathbf{I} \subset \mathbf{G}$ denote the corresponding Iwahori subgroup.  We define the (derived) Steinberg stack $\cZ/G$ over $\bQl$.  There is a monoidal equivalence on homotopy categories
$$D^b_c( \mathbf{I} \bs \mathbf{G} /\mathbf{I} ; \bQl) \simeq D^b\Coh(\cZ/G)$$
intertwining the pullback by geometric Frobenius and pushforward by multiplication by $q$ automorphisms.
\end{thm}

\medskip

\begin{remark}\label{mixed roma}
In view of Theorem \ref{roma thm}, we define the \emph{affine Hecke category} to be $\Hcat := \Coh(\cZ/G)$.  It is natural to expect a mixed version, identifying the \emph{mixed affine Hecke category} $\Hcat^\mix := \Coh(\cZ/\GG)$ with the mixed Iwahori-equivariant sheaves on the affine flag variety (as studied in~\cite{BY}). Indeed such a version is needed to directly imply the Kazhdan-Lusztig Theorem~\ref{KL theorem} by passing to Grothendieck groups, rather than its specialization at $q=1$.
\end{remark}

\medskip

Theorem~\ref{roma thm} establishes the ``principal block'' part of the local geometric Langlands correspondence. Namely, it implies a spectral description of 
module categories for the affine Hecke category (the geometric counterpart of unramified principal series representations) as suitable sheaves of categories on stacks of Langlands parameters. 

We apply Theorem~\ref{roma thm} in Section~\ref{iwahori} to construct a semiorthogonal decomposition of the affine Hecke category. This allows us to calculate its Hochschild homology and to establish the comparison with algebraic K-theory.

\medskip

\subsubsection{Trace Decategorifications}

To prove Theorem~\ref{second main thm intro} we use the relation between the ``horizontal" and ``vertical" trace decategorifications of a monoidal category, and the calculation of the subtler horizontal trace of the affine Hecke category in~\cite{BNP}.

Let $(\cat{C},\ast)$ denote a monoidal dg category. Then we can take the trace (or Hochschild homology) $\mathrm{tr}(\cat{C})=HH(\cat{C})$ of the underlying (i.e. ignoring the monoidal structure) dg category $\cat{C}$, which forms an associative (or $A_\infty$-)algebra $(\mathrm{tr}(\cat{C}),\ast)$ thanks to the functoriality (specifically the symmetric monoidal structure) of Hochschild homology, as developed in~\cite{TV2,HSS,CP,GKRV}.  This is the naive or ``vertical" trace of $\cat{C}$. On the other hand, a monoidal dg category has another trace or Hochschild homology $\Tr(\cat{C},\ast)$ using the monoidal structure which is itself a dg category -- the categorical or ``horizontal" trace of $(\cat{C},\ast)$. This is the dg category which is the universal receptacle of a trace functor out of the monoidal category $\cat{C}$.  In particular, the trace of the monoidal unit of $\cat{C}$ defines an object $[1_\cat{C}]\in \Tr(\cat{C},\ast)$ -- i.e., $\Tr(\cat{C},\ast)$ is a pointed (or $E_0$-)category\footnote{The horizontal trace is also the natural receptacle for characters of $\cat{C}$-module categories, and $[{\cat{C}}]$ appears as the character of the regular left $\cat{C}$-module, see Definition \ref{two trace maps}.}. Moreover, as developed in~\cite{CP,GKRV} the categorical trace provides a ``delooping" of the naive trace: we have an isomorphism of associative algebras
$$(\mathrm{tr}(\cat{C}),\ast) \simeq \End_{\Tr(\cat{C},\ast)}([1_{\cat{C}}]).$$
 In particular taking Hom from  $[1_{\cat{C}}]$ defines a functor
$$\Hom([1_{\cat{C}}],-): \Tr(\cat{C},\ast) \longrightarrow (HH(\cat{C}),\ast)\module.$$ Under suitable compactness assumptions the left adjoint to this functor embeds the ``naive" decategorification (the right hand side) as a full subcategory of the ``smart" decategorification (the left hand side).

More generally, given a monoidal endofunctor $F$ of $(\cat{C},\ast)$, we can replace Hochschild homology (trace of the identity) by trace of the functor $F$, obtaining two decategorifications (vertical and horizontal) with a similar relation
\begin{equation}
\Hom([1_{\cat{C}}],-):\Tr((\cat{C},\ast), F) \longrightarrow (\mathrm{tr}(\cat{C}, F),\ast)\module.
\end{equation}

\begin{remark}[Trace of Frobenius]\label{trace of F} When $\cat{C}$ is a category of $\ell$-adic sheaves on a stack defined over $\bF_q$ extended to $\overline{\bF}_q$ and $\Fr$ is the corresponding geometric Frobenius morphism, a formalism of categorical traces realizing the function-sheaf correspondence -- i.e. $\mathrm{tr}(\Sh(X), \Fr^*)$ should be the space of functions on $X(\bF_q)$ -- was recently established in \cite{AGKRRV}. The monoidal version of trace decategorification would then allow us to pass from Hecke categories to categories of representations directly.  Zhu~\cite{xinwen trace} explains some of the rich consequences of this formalism that can already be proved directly.
\end{remark}

\begin{example}[Finite Hecke Categories and unipotent representations]\label{finite Hecke} 
For the finite Hecke category $\cat{C}=\Sh(\BGB)$, the main theorem of~\cite{BN character} identifies $\Tr(\cat{C},\ast)$ with the full category of Lusztig unipotent character sheaves on $G$. The object $[1_{\cat{C}}]$ is the Springer sheaf itself, and modules for the naive decategorification $(\mathrm{tr}(\cat{C}, \mathrm{id}_{\cat{C}}), \ast)$ gives the Springer block, or unipotent principal series character sheaves, as modules for the graded Hecke algebra. Likewise the trace of Frobenius on $(\cat{C},\ast)$ is studied in~\cite[Section 3.2]{xinwen trace} (see also~\cite[Section 3.2]{dennis shtuka}), where the categorical trace is the category of all unipotent representations of $G(\bF_q)$, and the coherent Springer sheaf $[1_{\cat{C}}]$ generates the full subcategory consisting of the unipotent principal series, equivalent to modules for the naive decategorification $(\mathrm{tr}(\cat{C}, \Fr), \ast)$.
\end{example}

\medskip

\subsubsection{Trace of the affine Hecke category}
We now consider the two kinds of trace decategorification for the affine Hecke category $\Hcat$.  First our description of the Hochschild homology of the Steinberg stack provides a precise sense in which the affine Hecke category categorifies the affine Hecke algebra.  The following Corollary is a result of Theorems~\ref{roma thm} and~\ref{first main thm intro}.

\medskip

\begin{corollary}
The (vertical/naive) trace of Frobenius on the affine Hecke category is identified with the affine Hecke algebra $\cH_q \simeq \mathrm{tr}(\Hcat, \Fr^*).$ Hence the naive decategorification of $\Hcat\module$ is the category of unramified principal series representations of $\Gv(F)$. 
\end{corollary}

\medskip

\begin{remark}
Note that this corollary would follow directly from Theorem~\ref{roma thm} if we had available the hoped-for function-sheaf dictionary for traces of Frobenius on categories of $\ell$-adic sheaves (Remark~\ref{trace of F}). After this paper was complete Xinwen Zhu informed us that Hemo and he have a direct argument for this corollary, see the forthcoming~\cite{hemo zhu}. Combined with Bezrukavnikov's theorem and Theorem~\ref{BNP thm} this gives an alternative argument for the identification of $\Haff_q$ with the Ext algebra of the coherent Springer sheaf.
\end{remark}

The results of~\cite{BNP} (based on the technical results of~\cite{BNP2}) provide an affine analog of the results of~\cite{BN character, BFO} for finite Hecke categories and (thanks to Theorem~\ref{roma thm}) a spectral description of the full decategorification of $\Hcat$.  Statement (1) is directly taken from Theorem 4.4.1 in \cite{BNP}, statements (2)-(3) follow immediately from the same techniques and Theorem 3.8.5 of \cite{GKRV} (see Theorems \ref{trace theorem} and \ref{convolution category trace} and Lemma \ref{traceSpringerSheaf}), and the absence of a singular support condition is discussed in Remark \ref{no sing supp}.

\medskip

\begin{thm}[\cite{BNP}]\label{BNP thm} Let $G$ be a reductive group over $k = \bQl$ or $\C$.
\begin{enumerate}
\item The (horizontal/categorical) trace of the monoidal category $(\Coh(\cZ/G),\ast)$ is identified as
$$\Tr(\Coh(\cZ/G),*) = \Coh(\cL(\wh{\cN}/G)).$$ 
The same assertion holds with $G$ replaced by $\GG=G\times \Gm$. 
\item The trace of multiplication by $q \in \Gm$ acting on the monoidal category $(\Coh(\cZ/G),\ast)$ is identified as
$$\Tr((\Coh(\cZ/G),*), q_*) = \Coh(\BL_q^u).$$
\item The distinguished object $[1_{\cat{C}}]$ in each of these trace decategorifications is given by the coherent Springer sheaf $\cS$ (or its $q$-specialized version $\cS_q$).
Hence the endomorphisms of the coherent Springer sheaf recover the affine Hecke algebra (the vertical trace, as in  Theorem~\ref{second main thm intro}), and the natural functor in Theorem~\ref{two traces} is identified with
$$\Hom(\cS_q,-): \Coh(\BL_q^u)\longrightarrow \Haff_q\module.$$
\end{enumerate}
\end{thm}

\medskip

In other words, we identify the entire category of coherent sheaves on the stack of unipotent Langlands parameters as the categorical trace of the affine Hecke category. Inside we find the unramified principal series as modules for the naive trace (the Springer block). Just as the decategorification of the finite Hecke category (Example~\ref{finite Hecke}) knows all unipotent representations of Chevalley groups, 
the horizontal trace $\Coh(\BL_q^u)$ of the affine Hecke category contains in particular all unipotent representations of $\Gv(F)$ -- i.e., the complete L-packets of unramified principal series representations -- thanks to Lusztig's remarkable Langlands duality for unipotent representations:

\begin{thm}[\cite{lusztig unipotent}]\label{unipotent duality}
 The irreducible unipotent representations of $\Gv(F)$ are in bijection with $G$-conjugacy classes of triples $(s,n,\chi)$ with $s,n$ $q$-commuting as in Theorem~\ref{DL conjecture} and $\chi$ an {\em arbitrary} $G$-equivariant local system on the orbit of $(s,n)$. 
\end{thm}

\medskip

It would be extremely interesting to understand Theorem~\ref{unipotent duality} using trace decategorification of Bezrukavnikov's Theorem~\ref{roma thm}. In particular we expect the full category of unipotent representations to be embedded in $\IndCoh(\BL_q^u)$.

\medskip

\subsection{Assumptions and notation}

We work throughout over a field $k$ of characteristic zero.  Our results on traces hold in this general setting, though most representation theoretic applications will be in the specific case of $k = \bQl$ or $\bC$ (e.g. in Section \ref{two hecke}).  All functors and categories are dg derived unless noted otherwise.  

\subsubsection{Categories}

We work in the setting of $k$-linear stable $\infty$-categories, which for us will arise via applying the dg nerve construction (Construction 1.3.1.6 of \cite{HA}) to a pre-triangulated dg category.  These come in two primary flavors, ``big" and ``small": $\cat{dgCat}_k$ is the $\infty$-category of presentable stable $k$-linear $\infty$-categories (with colimit-preserving functors), and $\cat{dgcat}_k$ is the $\infty$-category of small idempotent-complete stable $k$-linear $\infty$-categories (with exact functors).   We denote the compact objects in a stable $\infty$-category $\cat{C}$ by $\cat{C}^\omega$, i.e. the objects $X \in \cat{C}$ for which $\Hom_{\cat{C}}(X, -)$ commutes with all infinite direct sums.  Both $\cat{dgCat}_k$ and $\cat{dgcat}_k$ are symmetric monoidal $\infty$-categories under the Lurie tensor product, with units $\cat{Vect}_k = k\dmod \in \cat{dgCat}_k$ and $\cat{Perf}_k = k\dperf \in \cat{dgcat}_k$ the dg categories of chain complexes of $k$-vector spaces and perfect chain complexes, respectively. We have a symmetric monoidal ind-completion functor:
$$\Ind: \cat{dgcat}_k\to \cat{dgCat}_k.$$
It defines an equivalence between $\cat{dgcat}_k$ and the subcategory of $\cat{dgCat}_k$ defined by {\em compactly generated} categories and compact functors (functors preserving compact objects, or equivalently, possessing colimit preserving right adjoints).

\medskip

Let $A$ be a Noetherian dg algebra.  We let $A\dmod = D(A) \in \cat{dgCat}_k$ denote the dg derived category of $A$-modules, $A\dperf = D_{perf}(A) \in \cat{dgcat}_k$ denote the full subcategory of perfect complexes, and $A\dcoh = D_{coh}(A)$ denote the full subcategory of cohomologically bounded complexes with coherent (i.e. finitely generated) cohomology.  Let $\cat{C}$ denote a symmetric monoidal dg category, and $A \in \mathrm{Alg}(\cat{C})$ an algebra object.  We denote by $A\dmod_{\cat{C}}$ (resp. $A\dperf_{\cat{C}}$) the category of $A$-module (resp. $A$-perfect) objects in $\cat{C}$; the category $A\dmod_{\cat{C}}$ is compactly generated by $A\dperf_{\cat{C}}$.  When $\cat{A} \in \cat{dgCat}_k$ is a cocomplete monoidal category, we denote by $\cat{A}\mh\cat{mod}$ the $(\infty, 2)$-category of $\cat{A}$-modules in $\cat{dgCat}_k$, i.e. cocomplete $\cat{A}$-module categories (see Section 3.6 of \cite{GKRV} for a definition).

\medskip

Assume that $\cat{C}$ is either small or that it is compactly generated, and let $X \in \cat{C}$ be an object, which we require to be compact in the latter case.  The notation $\langle X \rangle$ denotes the subcategory classically generated by $X$ when $\cat{C}$ is small (i.e. the smallest pretriangulated idempotent-complete subcategory containing $X$), and weakly generated by $X$ when $\cat{C}$ is cocomplete and compactly generated (i.e. the essential image of the left adjoint of $\Hom_{\cat{C}}(X, -)$).

\medskip

\subsubsection{Algebraic geometry}

In Section \ref{traces}, we work in the setting of derived algebraic geometry over an arbitrary field $k$ of characteristic zero as in~\cite{GR}.   Namely, this is a version of algebraic geometry in which functors of (discrete) categories from rings to sets are replaced by {\em prestacks}, functors of ($\infty$-)categories from connective commutative dg $k$-algebras to simplicial sets. Examples of prestacks are given by both classical schemes and stacks and topological spaces (or rather the corresponding simplicial sets of singular chains) such as $S^1$, considered as constant functors. 

\medskip

We will only be concerned with {\em QCA (derived) stacks} (or their formal completions along closed substacks) as in~\cite{DG}, i.e., quasi-compact stacks of finite presentation with affine\footnote{The notion of a QCA stack in \cite{DG} is slightly more general; only automorphism groups at geometric points are required to be affine, and they are not required to be of finite presentation.} finitely-presented diagonal (in fact only with quotients of schemes by affine group-schemes and their formal completions along closed substacks), and use the term {\em stack} to refer to such an object.

\medskip

A stack $X$ carries a symmetric monoidal $\oo$-category $\QC(X) \in \cat{dgCat}_k$ of quasicoherent sheaves, defined by right Kan extension from the case of representable functors $X=\Spec(R)$ which are assigned $\QC(\Spec R)=R\module$. For all stacks we will encounter (and more generally for {\em perfect} stacks in the sense of~\cite{BFN}), we have $\QC(X)\simeq \Ind(\Perf(X))$, i.e., quasicoherent sheaves are compactly generated and the compact objects are perfect complexes.

\medskip

We can also consider the category $\QC^!(X) \in \cat{dgCat}_k$ of {\em ind-coherent sheaves}, whose theory is developed in detail in the book~\cite{GR} (see also the earlier~\cite{indcoh}). The category $\QC^!(X)$ (under our assumption that $X$ is QCA) is compactly generated by $\Coh(X)$, the objects which are coherent after smooth pullback to a scheme (see Theorem 3.3.5 of \cite{DG}). 
For smooth $X$, the notions of coherent and perfect, hence ind-coherent and quasicoherent, sheaves are equivalent.

\medskip

A crucial formalism developed in detail in~\cite{GR} is the functoriality of $\QC^!$. Namely for an almost finite-type map $p:X\to Y$ of stacks, we have colimit-preserving functors of pushforward $p_*:\QC^!(X)\to \QC^!(Y)$ and exceptional pullback $p^!:\QC^!(Y)\to \QC^!(X)$, which form an adjoint pair $(p_*,p^!)$ for $p$ proper. These functors satisfy a strong form of base change, which makes $\QC^!$ a functor -- in fact a symmetric monoidal functor\footnote{In general $\QC^!$ is only lax symmetric monoidal but thanks to~\cite{DG} it is strict on QCA stacks. Also the full correspondence formalism in~\cite{GR} only includes pushforward for [inf,ind-]schematic maps.} -- out of the category of correspondences of stacks (the strongest form of this result is~\cite[Theorem III.3.5.4.3, III.3.6.3]{GR}).

\medskip

We note that for a closed substack $Z \subset X$, the category of quasicoherent (or ind-coherent, or perfect, et cetera) sheaves $\QCoh(\wh{Z})$ on the formal completion $\wh{Z}$ is canonically equivalent to the category $\QCoh_Z(X)$ of sheaves on $X$ set-theoretically supported on $Z$.

\medskip

See Definition 2.3.1 of \cite{Ch} for a definition of the derived loop space $\cL(-)$.  For a stack $X$ with a self-map $f$, we define $\cL_f(X)$ to be the derived fixed points of $f$, i.e. the derived fiber product
$$\begin{tikzcd}
\cL_f(X) \arrow[r] \arrow[d] & X \arrow[d, "{(f, \mathrm{id}_X)}"] \\
X \arrow[r, "\Delta"] & X \times X.
\end{tikzcd}$$
When $f = \mathrm{id}_X$, we have $\cL_f X = \cL X$.  Given a group action $G$ on a scheme $X$, and $f: X \rightarrow X$ commuting with the $G$-action, we have via Proposition 2.1.8 of \cite{Ch} a Cartesian diagram:
$$\begin{tikzcd}
\cL_f(X/G) \arrow[r] \arrow[d] & (X \times G)/G \arrow[d, "{(f \circ \alpha, \mathrm{id}_X)}"] \\
X/G \arrow[r, "\Delta"] & (X \times X)/G
\end{tikzcd}$$
where $\alpha$ is the action map.

\medskip
\subsubsection{Representation theory}\label{rep thy defs}

In Sections \ref{iwahori}, \ref{section coherent springer} and \ref{moduli}, unless otherwise noted, $G$ denotes a split reductive group over a field $k = \bQl$ or $\C$) with a choice of Borel $B$ and torus $T \subset B$ with universal Cartan $H$ and (finite) universal Weyl group $W_f$.  The extended affine Weyl group is denoted $W_a := X^\bullet(H) \rtimes W_f$.   We denote by $\Rep(G) = \QCoh(BG)$ the derived category of rational representations of $G$.  Likewise, $\mathfrak{g} = \Lie(G)$, $\mathfrak{b} = \Lie(B)$, et cetera.  

\medskip

Morally, we view $G$ as a group on the spectral side of Langlands duality.  On the automorphic side, one is interested in representations of the split group $G^\vee(F)$, where we let $F$ denote a non-archimedian local field with ring of integers $O$.   We denote by $I$ the Iwahori subgroup with pro-unipotent radical $I^\circ$, defined by the fixed Borel subgroup $B^\vee \subset G^\vee$ and maximal hyperspecial $G^\vee(O) \subset G^\vee(F)$.  \emph{In Section \ref{types}, we will reverse this convention for ease of reading}, and $G$ will denote a split reductive group over $F$.

\medskip

We will often be interested in equivariance with respect to the trivial extension of $G$ by $\bG_m$ which we denote 
$\GG=G\times \G_m$; this amounts to additional weight grading on coherent sheaves.  We fix once and for all a coordinate $z \in \G_m$.  For any geometric vector space or bundle $V$ (e.g. a Lie algebra $\mf{g}$ or the Springer resolution $\wt{\cN}_G$ introduced below), by convention the coordinate will act on geometric fibers by weight -1, i.e. $z \cdot x = z^{-1} x$ for $x \in V$, and therefore on functions by weight 1 (i.e. $z \cdot f(-) = zf(-)$ for $f \in V^*$).  This negative sign convention corresponds to the convention that the $z=q$ fixed points of $\cN/\GG$ correspond to unipotent Langlands parameters $(s, N)$ for a local field with residue $\mathbb{F}_q$, i.e. $(s, N, q) \cdot N = sNs^{-1}q^{-1} = N$.\footnote{Letting $q$ denote the action by $q$ in the above convention (i.e. multiplication by $q^{-1}$), we have $q_* = \mathbf{q}^*$, where $\mathbf{q}^*$ is the functor in Section 11.1 of \cite{roma hecke}.  Thus, given an identification $\cH \simeq \mathrm{tr}(\Hcat^\mix, \mathrm{id}_{\Hcat^\mix})$ as in Theorem \ref{ThmHH}, this implies an identification $\cH_q \simeq \mathrm{tr}(\Hcat, q_*) \simeq \mathrm{tr}(\Hcat, \mathrm{Fr}^*)$.  This convention is compatible with \cite{KL, CG, AB:affine flags, roma hecke}.}

\medskip

Let $\cB_G = G/B$ denote the flag variety, $\cN_G$ denote the nilpotent cone, and $\wh{\cN}_G$ its \emph{formal neighborhood} inside $\mf{g}$.  We let $\wt{\cN}_G$ denote the (reduced) Springer resolution, and denote by $\mu: \wt{\cN}_G = T^*(\cB_G) \rightarrow \cN_G \hookrightarrow \wh{\cN}_G$ the composition of the Springer resolution with the inclusion, and $\widetilde{\mathfrak{g}}$ the Grothendieck-Springer resolution, which is $\GG$-equivariant.  Sometimes, we take the codomain of $\mu$ to be all of $\mf{g}$.  Let $\cZ_G= \tilN_G \times_{\mathfrak{g}} \tilN_G$ denote the \emph{derived Steinberg
scheme}, $\cZ'_G = \tilN_G \times_{\mathfrak{g}} \widetilde{\mathfrak{g}}$ denote the \emph{non-reduced Steinberg scheme}, and $\cZ^\wedge_G = (\wt{\mf{g}} \times_{\mf{g}} \wt{\mf{g}})^\wedge$ denote the \emph{formal Steinberg scheme} via completing along the nilpotent elements.  We denote by $\pi_0(\cZ_G)$ the \emph{classical Steinberg variety}, which coincides with $(\cZ'_G)^{red} = (\cZ^\wedge_G)^{red}$.  We will drop the subscript if there is no ambiguity regarding the group $G$ in discussion.

\medskip

We denote the \emph{affine Hecke algebra} by $\Haff_{G}$; we use a Coxeter presentation, i.e. a definition on the spectral side, which  can be found e.g. in Definition 7.1.9 of \cite{CG}.  It is a $k[q, q^{-1}]$-algebra whose specializations at prime powers $q = p^r$ are isomorphic to the \emph{Iwahori-Hecke algebras} $\Haff_{q,G} \simeq \cH(G^\vee(F), I) := C^\infty_c(I \bs G^\vee(F)/I; k)$ of compactly supported Iwahori-biequivariant functions on a loop group (or $p$-adic group).  More generally, for a locally compact totally disconnected group $G$ (now viewed on the automorphic side), a compact open subgroup $K \subset G$ and a representation $\tau$ of $K$, we denote its Hecke algebra by $\cH(G, K, \tau) := \End_{G}(\cind_{K}^G \tau)$ (these appear in Section \ref{moduli}).

\medskip

The \emph{mixed affine Hecke category} is defined by $\Hcat^\mix_{G} := \Coh(\cZ/\GG)$, while the \emph{affine Hecke category} is defined to be $\Hcat_{G} := \Coh(\cZ/G)$.  Note that we define these categories directly on the spectral side of Langlands duality, while they are usually defined on the automorphic side.  That is, we implicitly pass through Bezrukavnikov's theorem (Theorem \ref{roma thm}).

\medskip

We define the \emph{coherent Springer sheaf} and the \emph{coherent $q$-Springer sheaf} by:
$$\cS_G := \LL \mu_* \OO_{\LL(\wt{\cN}/\GG)} \simeq \LL \mu_* \omega_{\LL(\wt{\cN}/\GG)} \in \Coh(\LL(\wh{\cN}/\GG)),$$
$$\cS_{q,G} := \LL_q\mu_* \OO_{\LL_q(\wt{\cN}/G)} \simeq  \LL_q\mu_* \omega_{\LL_q(\wt{\cN}/G)} \in \Coh(\LL_q(\wh{\cN}/G)).$$
The coherent $q$-Springer sheaf is a coherent sheaf on the \emph{stack of unipotent Langlands parameters}:
$$\mathbb{L}^u_{q,G} := \cL_q(\wh{\cN}/G).$$
Note that this definition is functorial and makes sense for \emph{any} affine algebraic group $G$ (still completing along nilpotents), and thus the coherent $q$-Springer sheaf may be realized by applying parabolic induction
$$\begin{tikzcd}
\bL^u_{q,T} & \arrow[l, "\nu"'] \bL^u_{q,B} \arrow[r, "\mu"] & \bL^u_{q,G}
\end{tikzcd}$$
to the structure sheaf of $\bL^u_{q,T}$, i.e. $\cS_{q,G} = \mu_*\nu^* \cO_{\bL^u_{q,T}}$ (where $T$ is the quotient torus of $B$, and does not depend on a choice of lift).  By Proposition \ref{no derived structure}, if $G$ is reductive then $\bL^u_{q,G}$ is a classical stack (i.e. no derived and no infinitesimal structure) when $q$ is not a root of unity.  Note that other authors \cite{ bellovin-gee, booher-patrikis, curtis,DHKM, xinwen survey} have defined a moduli stack of Langlands parameters $X_{F, G}$ for a given local field $F$ and a reductive group $G^\vee$ with coefficients in $F$.  Our stack embeds as a connected component of tame Langlands parameters.

\medskip

\subsection{Acknowledgments}

We would like to thank Xinwen Zhu for very enlightening conversations on the topic of categorical traces, the Drinfeld-Lafforgue sheaf and its relation to the coherent Springer sheaf and for sharing with us an early draft of his paper~\cite{xinwen survey}, Pramod Achar for discussions of purity and Tate-ness properties in Springer theory, and Sam Raskin for suggestions related to renormalized categories of sheaves on formal odd tangent bundles.  We would also like to thank Matthew Emerton for comments regarding Whittaker normalizations, Xuhua He for pointing out the reference~\cite{reeder}, Maarten Solleveld for discussion surrounding results in~\cite{OS} and Gurbir Dhillon for numerous helpful discussions.  Finally, we would like to thank the anonymous referee for their thoughtful and detailed suggestions.

\section{Hochschild homology of the affine Hecke category}\label{iwahori}
In this section we calculate the Hochschild homology of the affine Hecke category. In particular in Corollary \ref{actual theorem} we prove that the Chern character from $K$-theory factors through an isomorphism between $K_0$ and Hochschild homology. 
For this we use Bezrukavnikov's Langlands duality for the affine Hecke category to construct a semiorthogonal decomposition on the equivariant derived category of the Steinberg stack with simple components, from which the calculation of localizing invariants is immediate. 

\medskip
\subsection{Background}

We first review some standard notions regarding Hochschild homology and equivariant $\ell$-adic sheaves that we need for our arguments.  In this subsection we take $k$ to be any field of characteristic 0.

\subsubsection{Trace decategorifications and Hochschild homology}\label{HHSec}

An extended discussion of the notions in this subsection can be found in \cite{TV2, GKRV, BN:NT, Ch}.  We recall the notion of a {\em dualizable object} $X$ of a symmetric monoidal $(\infty, 2)$-category $\cat{C}_\otimes$ with monoidal unit $1_\otimes$ (see the Appendix of \cite{GR} for a definition).
\begin{defn}\label{def trace decat}
The object $X$ is \emph{dualizable} if there exists an object $X^\vee$ and coevaluation and evaluation morphisms $$\eta_X: 1_{\otimes} \rightarrow X \otimes X^\vee, \hskip.3in \epsilon_X: X^\vee \otimes X \rightarrow 1_{\otimes}$$ satisfying a standard identity.  Dualizability is a {\em property} rather than an additional structure on $X$ (see Proposition 4.6.1.10 in \cite{HA}).   The {\em trace} of an endomorphism $f\in \End_{\cat{C}}(X)$ of a dualizable object is defined by
$$\mathrm{tr}(X, f) = \epsilon_X \circ (f \otimes 1) \circ  \eta_X \in \End_{\cat{C}_{\otimes}}(1_{\otimes}).$$
\end{defn}

\begin{rmk}
Note that $\End(1_\otimes)$ is naturally enriched as an object of $\cat{C}_\otimes$ which is universal amongst objects tensored over $1_\otimes$, i.e. there is a natural equivalence of algebras $\End(1_\otimes) \simeq 1_\otimes$.  In particular, $\End(1_\otimes)$, which is a priori only an $A_\infty$-algebra, is an $E_\infty$-algebra (see the discussion in Section 4.7.1 of \cite{HA} for details).
\end{rmk}

\medskip

The notion of dualizability depends only on the $1$-categorical structure of $\cat{C}_\otimes$.  However, in our applications, 
we are interested in the case when $X$ is an algebra object in the symmetric monoidal $\infty$-category $\cat{C}_\otimes$, and the resulting algebra structure on traces.  To formulate this, we require a functoriality on traces involving (right-)dualizable 2-morphisms in $\cat{C}_\otimes$; this discussion requires the presence of non-invertible 2-morphisms in $\cat{C}_\otimes$.

\medskip

Since $\cat{C}_\otimes$ is a monoidal $(\infty, 2)$-category, the endomorphisms of the monoidal unit $\cat{End}_{\cat{C}_\otimes}(1_\otimes)$ in fact form an $(\infty, 1)$-category.  We have the following natural functoriality enjoyed by the abstract construction of traces in the higher-categorical setting; see~\cite{TV2,HSS,GKRV} (and~\cite{BN:NT} for an informal discussion). Namely the trace of an object is covariantly functorial under right dualizable morphisms.
\begin{defn}
A \emph{morphism of pairs} $(F, \psi): (X, f) \rightarrow (Y, g)$ is a right dualizable morphism $F: X \rightarrow Y$ (i.e. has a right adjoint $G$) between dualizable objects along with a commuting structure $\psi: F \circ f \rightarrow g \circ F$.  Given a morphism of pairs $(F, \psi)$, it defines a map $\mathrm{tr}(F, \psi)$ on traces via the composition
$$\begin{tikzcd}[column sep=2.3em]
\mathrm{tr}(X, f) \arrow[rr, "{\mathrm{tr}(X, \eta_F \mathrm{id}_f})"] & & \mathrm{tr}(X, GFf) \arrow[rr, "{\mathrm{tr}(X, \mathrm{id}_G \psi})"] & & \mathrm{tr}(X, GgF) \arrow[r, "\simeq"] & \mathrm{tr}(Y, gFG) \arrow[rr, "{\mathrm{tr}(Y, \mathrm{id}_g \epsilon_F)}"] & & \mathrm{tr}(Y, g)
\end{tikzcd}$$
where $\eta_F$ and $\epsilon_F$ are the unit and counit of the adjunction $(F, G)$, and the equivalence in the middle is via cyclic symmetry of traces (see also Definition 3.24 of \cite{BN:NT}).
\end{defn}

Note that taking the trace is canonically \emph{symmetric} monoidal with respect to the monoidal structure in $\cat{C}_\otimes$ and composition in $\cat{End}_{\cat{C}_\otimes}(1_\otimes)$ (or equivalently, tensoring in $1_\otimes$).  The trace construction enhances to a symmetric monoidal functor from the $\infty$-category of endomorphisms of dualizable objects in $\cat{C}_\otimes$ to the category $\cat{End}_{\cat{C}_\otimes}(1_\otimes) \simeq 1_\otimes$, see~\cite[2.5]{TV2},~\cite[2]{HSS}, and \cite[3]{GKRV} for details. In particular, if $X$ is an algebra object in $\cat{C}_\otimes$ with right dualizable unit and multiplication, and $f: X \rightarrow X$ is a map of algebra objects, then $\mathrm{tr}(X, f)$ is an algebra object in $\cat{End}_{\cat{C}_\otimes}(1_\otimes)$.

\medskip

In this paper, we consider the $\infty$-category $\cat{C}_\otimes = \cat{dgCat}_k$ of presentable (i.e. cocomplete) $k$-linear dg categories, with morphisms given by colimit-preserving (i.e. continuous, or left adjoint) functors, with monoidal product the Lurie tensor product.  We now specialize to this setting.
\begin{exmp}
Any presentable compactly generated dg category $\cat{C}=\Ind(\cat{C}^\omega)\in \cat{dgCat}_k$ is dualizable, with dual given by taking the ind-completion of the opposite of compact objects $\cat{C}^\vee=\Ind(\cat{C}^{\omega,op})$. Thus we may speak of traces of its endofunctors, which are endomorphisms of the unit, i.e. \emph{chain complexes}
$$\cat{End}_{\cat{dgCat}_k}(\cat{Vect}_k)\simeq \cat{Vect}_k.$$
Furthermore, note that a right dualizable morphism of presentable compactly generated dg categories must have a \emph{colimit-preserving} right adjoint, or equivalently is a functor which preserves compact objects.
\end{exmp}

\begin{defn} The {\em Hochschild homology} of a dualizable (for instance, compactly generated) presentable $k$-linear dg category $\cat{C} \in \cat{dgCat}_k$ is the trace of the identity functor
 $$\hh(\cat{C}/k):=\mathrm{tr}(\cat{C}, \mathrm{id}_{\cat{C}})\in \cat{Vect}_k.$$ We often omit $k$ from the notation above.  More generally, the Hochschild homology of $\cat{C}$ with coefficients in a colimit-preserving endofunctor $F$ is $HH(\cat{C},F)=\mathrm{tr}(\cat{C},F)\in \cat{Vect}_k$. 
\end{defn}

\begin{rmk}[Large vs. small categories]
The above definition is formulated in terms of large categories, but can be defined for small categories by taking ind-completions.  Since every compactly generated category is dualizable but not conversely, the notion of Hochschild homology for large categories is more general.  We will often not distinguish between the two.
\end{rmk}

We have a notion of characters of compact objects in categories, defined via functoriality of traces.
\begin{defn}\label{character def}
Let $\cat{C} \in \cat{dgCat}_k$ be dualizable, and $F: \cat{C} \rightarrow \cat{C}$ an endofunctor.  Any object $c \in \mathrm{Ob}(\cat{C})$ defines a functor $\alpha_c: \cat{Vect}_k \rightarrow \cat{C}$ by action on the object $c$, and a map $\psi: c \rightarrow F(c)$ defines a commuting structure.  If $c$ is a compact object, then $\alpha_c$ is right dualizable.  Thus, by functoriality of traces, we have a map 
$$\mathrm{tr}(\alpha_c, \psi): HH(\cat{Vect}_k) = k \longrightarrow HH(\cat{C}, F)$$
and the \emph{character}\footnote{This may also sometimes be referred to as a trace, but we call it a character to avoid overloading the term.} $[c] = \mathrm{tr}(\alpha_c, \psi)(1)$ of $c$ is the image of $1 \in k$ under this map.
\end{defn}

\begin{rmk}\label{HH properties}
We highlight a few properties of Hochschild homology which we use in our arguments:
\begin{enumerate}
\item Hochschild homology is a localizing invariant in the sense of \cite{BGT} by Theorem 5.2 of \cite{keller dgcat}, and in particular in the explicit algebraic model of Definition \ref{alg model} one can replace $\mathrm{Ob}(\cat{C})$ with any set of generating objects.
\item Hochschild homology takes (possibly infinite) $F$-stable semiorthogonal decompositions (see Section \ref{semiorthog}) of $\cat{C}$ to direct sums.  This is a consequence of (1) since semiorthogonal decompositions give rise to exact sequences of categories.
\item Let $A$ be a dg algebra, $M$ an dg $A$-bimodule, and define $F_M(-) = M \otimes_A -$.  Then, $HH(A\dmod, F_M) = A \otimes_{A \otimes_k A^{op}} M$. This derived tensor product can be computed via a bar resolution or otherwise.
\item The Hochschild homology receives a Chern character map from the connective $K$-theory spectrum (see Definition \ref{K theory chern character}).
\end{enumerate}
\end{rmk}

\begin{exmp}
We give a toy example to illustrate a canonical identification of two calculations of Hochschild homology.  Let $\cat{C} = \Coh(\mathbb{P}^1)$.  It is well-known that $\cO(-1) \oplus \cO$ generates the category, with endomorphism algebra represented by the Kronecker quiver.  Since the Kronecker quiver has no cycles, we have an identification $\hh(\Coh(\mathbb{P}^1)) \simeq k^2$.  The character map is the (twisted) algebraic Euler characteristic:
$[\cL] = (\chi(\mathbb{P}^1, \cL(1)), \chi(\mathbb{P}^1, \cL)).$

On the other hand, the Hochschild-Kostant-Rosenberg isomorphism (see Theorem 4.1 of \cite{caldararu hkr}) identifies the Hochschild homology of a smooth variety with the global sections of its negatively-shifted algebra of differential forms, which in this example produces an identification $\hh(\Coh(\mathbb{P}^1)) \simeq H^0(\mathbb{P}^1, \cO_{\mathbb{P}^1}) \oplus H^1(\mathbb{P}^1, \Omega^1_{\mathbb{P}^1}) \simeq k^2$.  The character map is the Chern character, i.e. $[\cO(n)] = (1, n)$; compatibility of traces forces a particular identification $H^{0,0}(\mathbb{P}^1) \oplus H^{1,1}(\mathbb{P}^1) \simeq \End(\cO(-1)) \oplus \End(\cO)$.
\end{exmp}

\medskip

\subsubsection{The cyclic bar and Block-Getzler complex}\label{block getzler sec}

The Hochschild homology of compactly generated (or equivalently, small) categories has an algebraic realization via the cyclic bar complex, which we briefly recall; see Section 5.3 in \cite{keller dgcat} for further discussion.  In the below definition, we relax the condition that $\cat{C}$ is pretriangulated; morally it should be thought of as a full subcategory of $F$-fixed compact generators of a cocomplete dg category.
\begin{defn}
The \emph{cyclic bar complex} of a small $k$-linear dg category $\cat{C}$, equipped with a dg-endofunctor $F$, is defined to be the sum totalization of the simplicial chain complexes with\footnote{Note that for consistency we label using cohomological grading, and that we are defining the complex of Hochschild chains and \emph{not} the complex of Hochschild cochains.}
$$\cC^{-n}(\cat{C}, F) = \bigoplus_{X_0, \ldots, X_n \in \mathrm{Ob}(\cat{C})} \Hom^\bullet_{\cat{C}}(X_0, X_1) \otimes \cdots \otimes \Hom^\bullet_{\cat{C}}(X_{n-1}, X_n) \otimes \Hom^\bullet_{\cat{C}}(X_n, F(X_0))$$
where the face maps $d_i: \cC^{-n} \rightarrow \cC^{-(n-1)}$ (for $i = 0, \ldots, n$) compose morphisms, i.e.
$$d_i(f_0 \otimes \cdots \otimes f_n) = f_0 \otimes \cdots f_i f_{i+1} \otimes \cdots \otimes f_n, \;\;\;\;\;\;\; i=0, \ldots, n-1$$
$$d_n(f_0 \otimes \cdots \otimes f_n) = f_n F(f_0) \otimes F(f_1) \otimes \cdots \otimes F(f_{n-1}).$$
If $\cat{C}$ is a monoidal dg category, and $F$ has the structure of a monoidal functor, then $HH(\cat{C}, F)$ is an (associative) dg algebra via functoriality and the shuffle or Eilenberg-Zilber map.
\end{defn}

\medskip

Now, let $k$ be a field of characteristic 0, and $G$ a reductive group over $k$.  For dg categories with a $\Rep(G)$-action, there is an explicit algebraic model for the Hochschild homology due to Block and Getzler~\cite{BG:ECH}.  We define a $\Rep(G)$-internal $\Hom$ for $\Rep(G)$-module categories in the following standard lemma.
\begin{lemma}\label{deeq}
Let $\cat{C}$ be a $\Rep(G)$-module category.  The $\Hom$-sets of $\cat{C}$ are canonically enriched in $\Rep(G)$ such that
$$\Hom_{\cat{C}}(X, Y) = \underline{\Hom}_{\cat{C}}(X, Y)^G$$
where $\underline{\Hom}$ denotes the $\Rep(G)$-internal Hom.  In particular, if $E \in \cat{C}$ is a compact $\Rep(G)$-generator for $\cat{C}$, then $\cat{C}$ is equivalent to modules in $\Rep(G)$ for the internal endomorphism algebra 
$$\underline{A} = \underline{\End}_{\cat{C}}(E)^{op}\in \mathrm{Alg}(\Rep(G)).$$
\end{lemma}
\begin{proof}
The lemma is an application of the rigidity of $\Rep(G)$ and the Barr-Beck-Lurie monadicity theorem.  
The internal $\Hom$ is defined in the following way.  For any $X \in \cat{C}$, the functor $\mathrm{act}_X:\Rep(G)\to \cat{C}$ given by action on $X$ has a $\Rep(G)$-linear colimit-preserving right adjoint $\Psi_X(-)=\underline{\Hom}_{\Rep(G)}(X,-)$.  We define $\underline{\Hom}_{\cat{C}}(X, Y) = \Psi_X(Y)$.  More explicitly, we have
$$\underline{\Hom}_{\cat{C}}(X, Y) = \Hom_{\cat{C}}(X, Y \otimes \cO(G)) = \bigoplus_{V \in \mathrm{Irr}(G)} \Hom_{\cat{C}}(X, Y \otimes V) \otimes V^*$$
where $G$ acts on $\cO(G)$ by conjugation.  Note that $\Psi_E$ takes $E$ to the internal endomorphism algebra, which represents the corresponding monad $\Psi_E \circ \mathrm{act}_E$ on $\Rep(G)$.  By rigidity, this monad is $\Rep(G)$-linear, thus is given by tensoring with $\underline{A} = \underline{\End}_{\cat{C}}(E)^{op}$, its value on the monoidal unit.  The functor $\Psi_X$ is monadic; it preserves colimits since its left adjoint preserves compactness, and it is conservative since $E$ is a $\Rep(G)$-generator, thus the claim follows by Barr-Beck.
\end{proof}

\medskip

Block and Getzler defined a chain complex in \cite{BG:ECH} associated to any dg category $\cat{C}$ enriched in $\Rep(G)$.  We review this notion here.  We often do not take the entire category $\cat{C}$, but a full subcategory which generates under the $\Rep(G)$-action (but is not closed under it).
\begin{defn}\label{alg model}
Let $G$ be a reductive group, and let $\cat{C}$ be a small dg category enriched in $\Rep(G)$ equipped with an $\Rep(G)$-enriched dg-endofunctor $F$.  For any $V \in \Rep(G)$, we abusively denote by $\gamma: V \rightarrow V \otimes k[G]$ the coaction map.  The \emph{Block-Getzler complex} (over $k$) $\cC_G^\bullet(\cat{C}, F)$ is defined to be the sum totalization of the simplicial object in chain complexes with
$$\cC_G^{-n}(\cat{C}, F) = \bigoplus_{X_0, \ldots, X_n \in \mathrm{Ob}(\cat{C})} \Big( \uHom^\bullet_{\cat{C}}(X_0, X_1) \otimes \cdots  \otimes \uHom^\bullet_{\cat{C}}(X_n, F(X_0)) \otimes k[G] \Big)^G$$
where $G$ acts on $k[G]$ by conjugation, and the face maps $d_i: \cC_G^{-n} \rightarrow \cC_G^{-(n-1)}$ (for $i = 0, \ldots, n$) compose morphisms, i.e.
$$d_i(f_0 \otimes \cdots \otimes f_n \otimes g) = f_0 \otimes \cdots f_i f_{i+1} \otimes \cdots \otimes f_n \otimes g, \;\;\;\;\;\;\; i=0, \ldots, n-1$$
$$d_n(f_0 \otimes \cdots \otimes f_n \otimes g) = \gamma(f_n) F(f_0) \otimes F(f_1) \otimes \cdots \otimes F(f_{n-1}) \otimes g.$$
We define the \emph{enhanced Block-Getzler complex} to $\underline{\cC}_G^\bullet(\cat{C}, F)$ to be the  complex above, but without taking $G$-invariants.\footnote{Note that if $F$ is the identity functor, then the Block-Getzler simplicial chain complex is a cyclic object, and thus the associated chain complex has the natural structure of a mixed complex.  However, the enhanced Block-Getzler complex is not cyclic, since the ``rotation'' twists by the coaction $\gamma$ which can be nontrivial on nontrivial $G$-isotypic components.  One can view this object as an $S^1$-equivariant object in $\QCoh(G/G)$.}  Finally, for a specified $g \in G(k)$ we define $$\cC^\bullet_{G,g}(\cat{C}, F) = \underline{\cC}_G^\bullet(\cat{C}, F) \otimes_{k[G]} k_g$$
where $k_g$ is the skyscraper module at $g \in G$.  Note that there are canonical maps
$$\cC^\bullet_G(\cat{C}, F) \hookrightarrow \underline{\cC}_G^\bullet(\cat{C}, F) \rightarrow \cC^\bullet_{G,g}(\cat{C}, F).$$
\end{defn}

\bigskip

We now wish to show that the Block-Getzler complex computes Hochschild homology for $\Rep(G)$-module categories.  Letting $(-)^{\deeq} := - \otimes_{\Rep(G)} \cat{Vect}_k$ denote the de-equivariantization, since $\cat{Vect}_k^\deeq \simeq \QCoh(G)$ for any $g \in G(k)$ we have an automorphism $g_*$ of the category $\cat{C}^\deeq$ induced by the action of the skyscraper sheaf $k_g$ at $g \in G(k)$.  For any $\Rep(G)$-linear endofunctor $F: \cat{C} \rightarrow \cat{C}$, consider the squares
$$\begin{tikzcd}
\cat{C} \arrow[d, "F"] \arrow[r] & \cat{C}^\deeq \arrow[d, "F^\deeq"] \arrow[r, "g_*"] & \cat{C}^\deeq \arrow[d, "F^\deeq"] \\
\cat{C} \arrow[r] & \cat{C}^\deeq \arrow[r, "g_*"] & \cat{C}^\deeq.
\end{tikzcd}$$
The left square is equipped with a canonical commuting structure coming from the $\Rep(G)$-linear structure of $F$, and the right square is equipped with a canonical commuting structure since $F^\deeq$ acquires a natural $\QCoh(G) \simeq \cat{Vect}_k^\deeq$-linear structure.  We denote by $F^\deeq_g := F^\deeq \circ g_* \simeq g_* \circ F^\deeq$, and consider the map of pairs $\Psi: (\cat{C}, F) \rightarrow (\cat{C}^\deeq, F_g^\deeq)$.
\begin{prop}\label{functoriality BG}
Let $G$ be a reductive group (over $k$), $\cat{C}$ be a small dg category with a $\Rep(G)$-action, and $F$ a $\Rep(G)$-linear endofunctor.  Let $\cat{C}_0 \subset \cat{C}$ be a full subcategory, closed under $F$, which generates $\cat{C}$ over $\Rep(G)$.  Then, the map $\cC_G^\bullet(\cat{C}_0, F) \rightarrow \cC^\bullet_{G,g}(\cat{C}_0, F)$ computes the map in Hochschild homology $HH(\Psi): HH(\cat{C}, F) \rightarrow HH(\cat{C}^\deeq, F_g^\deeq)$.
\end{prop}
\begin{proof}
The claim that $\cC_G^\bullet(\cat{C}_0, F)$ computes $HH(\cat{C}, F)$ is similar to Proposition 2.3.6 of \cite{Ch}.  Since $\cat{C}_0$ (compactly) generates $\cat{C}$ under the $\Rep(G)$-action, to compute Hochschild homology we may use the cyclic bar complex with $n$th term
$$\bigoplus_{X_i \in \cat{C}_0} \bigoplus_{V_i \in \mathrm{Irr}(G)} \Hom_{\cat{C}}(X_0 \otimes V_0, X_1 \otimes V_1) \otimes \cdots \otimes \Hom_{\cat{C}}(X_n \otimes V_n, F(X_0) \otimes V_0)$$
$$\simeq \bigoplus_{X_i \in \cat{C}_0} \bigoplus_{V_i \in \mathrm{Irr}(G)} \uHom_{\cat{C}}(X_0 \otimes V_0, X_1 \otimes V_1)^G \otimes \cdots \otimes \uHom_{\cat{C}}(X_n \otimes V_n, F(X_0) \otimes V_0)^G$$
$$\simeq \bigoplus_{X_i \in \cat{C}_0} \bigoplus_{V_i \in \mathrm{Irr}(G)} (V_0^* \otimes \uHom_{\cat{C}}(X_0, X_1) \otimes V_1)^G \otimes \cdots \otimes (V_n^* \otimes \uHom_{\cat{C}}(X_n, F(X_0)) \otimes V_0)^G.$$
By Proposition 2.3.2 of \emph{op. cit.} we have
$$\simeq \bigoplus_{X_i \in \cat{C}_0} \bigoplus_{V_0 \in \mathrm{Irr}(G)} (V_0^* \otimes \uHom_{\cat{C}}(X_0, X_1) \otimes \cdots \otimes \uHom_{\cat{C}}(X_n, F(X_0)) \otimes V_0)^G.$$
By Peter-Weyl, we have
$$\simeq \bigoplus_{X_i \in \cat{C}_0} (\uHom_{\cat{C}}(X_0, X_1) \otimes \cdots \otimes \uHom_{\cat{C}}(X_n, F(X_0)) \otimes k[G])^G.$$
These identifications are compatible with the face maps by a straightforward diagram chase.

The claim that $\cC_{G,g}^\bullet(\cat{C}_0, F)$ computes $HH(\cat{C}^\deeq, F_g^\deeq)$ follows from the observation that if $\cat{C}_0$ (compactly) generates $\cat{C}$ over $\Rep(G)$, then its image in the de-equivariantization (compactly) generates $\cat{C}^\deeq$, and that the $\Hom$-spaces in $\cat{C}^\deeq$ are obtained from the $\Rep(G)$-internal $\Hom$-spaces of $\cat{C}$ after forgetting the $G$-module structure.  Thus $\cC^\bullet_{G,g}(\cat{C}_0, F)$ is just the cyclic bar complex via the identification of the last tensor factor (implicitly using the commuting structure):
$$\Hom_{\cat{C}^\deeq}(\Psi(X_n), \Psi \circ F(X_0)) \tens{\cO(G)} k_g \simeq \Hom_{\cat{C}^\deeq}(X_n, F^\deeq(X_0) \tens{\cO(G)} k_g) \simeq \Hom_{\cat{C}^\deeq}(X_n, F^\deeq_g(X_0)).$$
Verification that the identifications are compatible under $\Psi$ is a straightforward diagram chase.
\end{proof}

\medskip

\subsubsection{Chern character from $K$-theory to Hochschild homology}

Finally, we will use the universal trace map from connective $K$-theory to Hochschild homology constructed in \cite{BGT}.
\begin{defn}\label{K theory chern character}
For any small $k$-linear dg-category $\cat{C}$, the \emph{connective $K$-theory spectrum} $K(\cat{C})$ is the connective $K$-theory of the corresponding Waldhausen category defined in Section 5.2 of \cite{keller dgcat}.  Since Hochschild homology is a localizing invariant, by Theorems 1.1 and 1.3 of \cite{BGT} it receives a canonical and functorial map from the connective $K$-theory spectrum  which we call the \emph{Chern character}:\footnote{We use this terminology to avoid overloading the word ``trace.''} 
$$\chern: K(\cat{C}) \rightarrow \hh(\cat{C}).$$
\end{defn}

\begin{rmk}\label{chern properties}
We note two important properties of the Chern character that we use.  Note that unlike in the definition of Hochschild homology, in this discussion we restrict ourselves to small categories $\cat{C}$ (i.e. the compact objects of a compactly generated cocomplete category).
\begin{enumerate}
\item Via functoriality of the Chern character, for any object $X \in \mathrm{Ob}(\cat{C})$, the Chern character sends $[X] \in K_0(\cat{C}) \mapsto [X] \in \hh_0(\cat{C})$, i.e. equivalence classes in the Grothendieck group to their characters in Hochschild homology in the sense of Definition \ref{character def}.
\item Using the lax monoidal structure of $K$-theory, we see that for a monoidal category $\cat{C}$ the Chern character defines a map of algebras (see also Theorem 1.10 of \cite{BGT2}).
\end{enumerate}
Often in applications to geometric representation theory, we are only interested in (or able to compute) the Grothendieck group $K_0$.  However, note that the map $K_0(\cat{C}) \rightarrow \hh_0(\cat{C})$ does not automatically induce a map of algebras $K_0(\cat{C}) \rightarrow \hh(\cat{C})$ at the chain level.  In order to compare $K_0$ with Hochschild homology, we require certain vanishing conditions to hold.  Namely, if $\hh(\cat{C})$ is concentrated in degrees $\geq  0$, then the Chern character canonically factors through the truncation of $K(\cat{C})$ to degrees $\geq 0$, i.e. $K_0(\cat{C})$ since $K(\cat{C})$ is connective:
$$\begin{tikzcd}
K(\cat{C}) \arrow[rr, "\chern"] \arrow[dr] & & \hh(\cat{C}) \\
& K_0(\cat{C}) \arrow[ur, dotted] & 
\end{tikzcd}$$ 
and we may ask whether this map is an equivalence.  In particular, given this vanishing, when $\cat{C}$ is a monoidal category the induced map from $K_0(\cat{C}) \rightarrow \hh(\cat{C})$ is automatically a map of dg algebras at the chain level.
\end{rmk}


%

\medskip

\subsubsection{Equivariant $\ell$-adic sheaves, weights, and Tate type}\label{l-adic}

In this subsection we review some standard notions concerning weights and the $\ell$-adic cohomology of $BG$.
In this section and the following one, we fix a prime power $q = p^r$ and a prime $\ell \ne p$, and will work with $\ell$-adic sheaves $\cF$ on $\Fqbar$-schemes $X$.
All schemes and sheaves on them that arise are defined over $\Fq$, i.e., $X$ will come with a geometric Frobenius automorphism $\Fr$ and $\cF$ with a $\Fr$-equivariant (Weil) structure, which will be left implicit.

Fix a square root of $q$ in $\bQl$, thereby defining a notion of half Tate twist (this choice can be avoided by judicious use of extended groups as in~\cite{BG, xinwen Satake, bernstein sign}). For $\mathcal{F} \in \Sh(X)$ where $X$ is over $\Fq$, we will denote the Tate twist by $\mathcal{F}(n/2)$ for $n \in \Z$.  For a scheme $X$ with an action by a smooth group scheme $G$, we denote by $\Sh(X/G) = \Sh^G(X)$ the bounded derived category of $G$-equivariant $\bQl$-sheaves on $X$ with constructible cohomology (see Section 1.3 of \cite{BY} and \cite{BL}).  In this context, the cohomology of a sheaf $H^\bullet(X, -)$ will be understood to mean \'{e}tale cohomology.

Following the Appendix of \cite{gaitsgory central}, this notion can be extended to $G$-equivariant ind-schemes (i.e. a functor which is representable by a directed colimit of schemes with transition maps closed embeddings), where $G$ is a pro-affine algebraic group (i.e. an inverse limit of finite-type affine algebraic groups in the category of schemes) acting in a sufficiently finite way. We say a $G$-action on $X$ is \emph{nice} if the following two properties hold: (1) every closed subscheme $Z \subset X$ is contained in a closed $G$-stable subscheme $Z' \subset X$ such that the action of $G$ on $Z'$ factors through a quotient of $G$ which is affine algebraic, and (2) $G$ contains a pro-unipotent subgroup of finite codimension, i.e. if $G = \displaystyle\lim_{\longleftarrow} G_n$, then there is an $n$ such that $\ker(G \rightarrow G_n)$ is a projective limit of unipotent affine algebraic groups. If $G$ is a pro-affine group scheme acting nicely on $X$, and $X = \displaystyle\colim_{\longrightarrow} X_i$ with affine quotient $G_i$ acting on $X_i$, then we define\footnote{This definition is independent of the choice of presentation, since by \cite{BL} Theorem 3.4.1(ii) if $G_i \rightarrow G_j$ is a surjection with unipotent kernel, then $\Sh^{G_j}(Y) \rightarrow \Sh^{G_i}(Y)$ is an equivalence for any $Y$ on which $G_j$ acts.  See also Section A.4 of \cite{gaitsgory central}.} $\Sh^G(X) = \displaystyle\colim_{\longrightarrow} \Sh^{G_i}(X_i)$.

\medskip

We recall the well-known calculation of the $\ell$-adic cohomology ring of $BG$, whose description we repeat for convenience following~\cite{Vi} (in the Hodge-theory context).
\begin{prop}\label{coh BG}
Let $G$ be a pro-affine group scheme with split reductive quotient over $\overline{\bF}_q$. Then, $H^\bullet(BG; \bQl)$ is polynomial, generated in even degrees, and $H^{2k}(BG; \bQl)$ is a Frobenius eigenspace with eigenvalue $q^k$.   Furthermore the dg algebra $C^\bullet(BG; \bQl)$ under the cup product is formal, i.e. there is an algebra quasi-isomorphism $C^\bullet(BG; \bQl) \simeq H^\bullet(BG; \bQl)$.
\end{prop}
\begin{proof}
First, since $G$ is pro-affine, there is a reductive (finite type) algebraic group $G_0$ such that the kernel $\ker(G \rightarrow G_0)$ is pro-unipotent.  By Theorem 3.4.1(ii) in \cite{BL} we may assume that $G$ is reductive (and finite type).  

We first establish the claim that $H^\bullet(BG; \bQl)$ is polynomial in even degrees and compute the action of Frobenius.  It is a standard calculation that $H^\bullet(\G_m, \bQl) = H^0(\G_m, \bQl) \oplus H^1(\G_m, \bQl)$, where $H^0$ is a $1$-eigenspace for the Frobenius and $H^1$ is a $q$-eigenspace.  By Corollary 10.4 of \cite{LO}, it follows that $H^\bullet(B\G_m, \bQl) \simeq \bQl[u]$ where $u$ has cohomological degree $|u| = 2$, and is a $q$-eigenvector for Frobenius.  In particular, by the K\"{u}nneth formula (Theorem 11.4 in \emph{op. cit.}) we have that for a split torus $T$, $H^\bullet(BT; \bQl)$ is polynomial in even degrees, and $H^{2k}$ is a Frobenius-eigenspace with eigenvalue $q^k$.  Thus, the claim is true when $G = T$ is a torus.  Now, assume $T$ is a split torus inside a reductive group $G$, and $B$ is a Borel subgroup with $T \subset B \subset G$.  Applying Theorem 3.4.1(ii) of \cite{BL} again, we have  $H^\bullet(BB; \bQl) \simeq H^\bullet(BT; \bQl)$.  By Theorem 1.1 of \cite{Virk motives}, $H^\bullet(BG; \bQl)$ is a polynomial subring of $H^\bullet(BB; \bQl) \simeq H^\bullet(BT; \bQl)$, completing the claim.  Formality follows by a standard weight-degree argument.
\end{proof}

\medskip

\subsection{Automorphic and spectral realizations of the affine Hecke category}\label{two hecke}

We follow the set-up of Bezrukavnikov in~\cite{roma hecke}, except that we view the group on the automorphic side as dual to a chosen group on the spectral side for ease of notation.  Let $G$ be a fixed reductive algebraic group over $\bQl$ on the spectral side of Langlands duality, and let $G^\vee$ be the extension of scalars to $\Fqbar$ of its dual group split form over $\mathbb{F}_q$ (equipped with corresponding Frobenius automorphism).

Let $F = \Fqbar((t))$ and $O = \Fqbar[[t]]$.   We denote by ${\bfG}$ the loop group, i.e. the group ind-scheme over $\Fqbar$ with ${\bfG}(\Fqbar) = G^\vee(F)$ defined in Section 0.2 of \cite{gaitsgory central}.  We denote by ${\bfG}_0$ the arc group, which is a pro-affine group scheme with $\bfG_0(\Fqbar) = G^\vee(O)$.  There is a group scheme homomorphism $\bfG_0 \rightarrow G^\vee$, and the \emph{Iwahori subgroup} of $\bfG$ is defined $\bfI := \bfG_0 \times_{G^\vee} B^\vee$, which inherits its structure as a closed subgroup and is therefore also a pro-affine group.  We let $\bfI^\circ := \bfG_0 \times_{G^\vee} U^\vee$ denote its pro-unipotent radical.

On the automorphic side, we are interested in equivariant $\bQl$-sheaves on the \emph{affine flag variety} $\Fl=\bfG/\bfI$, an ind-proper ind-scheme constructed in the Appendix of \cite{gaitsgory central}. It carries a left action of $\bfI$ whose orbits are of finite type and naturally indexed the affine Weyl group $W_a$ for the group $G^\vee$.  For $w \in W$, we denote by $\Fl^w$ the corresponding orbit.  Denote by $j_w: \Fl^w \hookrightarrow \Fl$ the inclusion of the corresponding $\bfI$-orbit.  Let $\ell: W_a \rightarrow \mathbb{Z}^{\geq 0}$ denote the length function on the affine Weyl group.  

On the spectral side, the stacks that appear are defined over $\bQl$.  Recall the derived Steinberg variety $\cZ = \wt{\cN} \times_{\mf{g}} \wt{\cN}$ and the classical non-reduced Steinberg variety $\cZ' = \wt{\mf{g}} \times_{\mf{g}} \wt{\cN}$ (see Section \ref{rep thy defs}).  Each of these (derived) schemes has a natural $G$-action, as well as a commuting $\G_m$-action which by our convention acts by scaling on the points of $\mf{g}$, $\wt{\cN}$, and $\wt{\mf{g}}$ by weight -1 (thus on linear functionals by weight 1).  Recall the notation $\GG = G \times \G_m$.

The following is  Theorem 1 of \cite{roma hecke}, while the Frobenius property of $\Phi$ appears as Proposition 53.
\begin{thm}[Bezrukavnikov]\label{roma equiv}
At the level of homotopy categories, there are equivalences of categories $\Phi$ and $\Phi'$ and a commutative diagram
$$\begin{tikzcd}
\Sh^{\bfI^\circ}(\Fl)  \arrow[r, "\Phi'", "\simeq"'] & \Coh(\cZ'/G) \\
\Sh^{\bfI}(\Fl) \arrow[r, "\Phi", "\simeq"'] \arrow[u,  "\pi^*"] & \Coh(\cZ/G) \arrow[u, "i_*"']
\end{tikzcd}$$
where $\pi: \bfI^\circ \bs \Fl \rightarrow \bfI \bs \Fl$ is the quotient map and $i: \cZ/G \hookrightarrow \cZ'/G$ is the inclusion. 
Moreover the functors admits the following natural structures: 
\begin{enumerate}
\item[$\bullet$] $\Phi$ is naturally an equivalence of {\em monoidal} categories, and
\item[$\bullet$] $\Phi$ and $\Phi'$ intertwine the action of Frobenius on $\Sh^{\bfI}(\Fl)$ (resp. $\Sh^{\bfI^\circ}(\Fl)$)  with the action of $q\in \Gm$ on $\Coh(\cZ/G)$ (resp. $\Coh(\cZ'/G)$).
\end{enumerate}
\end{thm}

\begin{rmk}
We note that while it is expected that the above equivalences lift to $\infty$-categorical enhancements, it is not currently written in the literature explicitly.  In our arguments (e.g. in Proposition \ref{Gm structure}) we do not need this stronger version; we use the equivalence to produce graded lifts of certain standard objects in the spectral affine Hecke category which may be done at the level of homotopy categories.
\end{rmk}

We point out certain distinguished sheaves in $\Sh^\bfI(\Fl)$ and $\Sh^{\bfI^\circ}(\Fl)$ (computed explicitly for $G = SL_2, PGL_2$ in Examples 2.2.3-5 in \cite{NY}).  
\begin{enumerate}[(a)]
\item Let $\lambda \in X_*(T^\vee) = X^*(T) \subset W_a$ be a character of the maximal torus of $G$, considered as an element of the affine Weyl group of the dual group.  The \emph{Wakimoto sheaves} $J_\lambda$ are defined as follows.  When $\lambda$ is dominant, we take $J_\lambda = j_{\lambda, *} \underline{\bQl}_{\Fl^\lambda}[\langle 2\rho, \lambda \rangle]$.  When $\lambda$ is antidominant, we take $J_\lambda = j_{\lambda,!} \underline{\bQl}_{\Fl^\lambda}[\langle 2\rho, -\lambda \rangle]$.  In general, writing $\lambda = \lambda_1 - \lambda_2$, we define $J_\lambda = J_{\lambda_1} * J_{-\lambda_2}$, which is independent of choices due to Corollary 1 in Section 3.2 of \cite{AB:affine flags}.

\item For any $w \in W_a$, we define the corresponding \emph{costandard} (resp. \emph{standard}) object by $\nabla_w := j_{w,*} \underline{\bQl}_{\Fl^w}[\ell(w)]$ (resp. $\Delta_w := j_{w,!} \underline{\bQl}_{\Fl^w}[\ell(w)]$).  They are monoidal inverses by Lemma 8 in Section 3.2 of \cite{AB:affine flags}.  By Lemma 4 of \cite{roma hecke}, we have $\nabla_w * \nabla_{w'} = \nabla_{ww'}$ (and likewise for standard objects) when $\ell(w) + \ell(w') = \ell(ww')$.  If $\lambda \in X_*(T^\vee) = X^*(T)$ is dominant, then the Wakimoto is costandard $J_\lambda = \nabla_{\lambda}$; if $\lambda$ is antidominant, the Wakimoto is standard $J_\lambda = \Delta_{\lambda}$.

\item Let $w_0 \in W_f \subset W_a$ be the longest element of the finite Weyl group.  The \emph{antispherical projector} or \emph{big tilting sheaf} $\Xi \in \Sh^{\bfI^\circ}(\Fl)$ is defined to be the tilting extension of the constant sheaf $\underline{\bQl}_{\Fl^{w_0}}$ of $\Fl^{w_0}$ to $\Fl$, as in Proposition 11 and Section 5 of \cite{roma hecke}.  Note that this object does not descend to $\Sh^{\bfI}(\Fl)$.
\end{enumerate}
We abusively use the same notation to denote sheaves in $\Sh^{\bfI^\circ}(\Fl)$; note that $\pi^* \Delta_w \simeq \Delta_w$ and $\pi^* \nabla_w \simeq \nabla_w$ by base change.  All sheaves above are perverse sheaves, since the inclusion of strata are affine.

\medskip

For our applications, we need to work not with $\cZ/G$ but with $\cZ/\GG$ (recall that $\GG = G \times \G_m$).
The following proposition is the key technical argument we need to construct the semiorthogonal decomposition of $\Coh(\cZ/\GG)$ and hence deduce results on its homological invariants -- a graded lift of standards and costandards under Bezrukavnikov's theorem. It is conjectured in \cite{roma hecke} (and announced in \cite{HL}) that the equivalences in Theorem~\ref{roma equiv} should have mixed versions, relating a mixed form of the Iwahori-equivariant category of $\Fl$ with a $\Gm$-equivariant version of $\Coh(\cZ/G)$, i.e. $\Coh(\cZ/\GG)$, which would immediately give us the desired result. In particular, see Example 57 in \cite{roma hecke} for an expectation of what the sheaves $\Phi(\Delta_w)$ are explicitly and note that they have $\G_m$-equivariant lifts.

\begin{prop}\label{Gm structure}
The objects $\Phi(\nabla_w), \Phi(\Delta_w) \in \Coh(\cZ/G)$ have lifts to objects in $\Coh(\cZ/\GG)$ for all $w \in W_a$, compatible with the action of Frobenius under the equivalence in Theorem \ref{roma equiv}.
\end{prop}
\begin{proof}
We will prove the statements for the standard objects; the statements for costandards follows similarly.  Wakimoto sheaves are sent to twists of the diagonal $\Phi(J_\lambda) \simeq \OO_{\Delta}(\lambda)$ by Section 4.1.1 of \cite{roma hecke}, which evidently have $\G_m$-equivariant lifts.   Convolution is evidently $\G_m$-equivariant, so the convolution of two sheaves with $\G_m$-lifts also has a $\G_m$-lift.  Assuming that the standard objects corresponding to finite reflections have $\G_m$-lifts,  
 by Lemma 4 of \cite{roma hecke} we can write the standard for the affine reflection as a convolution of Wakimoto sheaves and standard objects for finite reflections.  Thus, we have reduced to showing that all standard objects $\Phi(\Delta_w)$ have $\G_m$-lifts for $w$ a simple finite reflection.

By Corollary 42 of \cite{roma hecke} $\Phi'$ has the favorable property that $\cZ'$ is a classical (non-reduced) scheme, and that it restricts to a map on abelian categories on $\Perv^{U^\vee}(G^\vee/B^\vee) \subset \Perv^{\bfI^\circ}(\Fl)$ taking values in $\Coh(\cZ'/G)^\heartsuit$ (though it is not essentially surjective).  In particular, by Proposition 26 and Lemma 28 in \cite{roma hecke} it takes the tilting sheaf $\Xi$ to $\OO_{\cZ'/G}$, which manifestly has a $\G_m$-lift.

We claim that $\G_m$-lifts for the $\Phi'(\Delta_w) \in \Coh(\cZ'/G)$ for $w \in W_f$ induce $\G_m$-lifts for the $\Phi(\Delta_w) \in \Coh(\cZ/G)$.  Since $\cZ$ is a derived scheme, the functor $i_*: \Coh(\cZ/\GG) \rightarrow \Coh(\cZ'/\GG)$ is not fully faithful (i.e. objects on the left may have additional structure).  But since $\Phi'(\Delta_w) \simeq i_* \Phi(\Delta_w)$ are in the heart and $i_*$ is $t$-exact (for the standard $t$-structures) and conservative, we have that $\Phi(\Delta_w) \in \Coh(\cZ/G)^\heartsuit$.  Moreover, the restriction of $i_*$ to $\Coh(\cZ/G)^\heartsuit$ is fully faithful, 
proving the claim.  Thus, we have reduced to showing that the finite simple standard objects $\Phi'(\Delta_w) \in \Coh(\cZ'/G)^\heartsuit$ have $\G_m$-lifts; in particular these are objects in the abelian category of coherent sheaves.

By Lemma 4.4.11 in \cite{BY}, $\Xi$ is a successive extension of standard objects $\Delta_w(\ell(w)/2)$ for $w \in W_f$. Thus, there is a standard object $\Delta_w(\ell(w)/2)$ and a surjection $\Xi \twoheadrightarrow \Delta_w(\ell(w)/2)$, which is Frobenius-equivariant as it arises as a morphism in the mixed category.  This implies that the kernel $K = \ker(\Xi \twoheadrightarrow \Delta_w(\ell(w)/2))$ is a Frobenius-equivariant subobject of $K$.  On the spectral side, using Proposition 53 in \emph{op. cit.}, this means that $\Phi'(K) \subset \Phi'(\Xi) \simeq \cO_{/\cZ'/G}$ is a $q$-equivariant subobject with quotient $\Phi'(\Delta_w(\ell(w)/2))$.  We wish to show that the quotient has a $\G_m$-equivariant lift, which amounts to showing that $\Phi'(K)$ is a $\G_m$-equivariant subobject.

Since $\Phi(K)$ is already endowed with a $\G_m$-equivariant structure, $q$-equivariance for a subobject of a $\G_m$-equivariant object is property, not an additional structure.  We claim that for $q$ not a root of unity, any $q$-closed subsheaf of a $\G_m$-equivariant sheaf on a quotient stack must be $\G_m$-closed as well (i.e. the isomorphism defining the $\G_m$-equivariant structure restricts to the subsheaf).  Assuming this claim, and iterating the above argument replacing $\Xi$ with the kernel $K$, we find that $\Phi'(\Delta_w)$ has a $\G_m$-equivariant lift for every $w \in W_f$ (since the big tilting object contains every $\Delta_w$ as a subquotient), completing the proof.

We now justify the claim.  First, if $\mathcal{F}$ is a sheaf on a quotient stack $X/G$ with a $\G_m$-action, we can forget the $G$-equivariance (i.e. base change to the standard atlas $X \rightarrow X/G$).  Now, by reducing to an open affine $\G_m$-closed cover of $X$, we can assume $X$ is affine.  On an affine scheme $X = \Spec(A)$, the $\G_m$-action gives the structure of a $\Z$-grading on $A$, and a submodule of a graded $A$-module $M' \subset M$ is $q$-equivariant if it is a sum of $q$-eigenspaces, and $\G_m$-equivariant if it is a sum of homogeneous submodules.  The claim follows from the observation that any $m \in M'$ can only have eigenvalues $q^n$ for $n \in \Z$, which are distinct, so the $q$-eigenspaces entirely determine the $\G_m$-weights.
\end{proof}

\medskip

\subsection{A semiorthogonal decomposition}\label{semiorthog}

In this section, we describe an ``Iwahori-Matsumoto'' semiorthogonal decomposition of the mixed affine Hecke category $\Hcat^\mix := \Coh(\cZ/\GG)$, arising from the stratification of the affine flag variety $\Fl$ on the automorphic side of Bezrukavnikov's equivalence Theorem~\ref{roma equiv} and the lifting result in Proposition \ref{Gm structure}.  This will, in turn, induce a direct sum decomposition on Hochschild homology.  First, let us establish terminology.
\begin{defn}
Let $\{ \cat{S}_n \}_{n\in \mathbb{N}}$ denote a collection of full subcategories of a small dg category $\cat{C}$.  We say that $\{ \cat{S}_n \}$ defines a \emph{semiorthogonal decomposition} of $\cat{C}$ if there is an exhaustive left admissible filtration $F_n\cat{C}$ of $\cat{C}$ such that $\cat{S}_n$ is the left orthogonal of $F_{n-1}\cat{C}$ inside $F_n\cat{C}$.  In particular, in this case $\Hom^\bullet_{\cat{C}}(X_n, X_m) \simeq 0$ for $X_i \in \cat{S}_i$ and $n > m$. 
\end{defn}

The following result is standard.
\begin{prop}\label{semiorthog strat}
Let $G$ be a pro-affine group scheme acting nicely on an ind-scheme $X$.  Assume that the stabilizer of each orbit is connected, and that every $G$-closed subscheme of $X$ is a union of finitely many $G$-orbits.  Let $I$ be an indexing set for the $G$-orbits $X_i$ under the (partial) closure relation, i.e. $X_n \subset \overline{X_m}$ implies $m \geq n$, and let $j_n: X_n \hookrightarrow X$ denote the inclusion.  Then, $\langle j_{n!} {\bQl}_{X_n}\rangle$ defines a semiorthogonal decomposition of $\Sh^G(X)$, where the ordering is given by any choice of extension of the partial order to a total order.
\end{prop}
\begin{proof}
It is standard that stratifications of stacks give rise to semi-orthogonal decompositions on categories of $\ell$-adic sheaves.  We note that each orbit is equivariantly equivalent $BH$ where $H$ is the stabilizer (connected by assumption), and $\Sh(BH)$ is generated by the constant sheaf $\bQl$ when $H$ is connected.
\end{proof}

\begin{cor}\label{semiorthog bruhat}
Fix a Bruhat ordering of the affine Weyl group $W_a$.  The standard objects $\langle \nabla_w  = j_{n!} {\bQl}_{X_n}\rangle$ give a semiorthogonal decomposition of $\Sh^{\mathbf{I}}(\Fl)$.
\end{cor}

\begin{rmk}
The costandard objects $\Delta_w = j_{n*} {\bQl}_{X_n}$ define a semiorthogonal decomposition in the reverse order.
\end{rmk}

We would like to lift the above semiorthogonal decomposition of $\Coh(\cZ/G)$ to $\Coh(\cZ/\GG)$. We do so by applying Lemma \ref{deeq} to the $\G_m$-equivariant lifts of the objects $\Phi(\Delta_w)$ from Proposition \ref{Gm structure}.  
We will apply the following result to the setting:
$$\cat{C} = \cat{H}^\mix = \Coh(\cZ/\GG), \;\;\;\;\;\;\; \cat{C}^\deeq = \cat{H} = \Coh(\cZ/G), \;\;\;\;\;\;\; H = \G_m = \Spec k[z,z^{-1}]$$
recalling the de-equivariantization functor $(-)^{\deeq}: \cat{C} \rightarrow \cat{C}^\deeq = \cat{C} \otimes_{\Rep(H)} \cat{Vect}_k$  from Section \ref{HHSec}.
\begin{cor}\label{semiorthog barr beck}
Let $H$ be a group-scheme over a field $k$ of characteristic 0, and $\cat{C}$ a compactly generated cocomplete $\Rep(H)$-module dg category.   Let $\{E_n \in \cat{C} \mid n \in \mathbb{N}\}$ be a linearly ordered set of objects such that $\langle E_n^\deeq \rangle$ defines a semiorthogonal decomposition of $\cat{C}^\deeq$.  Denote by $A_n = \underline{\End}_{\cat{C}}(E_n)^{op}$ the $\Rep(H)$-algebras from Lemma \ref{deeq}. Then, we have an equivalence
$$\hh(\cat{C}) \simeq \bigoplus_\alpha \hh(A_n\dmod_{\Rep(H)}).$$ 
\end{cor}
\begin{proof}
Let $\cat{C}^\deeq_n := \langle E_n^\deeq \rangle$ be the category generated by $E_n^\deeq$, and let $\cat{C}_n$ be the preimage under $(-)^{\deeq}$.  The categories $\cat{C}_n$ form a semiorthogonal decomposition of $\cat{C}$, since $\Hom_{\cat{C}}(X, Y) = \uHom_{\cat{C}}(X, Y)^G$ by Lemma \ref{deeq}, and since $\uHom_{\cat{C}}(X, Y) = \Hom_{\cat{C}^\deeq}(X^\deeq, Y^\deeq)$ after forgetting the $\Rep(G)$-enriched structure on the left.  Hochschild homology is a localizing invariant in the sense of \cite{BGT}, and in particular takes semiorthogonal decompositions to direct sums.  Thus we have an equivalence $\hh(\cat{C}) \simeq \displaystyle\bigoplus_{n \in \mathbb{Z}} \hh(\cat{C}_n).$  
Applying Lemma \ref{deeq}, we find
$\displaystyle \hh(\cat{C}) \simeq \bigoplus_{n \in \mathbb{Z}} \hh(A_n\dmod_{\Rep(\G_m)}).$
\end{proof}

We now compute the endomorphism algebras $A_w$ as algebras in $\Rep(\G_m)$, using the graded lifts from Proposition \ref{Gm structure} and the semiorthogonal decomposition in Corollary \ref{semiorthog bruhat}.
\begin{prop}\label{compute end}
Let $E_w$ denote the $\G_m$-lifts of $\Phi(\Delta_w)$ constructed in Proposition \ref{Gm structure}, and $A_w = \underline{\End}_{\cZ/\GG}(E_w^\deeq)$.  We have quasi-isomorphisms $A_w \simeq \Symp_{\bQl} \mf{h}[-2]$ where $\mathfrak{h}[-2]$ is the universal dual Cartan shifted into cohomological degree 2 with $\G_m$-weight 1.  In particular, $A_w$ is formal.
\end{prop}
\begin{proof}
Since $\Phi$ is an equivalence of categories we can compute $A_w$ on the automorphic side. The unit map $\mathcal{F} \rightarrow j^! j_! \mathcal{F}$ is an equivalence for $j$ a locally closed immersion, so that
$$A_w = \Hom(j_{w,!}\underline{\bQl}_{\Fl^w}, j_{w,!}\underline{\bQl}_{\Fl^w}) = \Hom(\underline{\bQl}_{\Fl^w}, j_w^! j_{w,!}\underline{\bQl}_{\Fl^w})  \simeq R\Gamma(\bfI \bs \Fl^w, \underline{\bQl}_{\Fl^w}).$$
Since $\Fl^w$ is an $\bfI$-orbit, letting $\bfI^w$ denote its stabilizer for a choice of base point in $\Fl^w$, we find that $A_w \simeq C^\bullet(B\bfI^w; \bQl)$ is the equivariant cohomology chain complex for $B\bfI^w$ with $\bQl$-coefficients under the cup product.  The reductive quotient (i.e. by the pro-unipotent radical) of $\bfI^w$ is the quotient torus $H^\vee$, so $A_w \simeq C^\bullet(BH^\vee; \bQl)$.  By Proposition \ref{coh BG}, this algebra is formal and isomorphic to $H^\bullet(BH^\vee; \bQl) \simeq \Symp_{\bQl} \mf{h}[-2]$.

For the $\G_m$-weight, recall that the pullback along multiplication by $q$ corresponds under $\Phi$ to the Frobenius automorphism.  Thus, for $q$ not a root of unity, the $q^n$-eigenspace and the homogeneous $\G_m$-weight $n$ part coincide, and the claim follows by Proposition \ref{coh BG}.
\end{proof}

We now apply Corollary \ref{semiorthog barr beck} to the set-up in the above proposition.
\begin{corollary}\label{actual theorem}
Letting $k = \bQl$ or $\C$, we have an isomorphism of $k[z,z^{-1}]$-modules
$$\hh(\Hcat^\mix) \simeq kW_a \otimes_{k} k[z,z^{-1}].$$
In particular, we have that
\begin{enumerate}
\item the Hochschild homology $\hh(\Hcat^\mix)$ is cohomologically concentrated in degree zero,
\item the Chern character $K(\Hcat^\mix) \rightarrow \hh(\Hcat^\mix)$ factors through $K_0(\Hcat^\mix)$, and
\item the map $K_0(\Hcat^\mix) \otimes_{\Z} k \rightarrow \hh(\Hcat^\mix)$ is an equivalence.
\end{enumerate}
\end{corollary}
\begin{proof}
The claim for $\mathbb{C}$ follows from $\bQl$ by fixing an isomorphism.  Fix a Bruhat order on $W_a$, extended to a total order.   Applying Corollary \ref{semiorthog barr beck} in the case $\cat{C} = \Hcat^\mix = \Coh(\cZ/\GG)$, $\cat{C} = \Hcat = \Coh(\cZ/G)$, and $H = \G_m$, we have a canonical equivalence
$$\hh(\Hcat^\mix/\bQl) \simeq \bQl W_a \otimes_{\bQl} \hh(A \dperf_{\Rep(\G_m)}/\bQl)$$
where $A = \Sym_{\bQl} \mf{h}[-2] \simeq A_w$ is the algebra from Proposition \ref{compute end} (which does not depend on $w \in W_a$).  The Hochschild homology of of $A\dperf_{\Rep(\G_m)}$ is computed by the Block-Getzler complex of Definition \ref{alg model}, which we can compute explicitly.  Its terms are $(A^{\otimes n+1} \otimes \bQl[z,z^{-1}])^{\G_m}$, and since $z$ has $\G_m$-weight 0, there is an isomorphism $(A^{\otimes n+1} \otimes \bQl[z,z^{-1}])^{\G_m} \simeq (A^{\otimes n+1})^{\G_m} \otimes \bQl[z,z^{-1}]$ and we observe that $(A^{\otimes n+1})^{\G_m} = \bQl$ since each $A$ is generated over $\bQl$ by positive weights.  Thus, the natural map $\BG^\bullet_{\G_m}(\bQl) \rightarrow \BG^\bullet_{\G_m}(A)$ is a quasi-isomorphism, so the first claim follows.  Factorization through $K_0$ follows since the Hochschild homology is coconnective.

To show that the map $K_0(A\dmod_{\Rep(\G_m)}) \otimes_{\Z} \bQl \rightarrow \hh(A\dmod_{\Rep(\G_m)}/\bQl)$ is an equivalence, first note that since $\hh(A\dmod_{\Rep(\G_m)}/\bQl)$ is concentrated in degree zero, the Chern character factors through $K_0$, i.e. we have a commuting diagram for each summand
$$\begin{tikzcd}
K(\Rep(\G_m)) \otimes_{\Z} \bQl \arrow[r] \arrow[d] & K_0(\Rep(\G_m)) \otimes_{\Z} \bQl \arrow[r, "\simeq"] \arrow[d] & \hh(\Rep(\G_m)/\bQl) \arrow[d, "\simeq"] \\
K(A\dperf_{\Rep(\G_m)}) \otimes_{\Z} \bQl \arrow[r] & K_0(A\dperf_{\Rep(\G_m)}) \otimes_{\Z} \bQl \arrow[r] & \hh(A\dmod_{\Rep(\G_m)}/\bQl).
\end{tikzcd}$$
By Remark \ref{chern properties}, the map $K_0(\Rep(\G_m)) \rightarrow K_0(A\dperf_{\Rep(\G_m)})$ is an equivalence, since both sides are freely generated by $K_0(\Rep(\G_m)) = \hh(\Rep(\G_m))$ by the character of a single object $[A]$, i.e. the free object.  Using the semiorthogonal decomposition, these equivalences induce an equivalence $K_0(\Hcat^\mix) \otimes_{\bZ} \bQl \simeq \hh(\Hcat^\mix/\bQl)$, which is an equivalence of algebras by Remark \ref{chern properties}.
\end{proof}

We also have the following result for the non-$\G_m$-equivariant version.
\begin{cor}\label{non Gm hh}
Let $k = \bQl$ or $\bC$.  The map of algebras $K(\Coh(\cZ/G)) \rightarrow \hh(\Coh(\cZ/G))$ factors through $K_0$ and we have an isomorphism as dg $k$-modules
$$\hh(\Coh(\cZ/G)) \simeq kW_a \otimes H^\bullet(H^\vee \times BH^\vee; k) \simeq kW_a \otimes \Sym_k(\mf{h}[-1] \oplus \mf{h}[-2]).$$
\end{cor}
\begin{proof}
Essentially the same as the previous corollary, along with a direct calculation of the Hochschild homology of the formal dg ring $\hh(\Sh(BH^\vee)) = \hh(\Sym_k(\mf{h}[-2])\dmod).$
\end{proof}

\medskip

\subsection{Hochschild homology of the affine Hecke category}\label{hh aff hecke}

In this section, we will show that the trace decategorification of the mixed affine Hecke category $\Hcat^\mix$ is the affine Hecke algebra $\Haff$, while the trace decategorification of the affine Hecke category $\Hcat$ is a derived variant of the group algebra of the extended affine Weyl group $kW_a$.  We assume that $G$ has simply connected derived subgroup  until Section \ref{not sc}, where we remove the assumption.

\medskip

We begin by quoting the following celebrated theorem by Ginzburg, Kazhdan and Lusztig.
\begin{thm}[Ginzburg-Kazhdan-Lusztig]\label{KL thm}
Let $k = \bQl$ or $\mathbb{C}$, and assume that $G$ has simply connected derived subgroup.  Then there is an equivalence of associative algebras $\Haff \rightarrow K_0(\Hcat^\mix) \otimes_{\Z} k$, compatibly with an identification of the center with $K_0(\Rep(\GG)) \otimes_{\Z} k$.  Likewise, there is an equivalence of associative algebras $kW_a \simeq K_0(\Hcat) \otimes_{\Z} k$ with center $K_0(\Rep(G))$.
\end{thm}
\begin{proof}
The only difference between our statement and that in \cite{KL} \cite{CG} is their Steinberg stack is the classical stack $\pi_0(\cZ)/\GG$, which has no derived structure.  On the other hand, we are interested in $\cZ/\GG$ which has better formal properties.  The statement follows from the fact that the Grothendieck group is insensitive to derived structure, i.e. the ideal sheaf for the embedding $\pi_0(\cZ)/\GG \hookrightarrow \cZ/\GG$ acts nilpotently on any coherent complex.  Finally, note that while the statement of Theorem 3.5 of \cite{KL} and Theorem 7.2.5 in \cite{CG} are made for $k = \C$, the proofs do not employ topological methods and apply to the isomorphic field $\bQl$.
\end{proof}

For the remainder of the section, we let $k = \bQl$ or $\C$.  We combine the above theorem with Corollary~\ref{actual theorem} to arrive at the following main theorem.  We will remove the simply connectedness assumption in Section \ref{not sc}.
\begin{thm}\label{ThmHH}
Assume that $G$ has simply connected derived subgroup.  There is an equivalence of algebras, and an identification of the center:
$$\begin{tikzcd}
\Haff \arrow[r, "\simeq"] & \hh(\Hcat^\mix)\\
k[G]^G \otimes_k k[q,q^{-1}] \arrow[r, "\simeq"] \arrow[u, hook] & \hh(\Rep(G \times \G_m)) \arrow[u, hook].
\end{tikzcd}$$
\end{thm}  
\begin{proof}
That the map is an isomorphism is a combination of Theorem \ref{KL theorem} and Corollary \ref{actual theorem}.\end{proof}

The following non-mixed variant may also be of interest, and is the analogue to Corollary \ref{non Gm hh}. In this case, the map to Hochschild homology is not an equivalence, though it does induce an equivalence on $\hh_0$.  
We note that the dg algebra $\Sym_k(\mf{h}[-1] \oplus \mf{h}[-2])$ appearing in the statement is equivalent to $C^\bullet(H^\vee \times BH^\vee)$.
\begin{cor}\label{ThmHH no Gm}
With the assumptions above, there is a commuting diagram of algebras:
$$\begin{tikzcd}
kW_a \otimes_k \Sym_k(\mf{h}[-1] \oplus \mf{h}[-2]) \arrow[r, "\simeq"] & \hh(\Hcat)\\
k[G]^G \arrow[r, "\simeq"] \arrow[u, hook] & \hh(\Rep(G)) \arrow[u, hook].
\end{tikzcd}$$
\end{cor}
\begin{proof}
By Corollary \ref{non Gm hh}, the Hochschild homology $\hh(\Coh(\cZ/G))$ is coconnective, so the Chern character from $K(\Coh(\cZ/G))$ factors through $K_0(\Coh(\cZ/G)) \otimes_{\Z} k = kW_a$.  Thus we have a map of algebras $kW_a \rightarrow \hh(\Coh(\cZ/G))$ which induces an equivalence on $H^0$.  Next, note that the subcategory $\Sh^{\bfI}(\Fl)$ generated by the monoidal unit (i.e. the skyscraper sheaf $\delta_e$), which is closed under the monoidal structure, is in the center of $\Coh(\cZ/G)$, so that the subalgebra $\hh(\langle \delta_e \rangle) \simeq  \Sym_k(\mf{h}[-1] \oplus \mf{h}[-2]) \subset \hh(\Coh(\cZ/G))$ is central.  This defines a map of algebras $\hh(\langle \delta_e \rangle)\dmod \rightarrow \hh(\Coh(\cZ/G))$, which defines a map of algebras out of the tensor product 
$\hh(\langle \delta_e \rangle) \otimes_k kW_a \rightarrow \hh(\Coh(\cZ/G))$ which is an equivalence when restricted to each tensor factor; thus it is an equivalence.  
\end{proof}

\medskip

\subsubsection{$q$-specializations of the affine Hecke algebra}

Let $q: \cZ/G \rightarrow \cZ/G$ be the action by $q \in \G_m$ under our conventions, i.e. multiplying by $q^{-1}$.  In this section we compute the trace of the functor\footnote{Note that our $q_*$ corresponds to $\mathbf{q}^*$ in \cite{AB:affine flags}.} $q_*$ on the category $\Hcat = \Coh(\cZ/G)$.  First, we make the general observation that if $F$ is an automorphism of a small dg category $\cat{C}$ and $\cE \in \cat{C}$, then an $F$-equivariant structure on $\cE$ induces an automorphism of the dg algebra $A = \End_{\cat{C}}(\cE)$, and thus an automorphism of the category $A\dperf$ which we denote $F_A$.  This $F$-equivariant structure on $\cE$ defines a commuting structure for an equivalence of pairs $(A\dperf, F_A) \rightarrow (\langle \cE \rangle, F)$.

\begin{prop}\label{compute trace q}
Let $q \ne 1$ and let $A_w$ denote the algebras from Proposition \ref{compute end}.  Then, $\hh(A_w, q_*) \simeq k.$ 
\end{prop}
\begin{proof}
First, observe that the functor $q_*$ induces the automorphism on the algebra $A_w \simeq \Symp_k \mf{h}^*[-2]$ arising via the $q$-scaling map on $\mf{h}$ (in particular, $\mf{h}^*$ has weight $-1$).  The claim is a direct calculation using the complex $C_q(A_w, \G_m)$ from Definition \ref{alg model} via Koszul resolutions: $C_q(A_w, \G_m)$ is the derived tensor product $A_w \otimes_{A_w \otimes A_w}^L A_w$ where $A_w$ is the diagonal bimodule for one factor and is twisted by $q_*$ on the other factor.

Rather than a direct calculation, we give a geometric argument.  First, note that $q_*$ preserves the $\G_m$-weights of $A_w \simeq \Sym_k \mf{h}^*[-2]$ (i.e. since $q \in \G_m$ is central).  We apply a Tate shearing (i.e. sending cohomological-weight bidegree $(a, b)$ to $(a - 2b, b)$) to the algebra $\Symp_k \mf{h}^*[-2]$ to obtain the algebra $\OO(\mf{h}) = \Sym_k \mf{h}^*$.  Note that $\hh(\Perf(\h), q_*) = \OO(\h^q)$, i.e. functions on the derived fixed points of action by $q$.  When $q \ne 1$ we have $\h^q = \{0\}$, so $\hh(\Perf(\h), q_*) = k$.  Undoing the shearing, we find that the natural map $\hh(A_w, q_*) \rightarrow \hh(k, q_*)$ is an equivalence.
\end{proof}

\begin{cor}\label{trace q id}
Let $\cH_q$ denote the specialization of the affine Hecke algebra at $q \in \G_m$.  If $q \ne 1$, we have an equivalence of algebras
$$\hh(\Hcat, q_*) \simeq \cH_q.$$
\end{cor}
\begin{proof}
The calculation in Proposition \ref{compute trace q} shows that specialization at $q \in \G_m$ induces an equivalence on Block-Getzler complexes (viewing $A_w$ as an algebra in $\Rep(\G_m)$):
$$\BG^\bullet_{\G_m}(A_w) \otimes_{k[z,z^{-1}]} k_q \rightarrow \underline{\BG}_{\G_m}^\bullet(A_w) \otimes_{k[z,z^{-1}]} k_q \rightarrow \BG_{\G_m,q}^\bullet(A_w)$$  
inducing an equivalence $\hh(\Coh(\cZ/\GG)) \otimes_{k[z,z^{-1}]} k_q \simeq \hh(\Coh(\cZ/G), q_*)$, since the trace of an endofunctor $F$ on a category $\cat{C}$ takes semiorthogonal decompositions preserved by $F$ to direct sums.  Consequently, under the identification of algebras $\hh(\Coh(\cZ/\GG)) \simeq \Haff$, specialization at $q$ defines an equivalence $\hh(\Coh(\cZ/G), q_*) \simeq \Haff_q$.
\end{proof}

\begin{rmk}
The above corollary is evidently untrue for $q=1$, since $\cH$ is flat over $k[z,z^{-1}]$ but $\hh(\Hcat)$ has derived structure by Corollary \ref{ThmHH no Gm}.
\end{rmk}

\begin{rmk}\label{hecke variants}
Our methods also allow for an identification of the following monodromic variants of the affine Hecke category introduced in \cite{roma hecke} (where $\cZ' = \wt{\mf{g}} \times_{\mf{g}} \wt{\cN}$ and $\cZ^\wedge$ is the formal completion of $\wt{\mf{g}} \times_{\mf{g}} \wt{\mf{g}}$ along $\cZ$):
$$ \hh(\Coh(\cZ'/\GG)) \simeq \hh(\Coh(\cZ^\wedge\!/\GG)) \simeq \Haff ,$$
$$\hh(\Coh(\cZ'/G), q_*) \simeq \Haff_q, $$
$$\hh(\Coh(\cZ^\wedge\!/G), q_*) \simeq \begin{cases} kW_a \otimes_k \Sym_k(\mf{h} \oplus \mf{h}[-1]) & q=1, \\ \Haff_q & q \ne 1. \end{cases}$$
The category $\Coh(\cZ'/\GG)$ is not monoidal, so it does not make sense to ask that it is identified with $\cH$ as an algebra.  However, it is equivalent to $\cH$ as a (right) module for $\hh(\Coh(\cZ/\GG)) \simeq \cH$.  The category $\Coh(\cZ^\wedge\!/\GG)$ does not have a monoidal unit, and its monoid structure is trivial; in \cite{CD} an enlargement of $\Coh(\cZ^\wedge\!/\GG)$ will be defined to resolve these issues (see also \cite{BY}) but we will not address it here.

In these cases the generating object $E_w = \underline{\bQl}_{\Fl^w}$ for each stratum on the automorphic side live in different categories, resulting in different endomorphism algebras (see Proposition \ref{compute end}).  Recall that for $\cZ$, the category appearing is $D(BH^\vee)$ so we had $A_w = C^\bullet(BH^\vee; \bQl) \simeq \Sym_{\bQl} \mf{h}[-2]$.  For $\cZ'$, the category is $D(\pt)$, so $A'_w = \bQl$.  For $\cZ^\wedge\!$, the category is  $D^u(H^\vee \brslash H^\vee) \subset D(H^\vee)$, the full subcategory of sheaves with unipotent monodromy, and $A^\wedge_w \simeq C^\bullet(H^\vee; \bQl) \simeq \Sym_{\bQl} \mf{h}[-1]$.

For the enlargement of $\Coh(\cZ^\wedge\!)$ from \cite{CD}, the constant sheaf does not generate on each stratum, and instead one should take a ``cofree monodromic'' sheaf (defined in \emph{op. cit.}) whose endomorphisms $\wh{\Symp}^\bullet_{\bQl} \mf{h}^*$ are Koszul dual to $A^\wedge_w = \Sym \mf{h}[-1]$.  Likewise, for $\Coh(\cZ)$ the constant sheaf is not compact, and rather than coherent sheaves one could have considered the smaller category of compact sheaves.  The generator then is the induced sheaf, which has endomorphisms $C_\bullet(H^\vee;\bQl) \simeq \Sym_{\bQl} \mf{h}^*[1]$, which is Koszul dual to $A_w = \Sym_{\bQl} \mf{h}[-2]$.
\end{rmk}

\medskip

\subsubsection{Groups of non-simply connected type}\label{not sc}

In this section we will remove the simply connectedness assumptions from earlier theorems.  We work in the following set-up.  Let $G$ be a reductive algebraic group with simply connected derived subgroup, and $\phi: G \rightarrow G'$ a central isogeny with kernel $Z$ (i.e. a quotient by a finite subgroup $Z$ of the center).  Following Section 1.5 of \cite{reeder}, this induces a $Z$-action on $\cH_G$ via the formula
\begin{equation}\label{reeder formula}
z \cdot (T_w \otimes e^\lambda) = \lambda(z)(T_w \otimes e^\lambda), \;\;\;\;\;\;\; w \in W_f, \lambda \in X^*(T), z \in Z.
\end{equation}
Equivalently, the affine Hecke algebra has a multiplicative grading by characters of $Z$, i.e. 
$$\cH_G = \bigoplus_{\chi \in X^\bullet(Z)} \cH_{G, \chi}$$
and we have an identification of $\cH_{G'}$ with the trivial graded part or $Z$-invariants
$$\cH_{G'} \simeq \cH_G^Z = \cH_{G, \mathrm{triv}} \hookrightarrow \cH_G.$$

Our goal will be to prove a similar formula in Hochschild homology, which arises when the category is equipped with a $Z$-trivialization in the following sense.
\begin{defn}
Let $G$ be an affine algebraic group with central subgroup $Z \subset G$, and $\cat{C}$ be a $\Rep(G)$-module category.  A \emph{$Z$-trivialization} of $\cat{C}$ is a $\Rep(G/Z)$-linear category $\cat{C}'$ and an equivalence $\cat{C} \simeq \cat{C}' \otimes_{\Rep(G/Z)} \Rep(G)$.
\end{defn}

\begin{rmk}\label{Z inf HH def}
If $G$ is reductive (thus $Z$ is semisimple), then we have a decomposition of $\Rep(G)$ into $\Rep(G/Z)$-module categories by $Z$-characters.  Via the $Z$-trivialization, this gives a decomposition of $\cat{C}$ into $\Rep(G)$-module categories
$$\Rep(G) = \bigoplus_{\chi \in X^\bullet(Z)} \Rep(G)_\chi, \;\;\;\;\;\;\;\;\;\; \cat{C} \simeq \bigoplus_{\chi \in X^\bullet(Z)} \cat{C}_\chi$$
where the natural functor $\cat{C}' \rightarrow \cat{C}$ induces an equivalence $\cat{C}' \simeq \cat{C}_{\mathrm{triv}}$ with the trivial block.    In this setting, the direct sum decomposition of $\cat{C}$ induces a $X^\bullet(Z)$-grading in Hochschild homology
$$HH(\cat{C}) = \bigoplus_{\chi \in X^\bullet(Z)} HH(\cat{C}_\chi)$$ 
such that $HH(\cat{C}_{\mathrm{triv}}) \simeq HH(\cat{C}').$   Since the sum decomposition is evidently functorial for $\Rep(G)$-functors compatible with trivializations, so is the grading on Hochschild homology.
\end{rmk}

\medskip

It remains to show that these $X^\bullet(Z)$-gradings agree via the identifications in Theorem \ref{ThmHH}.
\begin{prop}\label{Z action id}
The identification $\cH \simeq \hh(\Hcat^\mix)$ of Theorem \ref{ThmHH} are compatible with the $X^\bullet(Z)$-gradings defined in Equation \ref{reeder formula} and Remark \ref{Z inf HH def}.
\end{prop}
\begin{proof}
We claim that the $Z$-action on $\Haff$ defined in \cite{reeder} induces a decomposition of $\cH$ into eigenspaces indexed by $W_f$ double cosets $W_f \lambda W_f \subset W_a$ for $\lambda \in X^\bullet(T)$, spanned by Iwahori-Matsumoto basis elements $T_w$ for $w \in W_f \lambda W_f$, with eigenvalue $\lambda|_Z$.  This claim can be directly verified, e.g. using the Bernstein relations in Section 7.1 of \cite{CG}.  This $X^\bullet(Z)$-eigenbasis of $\cH$ corresponds under Theorem \ref{ThmHH} to the basis $\{[\mathrm{id}_{\Phi(\Delta_w)}]  \mid w \in W_a\} \subset HH(\Hcat^\mix)$, i.e. identity maps for the spectral-side standard objects $\Phi(\Delta_w)$ described in Section \ref{two hecke}, which we need to verify is an eigenbasis with corresponding eigenvalues.

By functoriality, for any functor $F: \cat{C} \rightarrow \cat{D}$ of categories in our set-up, if $[\mathrm{id}_X] \in HH_0(\cat{C})$ is a $\lambda$-eigenvector for $Z$, then $[\mathrm{id}_{F(X)}] \in HH_0(\cat{D})$ is as well; the converse is true if $F$ is faithful on the homotopy category (i.e. $H^0(\Hom^\bullet(X, X)) \rightarrow H^0(\Hom^\bullet(F(X),F(X))$ is injective).  We will use this fact repeatedly.  In particular, since the forgetful functor $\Coh(\cZ/\GG) \rightarrow \Coh(\cZ/G)$ is faithful, we can forget $\G_m$-equivariance, and since the $Z$-action is compatible with convolution, it suffices to check our statement for finite reflections and the lattice.

For the lattice, we have $\Phi(\Delta_\lambda) \simeq \Delta_* \cO_{\wt{\cN}}(\lambda) = \Delta_* p^* V_\lambda \in \Coh(\cZ/G)$, where $p: \wt{\cN}/G \rightarrow BB$ is the projection.  The eigenvalue for the identity map of $V_\lambda \in \Coh(BB)$ is evidently $\lambda|_Z$.  For finite simple reflections, since $i_*$ is fully faithful on the homotopy category we may instead consider the equivalence $\Phi'$.  Here, the spectral-side object corresponding to the automorphic big tilting object is $\cO_{\cZ'/G}$.  By applying functoriality to the pullback from a point we see that the identity on any structure sheaf has trivial $Z$-eigenvalue, and therefore any subquotient does, thus $\Phi'(\Delta_w)$ and $\Phi(\Delta_w)$ do.
\end{proof}

\begin{cor}\label{no sc}\label{thm no sc}
The statements of Theorem \ref{ThmHH}, Corollary \ref{ThmHH no Gm} and Corollary \ref{trace q id} hold without the assumption that $G$ has simply connected derived subgroup.
\end{cor}
\begin{proof}
By Theorem \ref{ThmHH}, we have an identification $\hh(\Hcat^\mix_G) \simeq \cH_G$.  Since the center $Z$ acts on $\cZ$ and $\cZ'$ trivially, the categories $\Coh(\cZ/G)$ and $\Coh(\cZ'/G)$ come equipped with natural $Z$-trivializations, and thus their Hochschild homologies have $X^\bullet(Z)$-gradings.  By Proposition \ref{Z action id} the two gradings coincide under our equivalence, proving the claim.
\end{proof}

\section{Traces of representations of convolution categories}\label{traces}

We have seen in Theorem~\ref{ThmHH} that the affine Hecke algebra $\Haff$ is identified with the Hochschild homology of the (mixed) affine Hecke category $\Hcat^\mix=\Coh(\cZ/\GG)$. In this section we describe a general theory of categorical traces in derived algebraic geometry to explain why this is a useful realization. Namely, as an application we will see in Section \ref{section coherent springer} that the geometric realization of Hochschild homology via derived loop spaces implies a realization of the affine Hecke algebra as endomorphisms of the {\em coherent Springer sheaf}, a certain coherent sheaf on the loop space of the stacky nilpotent cone.  Hence, we deduce a localization description of the category of modules for the affine Hecke algebra as the category of coherent sheaves generated by the coherent Springer sheaf.

\medskip

\subsection{Traces of monoidal categories}\label{traces background}

In this section we present the two different trace decategorifications for a monoidal category and their relation.  See \cite{BFN, HSS, CP, BN:NT, GKRV}  for detailed exposition.
\begin{defn}\label{two trace definition}
Let $(\cat{A},\ast)$ denote an $E_1$-monoidal compactly generated cocomplete $k$-linear dg category and $F$ a monoidal endofunctor.  There are two notions of its Hochschild homology or trace.  See definitions in Section \ref{HHSec}.
\begin{enumerate}
\item  The naive or \emph{vertical trace} (or \emph{Hochschild homology}) is a chain complex $\mathrm{tr}(\cat{A}, F)=\hh(\cat{A}, F)$. Via functoriality of traces, and under the assumptions that the multiplication functor $\ast: \cat{A} \otimes \cat{A} \rightarrow \cat{A}$ preserves compact objects and that the monoidal unit is compact, it has the additional structure of an associative (or $E_1$-)algebra $(\hh(\cat{A}),\ast)$. 
\item The 2-categorical or \emph{horizontal trace} (or \emph{monoidal/categorical Hochschild homology}) is a dg category\footnote{The category $\cat{A}^{rv}$ is obtained by reversing the monoidal product, not taking opposite morphisms.} $\Tr((\cat{A}, \ast), F) = \cat{A} \otimes_{\cat{A} \otimes \cat{A}^{rv}} \cat{A}_F$ where $\cat{A}_F$ is the ($E_1$-)monoidal category whose left action is twisted by $F$.\footnote{More generally, the horizontal trace may take as an input an $\cat{A}$-bimodule category $\cat{Q}$; we will not need this.}  Via functoriality of traces, the horizontal trace is the tautological receptacle for characters in $\cat{A}$:
$$[-]: \cat{A} \rightarrow \Tr((\cat{A}, \ast), F).$$
The monoidal unit $1_{\cat{A}}$ itself defines an object $[1_{\cat{A}}]\in \cat{Tr}((\cat{A},\ast),F)$, i.e. $\cat{Tr}((\cat{A},\ast), F)$ is a pointed (or $E_0$-)category.  
\end{enumerate}
We sometimes omit the monoidal product $\ast$ from the notation, and when $F = \mathrm{id}_{\cat{C}}$ we also sometimes omit it from the notation.
\end{defn}

\medskip

We define the notion of characters in horizontal traces more precisely and generally below.  These more general notions are used primarily in Section \ref{standard rep}.
\begin{defn}\label{two trace maps}
One can view the horizontal trace as a trace decategorification in the sense of Definition \ref{def trace decat} in the following way, following Section 3.6 of \cite{GKRV}.  We consider the symmetric monoidal ``Morita'' category $\cat{Mor}_k$, whose objects are the $(\infty, 2)$-categories $\cat{A}\mh\cat{mod}$, i.e. left-module categories for a monoidal category $\cat{A}$, and whose 1-morphisms 
$$\cat{Map}_{\cat{Mor}_k}(\cat{A}\mh\cat{mod}, \cat{B}\mh\cat{mod}) := \cat{B} \otimes \cat{A}^{rv}\mh\cat{mod}$$
are $(\cat{B}, \cat{A})$-bimodule categories, and 2-morphisms are functors of bimodule categories.\footnote{The arguments in \cite{GKRV} do not require the use of non-invertible 3-morphisms in $\cat{Mor}_k$.}  Then, for a monoidal endofunctor $F: \cat{A} \rightarrow \cat{A}$, we have $\mathrm{tr}(\cat{A}\mh\cat{mod}, F) = \Tr(\cat{A}, \cat{A}_F)$.

We can apply Definition \ref{character def} to obtain the following more general notion of character map for the horizontal trace (see Section 3.8.2 in \cite{GKRV}).  That is, the horizontal trace $\Tr(\cat{A}, F)$ can be viewed as the tautological receptacle for characters $[(\cat{M}, F_{\cat{M}})]$ of left $\cat{A}$-module categories $\cat{M}$ equipped with an $F$-semilinear endofunctor $F_{\cat{M}}$, i.e. a map of $\cat{A}$-module categories $F_{\cat{M}}: \cat{M} \rightarrow \cat{M}_F := \cat{A}_F \otimes_{\cat{A}} \cat{M}$.\footnote{Roughly, this is the data of $F_{\cat{M}} \in \cat{End}(\cat{M})$ with natural compatibility isomorphisms $F_{\cat{M}}(A \ast M) \simeq F(A) \ast F_{\cat{M}}(M)$ for $A \in \cat{A}, M \in \cat{M}$, i.e. for a functor to be $\cat{A}$-linear is a structure, not merely a property.}  

The trace $[A]$ of objects $A \in \cat{A}$ in Definition \ref{two trace definition} above is a special case in the following way: consider $\cat{M} := \cat{A}$ as the usual (left) regular $\cat{A}$-module category; for $A \in \mathrm{Ob}(\cat{A})$, we define $F_A(-) := F(-) \ast A$. In this case, we have $[A] = [\cat{A}, F_A]$.  In particular, the trace of the monoidal unit\footnote{The monoidal structure on $F$ gives rise to an $F$-equivariant structure on $1_{\cat{A}}$.} is $[1_{\cat{A}}] = [\cat{A}, F]$, i.e. the trace of the regular representation.
\end{defn}

\medskip

Moreover, the categorical trace provides a ``delooping'' of the naive trace.  To make the relationship between the two traces precise, we first recall the notion of a rigid monoidal category (see Definition 9.1.2 and Lemma 9.1.5 in \cite{GR}). 

\begin{defn}
Let $\cat{A}$ be a compactly generated stable monoidal $\infty$-category, with multiplication $\mu: \cat{A} \otimes \cat{A} \rightarrow \cat{A}$.  We say $\cat{A}$ is \emph{rigid} if the monoidal unit is compact, $\mu$ preserves compact objects, and if every compact object of $\cat{A}$ admits a left and right (monoidal) dual.
\end{defn}

We have the following relationship between vertical and horizontal traces of \cite{GKRV}, which may be interpreted via Theorem 1.1 of \cite{CP} as a compatibility of iterated traces.  Let $\cat{A}$ be a monoidal category, and $F$ a monoidal endofunctor.  We denote by $(\cat{A}, F)\mh\cat{mod}$ the 1-category (i.e. forget the 2-morphisms) of $\cat{A}$-module categories with $F$-semilinear endofunctors as in Definition \ref{two trace maps}.
\begin{thm}[Theorem 3.8.5 \cite{GKRV}, Theorem 1.1 \cite{CP}]\label{traceRelationship}\label{trace theorem}
Assume that $\cat{A}$ is compactly generated and rigid monoidal, and $F$ a monoidal endofunctor. Then, there is an equivalence of algebras\footnote{The opposite algebra appears because we took left modules in Definition \ref{two trace maps}.}
$$HH(\cat{A}, F) \simeq \End_{\Tr(\cat{A}, F)}([\cat{A}, F])^{op},$$
More generally, there is an equivalence of functors from the category of $F$-equivariant module categories:
\begin{equation*}\label{two traces}
HH(-) \simeq \Hom_{\Tr(\cat{A}, F)}([\cat{A},F], [-]): (\cat{A},F)\md\cat{mod}^R \longrightarrow HH(\cat{A}, F)\dmod.
\end{equation*}
In particular, assuming that $[\cat{A}, F]$ is a compact object, then the left adjoint to the functor $\Hom_{\Tr(\cat{A}, F)}([\cat{A},F], -)$ defines a fully faithful embedding which preserves compact objects, whose essential image is the category generated by $[\cat{A}, F]$:
\begin{equation*}
\begin{tikzcd}
\hh(\cat{A}, F)\module  \arrow[dr, "\simeq"] \arrow[rr, hook, "{[\cat{A}, F] \otimes_{\End([\cat{A}, F])} -}", shift left] & &  \arrow[ll, shift left, "{\Hom([\cat{A}, F], -)}"] \Tr(\cat{A}, F) \\
& \langle [\cat{A}, F] \rangle. \arrow[ur, hook] & 
\end{tikzcd}
\end{equation*}
\end{thm}

\medskip

\subsection{Traces in geometric settings}\label{HH geometric}

The geometric avatar for Hochschild homology is the derived loop space (or more generally, the derived fixed points of a self-map), see~\cite{BN:NT,loops and conns} for extended discussions.
\begin{defn}\label{defn loop space}
Let $X$ be a derived stack.
\begin{enumerate}
\item We define the \emph{derived loop space} $\cL X$ (or derived inertia stack) to be
$$\cL X = \Map_{\cat{DSt}_k}(S^1, X) \simeq X \utimes{X \times X} X$$
i.e. the derived mapping stack from a circle, or more concretely the derived self-intersection of the diagonal.
\item More generally, if $\phi: X \rightarrow X$ is a self-map, we define the \emph{derived fixed points} or \emph{$\phi$-twisted loop space} $\cL_\phi X$ to be the fiber product
$$\begin{tikzcd}
\cL_\phi X \arrow[r] \arrow[d, "\mathrm{ev}"] & X \arrow[d, "\Gamma_\phi"] \\
X \arrow[r, "\Delta"] & X \times X.
\end{tikzcd}
$$
i.e. the derived intersection of the diagonal with the graph $\Gamma_\phi = \mathrm{id}_X \times \phi$ of $\phi$.  Note that the derived fixed points of the identity is the derived loop space, i.e. $\cL_{\mathrm{id}_X} X = \cL X$.
\item The formation of derived loop spaces and derived fixed points are functorial, i.e. if $f: X \rightarrow Y$ is map of derived stacks, and $\phi_X, \phi_Y$ are compatible self-maps, then we have a map of derived stacks $\cL_{\phi} f: \cL_{\phi_X} X \rightarrow \cL_{\phi_Y} Y.$
\end{enumerate}
\end{defn}

\begin{exmp}\label{exmp hkr}
For $X$ a scheme over a characteristic 0 field $k$ we have that the derived loop space $\LL X\simeq \mathbb{T}_X[-1]$ is the total space of the shifted tangent complex to $X$ (see Proposition 4.4 in \cite{loops and conns}), while for $X=\pt/G$ we have $\LL X=G/G\simeq \mathrm{Loc}_G(S^1)$, i.e. the classical inertia stack (see Proposition 2.1.8 in \cite{Ch}). For a general stack the loop space is a combination of the shifted tangent complex with the inertia stack.
\end{exmp}

\begin{exmp}
For us, the self-maps above will arise via a action of a group $G$ on $X$, i.e. for $g \in G(k)$ we obtain a map $g: X \rightarrow X$.  Then, we have the relationship $\cL_g X = \cL(X/G) \times_{\cL(BG)} \{g\}$.
\end{exmp}

Note the parallel between the loop space, which is the self-intersection of the diagonal (the identity self-correspondence from $X$) and Hochschild homology (the trace of the identity on a category). As a result the push-pull functoriality of  categories of sheaves under correspondences implies an immediate relation between their Hochschild homology and loop spaces. Since $\QC$ is functorial under $*$-pullbacks and $\QC^!$ under $!$-pullbacks, this produces the following answers, both of which hold in particular for QCA stacks (see Corollary 4.2.2 of \cite{DG}, \cite{BN:NT}, and Example 2.2.10 in \cite{Ch}):
\begin{equation}\label{loop space hh ident}
HH(\QC(X), \phi_*)\simeq \Gamma(\cL_\phi X, \cO_{\cL_\phi X}), \;\;\;\;\;\;\;\;\;\; HH(\QC^!(X), \phi_*)\simeq \Gamma(\cL_\phi X, \omega_{\cL_\phi X}).
\end{equation}
In other words, taking $\phi = \mathrm{id}_X$, the Hochschild homology of $\QC(X)$ (respectively $\QC^!(X)$) is given by functions (respectively volume forms) on the derived loop space.  For $X=\Spec(R)$ a smooth affine scheme, along with Example \ref{exmp hkr} this recovers the Hochschild-Kostant-Rosenberg identification of Hochschild homology of $R\module$ with differentials on $R$, $$HH(R\module)=\cO(\cL X)=\cO(\mathbb{T}_X[-1])=\Sym(\Omega^1_R[1]).$$

\begin{exmp}[Quasicoherent sheaves under tensor product]
Let $X$ be a perfect stack in the sense of \cite{BFN}.  Then, the symmetric monoidal structure on $\QCoh(X)$ via tensor product has compact unit and multiplication.  We have that $\hh(\QCoh(X)) = \OO(\LL X)$ is an algebra object (with the multiplication given by the shuffle product after passing through HKR as in Example \ref{exmp hkr}; see Section 4.2 of \cite{loday} for a discussion of this structure), and the universal trace $\QCoh(X) \rightarrow \Tr(\QCoh(X)) = \QCoh(\LL X)$ given by pullback along evaluation at the identity.  Furthermore, the monoidal unit is $\OO_X \in \QCoh(X)$ with trace $[\OO_X] = \OO_{\LL X} \in \QCoh(\LL X)$.  Finally, we have
$$\OO(\LL X)\dmod \simeq \langle \OO_{\LL X} \rangle \subset \QCoh(\LL X)$$
where the fully faithful inclusion is an equivalence if $X$ is affine.
\end{exmp}

\medskip

We now establish a certain Calabi-Yau property of derived fixed points of smooth stacks (or more generally, smooth maps).  In our arguments it will be useful to factor the loop space of a map $\cL f: \cL X \rightarrow \cL Y$ through the following intermediate derived stack, which we define in three equivalent ways. 
\begin{defn}\label{defn rel loops}
 Let $f: X \rightarrow Y$ be a map of derived stacks with compatible self-maps $\phi_X, \phi_Y$, and define $Z := X \times_Y X$.  We define $\cL_\phi Y_X$ via the pullback diagrams:
$$\begin{tikzcd}
\cL_\phi Y_X \arrow[r] \arrow[d] & X \arrow[d, "\Gamma_\phi"] &\cL_\phi Y_X\arrow[r] \arrow[d] & X \arrow[d, "f \times \phi_X"] & \cL_\phi Y_X \arrow[r] \arrow[d] & \cL_\phi Y \arrow[d, "\mathrm{ev}"] \\
Z \arrow[r] & X \times X & X \arrow[r, "f \times \mathrm{id}_X"'] & Y \times X & X \arrow[r, "f"'] & Y.
\end{tikzcd}$$
Roughly, this is the derived moduli stack of paths in $X$ mapping to loops in $Y$.
\end{defn}

The following lemma is a straightforward verification of the depicted diagrams, which we leave to the reader.
\begin{lemma}\label{lemma rel loops}
The above three presentations are canonically equivalent, and we have a canonical factorization
$$\begin{tikzcd}
\cL_\phi X \arrow[r, "\delta"] & \cL_\phi Y_X \arrow[r, "\pi"] & \cL_\phi Y
\end{tikzcd}$$
where the maps are realized via the base change
$$\begin{tikzcd}
\cL_\phi X \arrow[r, "\delta"] \arrow[d, "\mathrm{ev}_X"'] & \cL_\phi Y_X \arrow[d] \arrow[r] & X \arrow[d, "\Gamma_{\phi_X}"] & \cL_\phi Y_X \arrow[r, "\pi"] \arrow[d, "\mathrm{ev}_{X/Y}"'] & \cL_\phi Y \arrow[d] \arrow[r] & Y \arrow[d, "\Gamma_{\phi_Y}"] \\
X \arrow[r,"\Delta_f = \Delta_{X/Y}"'] & Z \arrow[r] & X \times X & X \arrow[r, "f"'] & Y \arrow[r, "\Delta_Y"'] & Y \times Y.
\end{tikzcd}$$
i.e. $\delta$ is a base change of the relative diagonal for $f$, and $\pi$ is a base change of $f$ itself.
\end{lemma}

\begin{exmp}
When $\phi$ is the identity and $Y = \pt$, the factorization above is just $\cL X \rightarrow X \rightarrow \pt$.
\end{exmp}

When $X$ is a smooth stack, there is an equivalence of categories $\Perf(X) = \Coh(X)$, thus by (\ref{loop space hh ident}) we expect that $\cO(\cL X) \simeq \omega(\cL X)$.  It turns out that this equivalence on global sections comes from a map on the underlying sheaves themselves.  We now establish the following Calabi-Yau property of derived fixed points of smooth stacks, which we will use repeatedly in our arguments.  We refer the reader to Section 8 of \cite{AG} for discussion of quasi-smoothness for derived Artin stacks.
\begin{lemma}\label{calabi yau}
Let $X, Y$ be derived Artin stacks equipped with proper self-maps $\phi_X, \phi_Y$, and let $f: X \rightarrow Y$ be a smooth relative Artin 1-stack\footnote{By this we mean such that the relative cotangent complex is perfect of Tor amplitude $[0, 1]$, i.e. the fibers are are allowed to be stacky, and in particular, this map does not need to be representable by schemes.} commuting with $\phi_X, \phi_Y$. Then, there is a canonical equivalence of functors 
$$\cL_\phi f^! \simeq \cL_\phi f^*: \IndCoh(\cL_\phi Y) \longrightarrow \IndCoh(\cL_\phi X).$$
In particular, if $X$ is a smooth Artin 1-stack with a proper self-map $\phi$, then $\omega_{\cL_\phi  X} \simeq \cO_{\cL_\phi X}$, and if $f$ is proper then $\cL_\phi f^*$ is biadjoint to $\cL_{\phi} f_*$. 
\end{lemma}
\begin{proof}
Following the notation and factorization in Lemma \ref{lemma rel loops}, we have canonical identifications:
$$\omega_{\cL_\phi X/\cL_\phi Y_X} \simeq \mathrm{ev}_X^*\omega_{X/Z}, \;\;\;\;\;\;\;\;\;\;\;\;\;\; \omega_{\cL_\phi Y_X/\cL_\phi Y} \simeq \mathrm{ev}_{X/Y}^*\omega_{X/Y}.$$
Furthermore, after choosing\footnote{The definition of Hochschild homology implicitly requires us to choose an orientation on the circle $S^1$.  We make one such choice, once and for all, which forces a particular choice here (i.e. a choice of sign).} one of the projections $Z = X \times_Y X \rightarrow X$, the usual exact triangle for cotangent complexes for the composition $X \rightarrow Z \rightarrow X$ gives a canonical equivalence 
$$\omega_{X/Z} \simeq \Delta_{X/Y}^*\omega_{Z/X}^{-1} \simeq \omega_{X/Y}^{-1}.$$
Thus, we have a canonical equivalence
$$\omega_{\cL_\phi X/\cL_\phi Y} \simeq \mathrm{ev}_X^*\omega_{X/Y}^{-1} \otimes \delta^*\mathrm{ev}_{X/Y}^*\omega_{X/Y} \simeq \cO_{\cL_\phi X}.$$
By assumption the cotangent complex $\mathbb{L}_f$ is perfect in degrees $[0, 1]$, so the relative cotangent complex $\mathbb{L}_{\Delta_{X/Y}}$ is perfect in degrees $[-1, 0]$; in particular, $\Delta_{X/Y}$ is representable by schemes and quasi-smooth and thus we have a canonical equivalence (see Proposition 7.3.8 of \cite{indcoh}) $\cL_\phi f^! \simeq \cL_\phi f^* \otimes \omega_{\cL_\phi X/\cL_\phi Y} \simeq \cL_\phi f^*$ as desired.
\end{proof}

Furthermore, by functoriality of Hochschild homology, for a map of stacks $f: X \rightarrow Y$ we expect that the pullback and pushforward functors define maps of global functions or volume forms $HH(f^*): \cO(\cL Y) \rightarrow \cO(\cL X)$ and (if $f$ is proper) $HH(f_*): \omega(\cL X) \rightarrow \omega(\cL Y)$.  We identify this map with the global sections of a natural map on the underlying sheaves in two cases of concern (see Appendix \ref{app functoriality} for the proof).
\begin{defn}\label{HH decat functoriality}
Let $f: X \rightarrow Y$ be a map of QCA stacks, and $\phi_X, \phi_Y$ compatible proper self-maps.
\begin{enumerate}
\item If $f$ is proper, then we have a pushforward map $\omega(\cL_\phi f_*): \omega(\cL_\phi X) \rightarrow \omega(\cL_\phi Y)$ of global volume forms.  That is, by Remark 4.6 in \cite{BN:NT}, since $f$ is proper, $\cL_\phi f: \cL_\phi X \rightarrow \cL_\phi Y$ is proper; $\omega(\cL_\phi f_*)$ is the global sections of the counit of the adjunction $(\cL_\phi f_*, \cL_\phi f^!)$ applied to $\omega_{\cL_\phi Y}$.
\item If $f$ is smooth, then we have a ``Gysin'' pullback $\omega(\cL_\phi f^*): \omega(\cL_\phi Y) \rightarrow \omega(\cL_\phi X)$ of global volume forms.  That is, by Proposition \ref{calabi yau}, if $f$ is smooth then $\cL_\phi f$ is Calabi-Yau; passing through this equivalence, $\omega(\cL_\phi f^*)$ is the global sections of the unit of the adjunction $(\cL_\phi f^*, \cL_\phi f_*)$ applied to $\omega_{\cL_\phi Y}$.
\end{enumerate}
\end{defn}

\begin{prop}\label{HH prop}
Let $f: X \rightarrow Y$ be map of QCA stacks with compatible proper self-maps $\phi_X, \phi_Y$.
\begin{enumerate}
\item There are canonical identifications
$$HH(\IndCoh(X), \phi_*) \simeq \omega(\cL_\phi X).$$
\item Suppose $f$ is proper, and consider $f_*: \IndCoh(X) \rightarrow \IndCoh(Y)$.  Then, the map $HH(f_*, \phi_*)$ is canonically identified with the map on global volume forms $\omega(\cL_\phi f_*)$.
\item Suppose that $f$ is smooth, and consider $f^*: \IndCoh(Y) \rightarrow \IndCoh(X)$.  Then, the map $HH(f^*, \phi_*)$ is canonically identified with the map on volume forms $\omega(\cL_\phi f^*)$.
\end{enumerate}
\end{prop}

\medskip

\subsection{Convolution patterns in Hochschild homology}\label{convolution}
Convolution patterns in Borel-Moore homology and algebraic $K$-theory play a central role in the results of \cite{CG}.  We now describe a similar pattern which appears in Hochschild homology. 

\begin{defn}\label{conv setup}
We will work with the following general setup (see Section 1.5 of \cite{BNP}).
\begin{enumerate}
\item[$\bullet$] $f:X\to Y$ is a proper morphism of smooth, QCA stacks over $k$, and $Z=X\times_Y X$.
\item[$\bullet$] $\phi_X: X \rightarrow X$ and $\phi_Y: Y \rightarrow Y$ are (representable) proper self-maps commuting with $f$, inducing a proper self-map $\phi: Z \rightarrow Z$.
\end{enumerate}
We refer to any $Z$ arising from the set-up above a \emph{convolution space}, and call the category $\IndCoh(Z)$ a \emph{convolution category}.
\end{defn}

In this setup the category $\QC^!(Z)$ carries a monoidal structure under convolution\footnote{As  explained in Remark 3.0.7 and Lemma 3.0.8 of \cite{BNP2}, on the compact objects $\Coh(Z)$ there are two monoidal products, given by $*$- or $!$-convolution, intertwined by Grothendieck duality. We will default to the $!$-version, which is amenable to the ind-completed category $\IndCoh(Z)$.}, and $\phi_*$ is a monoidal endofunctor.  The convolution monoidal structure restricts to the compact objects $\Coh(Z)$ thanks to the smoothness of $X$ (hence finite Tor-dimension of the diagonal of $X$) and the properness of $f$; furthermore, since $\phi$ is proper, $\phi_*$ has a colimit-preserving right adjoint, and preserves $\Coh(Z)$.  

\medskip

By Theorem 1.1.3 of \cite{BNP2}, there is an equivalence of small monoidal categories\footnote{Via the discussion in Section 4.7 of \cite{HA}, endofunctor categories naturally possess the structure of an associative monoidal $\infty$-category.  Theorem 1.1.3 in \cite{BNP2} identifies the underlying categories, with convolution corresponding to composition object-by-object.  Thus we can simply define the monoidal structure (with all its higher coherence compatibilities) on the left by transporting it from the right.}
$$\begin{tikzcd}(\Coh(X \times_Y X),\ast) \arrow[r, "\simeq"] & (\Fun^{ex}_{\Perf(Y)}(\Perf(X), \Perf(X)),\circ)\end{tikzcd}$$
which takes an integral kernel $\cK \in \Coh(X \times_Y X)$ to the functor $\cF \mapsto R\pi_{2*}(R\pi_1^!\cF \otimes^L \cK)$ where $\pi_1, \pi_2: X \times_Y X \rightarrow X$ are the projections.  Moreover, we will argue in Theorem \ref{rigid} that $(\QC^!(Z),\ast)$ is rigid monoidal.  The monoidal unit is the dualizing sheaf of the relative diagonal $\omega_{\Delta} := \iota_* \omega_{X}$, where $\iota:X\to X\times_Y X$.

Recall (from Section~\ref{HH geometric}) that the Hochschild homology of $\Coh(Z)$ (or equivalently of its large variant $\QC^!(Z)$ by Remark 2.2.11 of \cite{Ch}) for a stack $Z$ is given geometrically by volume forms on the loop space, or in the case of the trace of $\phi_*$ the derived fixed points:
$$\hh(\IndCoh(Z), \phi_*) \simeq \Gamma(\cL_\phi Z, \omega_{\cL_\phi Z}).$$ Thus the vertical trace of the monoidal category $\Coh(Z)$ defines an algebra structure on global distributions $\Gamma(\cL_\phi Z, \omega_{\cL_\phi Z})$.

We want to relate this convolution structure on sheaves to its decategorified version involving volume forms on the corresponding loop spaces. Thus we consider the loop map $\LL_\phi f:\LL_\phi X\to \LL_\phi Y$ to $f$, whose self-fiber product is $\LL_\phi Z\simeq \LL_\phi X \times_{\LL_\phi Y} \LL_\phi X$.  Note that $\LL_\phi f$ is a proper map of quasismooth derived stacks.  In particular, $\omega_{\LL_\phi X}$ is coherent (a compact object in $\QC^!(\LL_\phi X)$) and $\LL_\phi f_*$ preserves coherence. We thus define our main object of interest.

\begin{defn}\label{univ trace sheaf}\label{universal trace sheaf}
We define the \emph{universal trace sheaf}
$$\cS_{X/Y,\phi} := \cL_\phi f_* \omega_{\cL_\phi X} \simeq \cL_\phi f_* \OO_{\cL_\phi X} \in \Coh(\cL_\phi Y).$$
The latter isomorphism follows since the loop space of smooth stacks are naturally Calabi-Yau (see Lemma \ref{calabi yau}).
\end{defn}

The endomorphisms of the universal trace sheaf have a close relationship to volume forms on the loop space of the convolution space.  Namely, we have a canonical equivalence
$$\omega(\cL_\phi Z) \simeq \End_{\cL_\phi Y}(\cS_{X/Y,\phi}).$$
Furthermore, these equivalences are functorial at the sheafy level; on the left, this was discussed in Definition \ref{HH decat functoriality}.  On the right, the functoriality arises via the following functoriality of the universal trace sheaf.
\begin{defn}\label{springer functorial}
Let $(X, Y, f, \phi)$ and $(X', Y', f', \phi')$ as in Definition \ref{conv setup} (with convolution spaces $Z, Z'$), and write $\cS := \cS_{X/Y,\phi}$ and $\cS' := \cS_{X'/Y',\phi'}$.  Suppose we have maps $\alpha_X: X \rightarrow X'$ and $\alpha_Y: Y \rightarrow Y'$ commuting with $f, f'$, inducing $\alpha_Z: Z \rightarrow Z'$.  Then, we have the following due to base change.
\begin{enumerate}
\item Suppose that $X = X'$ and that $\alpha_Y$ is proper.  Then, there is a canonical equivalence $\cL\alpha_{Y*} \cS \simeq \cS'$, and the functor $\alpha_{Z*}: \Coh(Z) \rightarrow \Coh(Z')$ is monoidal.
\item Suppose that $\alpha_Y$ is smooth and $f$ is base-changed from $f'$, i.e. $X = X' \times_{Y'} Y$.  Then there is a canonical equivalence $\cL\alpha_Y^! \cS' \simeq \cS$, and the functor $\alpha_{Z}^!: \Coh(Z') \rightarrow \Coh(Z)$ is monoidal.
\end{enumerate}
\end{defn}

The functorialities on the two sides of the equivalence are compatible.  
\begin{prop}\label{conv volume}
We let $p: Z \rightarrow Y$ denote the structure map.  In the set-up of Definition \ref{conv setup}, we have a canonical equivalence
$$\zeta: \cL_\phi p_* \omega_{\cL_\phi Z} \simeq \intEnd_{\cL_\phi Y}(\cS)$$
with $\zeta'$ defined analogously, such that if $\alpha: Y \rightarrow Y'$ is proper and $X = X'$, we have commuting squares
$$\begin{tikzcd}[column sep=huge]
\cL_\phi \alpha_{*} \cL_\phi p_*\omega_{\cL_\phi Z} \arrow[r, "\simeq"', "\cL_\phi \alpha_{*} (\zeta)"] \arrow[d, "{\mathrm{Def. \ref{HH decat functoriality}}}"'] & \cL_\phi \alpha_{*} \intEnd_{\cL_\phi Y}(\cS) \arrow[d, "{\mathrm{Def. \ref{springer functorial}}}"] \\
\cL_\phi p'_*\omega_{\cL_\phi Z'} \arrow[r, "\simeq"', "\zeta'"] & \intEnd_{\cL_\phi Y'}(\cS').
\end{tikzcd}$$
while if $\alpha: Y \rightarrow Y'$ is smooth and $X = X' \times_{Y'} Y$, we have commuting squares
$$\begin{tikzcd}[column sep=huge]
\cL_\phi p'_* \omega_{\cL_\phi Z'}\arrow[d, "{\mathrm{Def. \ref{HH decat functoriality}}}"'] \arrow[r, "\simeq"', "\zeta'"] & \intEnd_{\cL_\phi Y'}(\cS') \arrow[d, "{\mathrm{Def. \ref{springer functorial}}}"] \\
\cL_\phi \alpha_{*} \cL_\phi p_*\omega_{\cL_\phi Z}\arrow[r, "\simeq"', "\cL_\phi \alpha_{*} (\zeta)"]  & \cL_\phi \alpha_{*} \intEnd_{\cL_\phi Y}(\cS).
\end{tikzcd}$$
\end{prop}
\begin{proof}
Application of Proposition \ref{app conv volume}, noting that if $f$ is smooth then $\cL_\phi f$ is Calabi-Yau by Proposition \ref{calabi yau}.
\end{proof}

\begin{remark}[Convolution of volume forms and endomorphisms of $\cS_{X/Y}$]
Applying the above proposition to $\cL_\phi f: \cL_\phi X \rightarrow \cL_\phi Y$, i.e. if we sheafify over $\LL_\phi Y$, we can identify this algebra structure more concretely as convolution of volume forms on $\LL_\phi Z$.  That is, $\LL_\phi Z=\LL_\phi X\times_{\LL_\phi Y} \LL_\phi X$ has the structure of proper monoid in stacks over $\LL_\phi Y$, from which one deduces the structure of algebra object in $(\QC^!(\LL_\phi Y),\otimes^!)$ on the pushforward of $\omega_{\LL_\phi Z}$. One can also use proper descent for $\LL_\phi f:\LL_\phi X\to \LL_\phi Y$ to identify this sheaf of algebras with the internal endomorphism sheaf of $\cS_{X/Y}$ -- an analog, in the setting of derived categories of coherent sheaves on derived stacks, of the standard proof (see e.g. \cite{CG}) that self-Ext of the Springer sheaf is identified with Borel-Moore homology of $Z$. It would be interesting to see how these arguments globalize over $\LL_\phi Y$ to give the isomorphism $\Gamma(\cL_\phi Z, \omega_{\LL_\phi Z})\simeq \End_{\IndCoh(\cL_\phi Y)}(\cS_{X/Y})$ of Theorem~\ref{traces convolution}.
\end{remark}

\medskip

\subsubsection{Horizontal trace of convolution categories}\label{section horizontal trace}

Recall that Theorem \ref{traceRelationship} identifies the vertical trace $\hh(\IndCoh(Z),\ast)$ as the endomorphism algebra of the distinguished object in the horizontal trace $\Tr(\IndCoh(Z),\ast)$, under the assumption that this distinguished object is compact (and a rigidity condition to be addressed in Theorem \ref{rigid}).  In this section we discuss this horizontal trace in the context of convolution spaces following \cite{BNP}, slightly generalizing the main theorem of \emph{op. cit.}

For this we require a discussion of singular supports; we summarize the main points and refer the reader to \cite{AG, BNP} for details.  Note that singular supports do not appear in our main application Theorem \ref{thm endomorph}, since the singular support condition there is actually a classical support condition (see Remark \ref{no sing supp}).
\begin{defn}
Let $f: X \rightarrow Y$ be a representable map of quasi-smooth stacks.  
\begin{enumerate}
\item We define the \emph{scheme of singularities} or (classical) \emph{odd cotangent bundle} to be
$$\TT_X := \Spec_X \Sym_X H^1(\cT_X) = \Spec_X \Sym_X H^0(\cT_X[1])$$
where $\cT_X$ denotes the tangent complex of $X$, i.e. the $\cO_X$-linear dual of the cotangent complex. 
\item Any ind-coherent sheaf $\cF \in \IndCoh(X)$ has a closed conical \emph{singular support} $\mathrm{SS}(\cF) \subset \TT_X$. To any subset $\Lambda \subset \TT_X$ we can associate the full category $\IndCoh_\Lambda(X) \subset \IndCoh(X)$ consisting of sheaves with the specified singular support.
\item Let $\Lambda_X \subset \TT_X$ and $\Lambda_Y \subset \TT_Y$, and consider the correspondence
$$\begin{tikzcd} \TT_X & \arrow[l, "df"'] \arrow[r, "p"] \TT_Y \times_Y X & \TT_Y \end{tikzcd}.$$
One can push forward and pull back singular support conditions 
$$f_*\Lambda_X = \overline{p(df^{-1}(\Lambda_X))}, \;\;\;\;\;\;\; f^!\Lambda_Y = \overline{df(p^{-1}(\Lambda_Y))}$$ 
such that the pushforward and pullback functors preserve singular supports, i.e.
$$f_*: \IndCoh_{\Lambda_X}(X) \rightarrow \IndCoh_{f_*\Lambda_X}(Y), \;\;\;\;\;\;\; f^!: \IndCoh_{\Lambda_Y}(Y) \rightarrow \IndCoh_{f^!\Lambda_Y}(X).$$
\end{enumerate}
\end{defn}

\begin{exmp}\label{ssupp example}
If $X$ is smooth, then $\TT_X = X$, i.e. there are no possible singular codirections to consider.  In particular, the nontrivial fibers of the map $\TT_X \rightarrow X$ live over the singular locus of $X$.

When $\Lambda = \TT_X$, we have $\IndCoh_\Lambda(X) = \IndCoh(X)$.  At the opposite extreme, when $\Lambda = \{0\}_X$ is the zero section, we have $\IndCoh_\Lambda(X) = \QCoh(X)$.  If $Z \subset X$ is a closed subscheme and $\Lambda = Z \times_X \TT_X$, then $\IndCoh_\Lambda(X) = \IndCoh_Z(X)$, i.e. the full subcategory of ind-coherent sheaves with classical support at $Z \subset X$.  If instead we take $\Lambda = Z \times \{0\}_X$, then $\IndCoh_\Lambda(X) = \QCoh_Z(X)$.
\end{exmp}

\medskip

The following singular support condition appears when taking traces of convolution categories.
\begin{defn}\label{trace correspondence}
Recall the notation from Definition \ref{defn loop space} and Definition \ref{defn rel loops}.  We have the following \emph{trace correspondence}:
$$\begin{tikzcd}
Z=X \times_Y X & \arrow[l, "\delta"']  \cL_\phi Y_X = Z \utimes{X \times X} X \simeq X \utimes{Y \times X} X \arrow[r, "\pi"]  & \cL_\phi Y.
\end{tikzcd}$$
We define a singular support condition $\Lambda_{X/Y,\phi} := \pi_*\delta^!\TT_Z$.
\end{defn}

We now give a description of the horizontal trace.  The following statement is more general than the statement of Theorem 3.3.1 in \cite{BNP}, but follows from the same argument in the proof with the definitions given above; the proof is in Appendix \ref{app horiz}.
\begin{thm}\label{convolution category trace}
There is a canonical identification of the horizontal trace (i.e. the monoidal Hochschild homology)
$$\Tr((\IndCoh(Z),\ast), \phi_*) \simeq \IndCoh_{\Lambda_{X/Y,\phi}}(\cL_\phi Y),$$ with the universal trace given by\footnote{Note that our trace functor is given by $\delta^!$ rather than the $\delta^*$ in \cite{BNP}, since we employ the $!$-transform rather than the $*$-transform.}
$$[-] = \pi_* \delta^!: \IndCoh(Z) \rightarrow \IndCoh_{\Lambda_{X/Y,\phi}}(\cL_\phi Y).$$
\end{thm}

Next we identify the universal trace sheaf (i.e. coherent Springer sheaf) as the trace of the monoidal unit (which is a compact object of the trace category) or regular representation.

\begin{lemma}\label{traceSpringerSheaf}
There is a natural equivalence $\cS_{X/Y,\phi} \simeq [\omega_\Delta] = \pi^* \delta^! \omega_{\Delta}$ in $\Coh(\cL_\phi Y)$.
\end{lemma}
\begin{proof}
The calculation of $\delta^! \omega_{\Delta} = \delta^! \Delta_* \omega_X$ arises via base change along the diagram
$$\begin{tikzcd}
\cL_\phi X \arrow[d] \arrow[r] & \cL_\phi Y_X = Z \times_{X \times X} X \arrow[d] \\
X \arrow[r, "\Delta"'] & Z = X \times_Y X
\end{tikzcd}$$
and the statement follows.
\end{proof}

\medskip

\subsubsection{Trace delooping in convolution categories}

We now deduce the main structural relation between universal trace sheaves (see Definition \ref{universal trace sheaf}) and iterated categorical traces of convolution categories.
\begin{thm}\label{traces convolution}\label{rigid}
Let $f: X \rightarrow Y$ be as in Definition \ref{conv setup}.  Then, the convolution category $\IndCoh(X \times_Y X)$ is rigid.  In particular, the statements of Theorem \ref{trace theorem} apply: the vertical trace of the convolution category $(\IndCoh(Z),\ast)$ is identified as an algebra with the endomorphisms of the universal trace sheaf
$$\hh(\IndCoh(X \times_Y X), \phi_*) \simeq \End_{\QC^!(\cL_\phi Y)}(\cL_\phi f_* \omega_{\cL_\phi X}).$$
\end{thm}
\begin{proof}
We need to verify that $\IndCoh(Z)$ is rigid monoidal.  Standard arguments show that integral transforms arising via coherent sheaves preserve compact objects; this statement is also contained within Theorem 1.1.3 in \cite{BNP2}; one further immediately observes that the monoidal unit $\Delta_* \omega_X$ is a compact object, i.e. coherent, since the diagonal is a closed embedding.  It remains to verify the existence of right and left duals of coherent sheaves $\cK \in \IndCoh(Z)$.  Using \emph{loc. cit.}, it suffices to show that the right and left adjoints of the corresponding integral transform $F_{\cK}: \QCoh(X) \rightarrow \QCoh(X)$ preserve compact objects, thus are realized by integral transforms with coherent kernels.  We note that since the projection maps $p: Z \rightarrow X$ are quasi-smooth, the functors $p^!$ and $p^*$ differ by a shifted line bundle.  By Lemma 3.0.8 in \emph{op. cit.} we can consider equivalently either the $*$ or $!$-transforms up to twisting by Grothendieck duality.  For convenience we will consider the $*$-transform.

To see the claim, note that we can write the $*$-integral transform $F_{\cK}$ as a composition:
$$\begin{tikzcd}[column sep=huge]
\QCoh(X) \arrow[r, "p^*"] & \QCoh(Z) \arrow[r, "{- \otimes \cK}"] & \IndCoh(Z) \arrow[r, "p_*"] & \IndCoh(X).
\end{tikzcd}$$
We claim that the right adjoint preserves compact objects.  The claim for the left adjoint follows similarly by replacing $p^*$ with a twist of $p^!$ by a shifted line bundle.  The right adjoints define a sequence of functors
$$\begin{tikzcd}[column sep=huge]
\QCoh(X) & \arrow[l, "p_*"'] \QCoh(Z)&  \arrow[l, "{\intHom_{\IndCoh(Z)}(\cK, -)}"']  \IndCoh(Z)&  \arrow[l, "{p^! = p^* \otimes \cL}"'] \IndCoh(X).
\end{tikzcd}$$
The functor $\intHom_{\IndCoh(Z)}(\cK, -): \IndCoh(Z) \rightarrow \QCoh(Z)$ is defined as follows.  Given $\cG \in \IndCoh(Z)$, we may write $\cG = \colim_i \cG_i$ with $\cG_i \in \Coh(Z)$.  Since $\cK$ is compact, we may define:
$$\intHom_{\IndCoh(Z)}(\cK, \cG) := \lim_i \intHom_Z(\cK, \cG_i) \in \QCoh(Z)$$
where the internal Hom on the right is taken inside $\Coh(Z) \subset \QCoh(Z)$ as usual.  Let us justify the claim that this functor is a right adjoint to tensoring with $\cK$.  Let $\cF \in \QCoh(Z)$, and write $\cF = \colim_j \cF_j$ with $\cF_j \in \Perf(Z)$.  Then, by the usual adjunction in $\QCoh(Z)$, and using the facts that the $\cF_j$ are compact in $\QCoh(Z)$ and that $\cF_j \otimes \cK \in \Coh(Z)$ are compact in $\IndCoh(Z)$ since $\cF_j$ are perfect, we have:
$$\Hom_{\QCoh(Z)}(\cF, \intHom_{\QCoh(Z)}(\cK, \cG)) \simeq \Hom_{\QCoh(Z)}(\colim_j \cF_j, \lim_i \intHom_Z(\cK, \cG_i)) $$
$$\simeq \lim_{i,j} \Hom_{\QCoh(Z)}(\cF_j, \intHom_Z(\cK, \cG_i)) \simeq \lim_{i,j} \Hom_{\IndCoh(Z)}(\cF_j \otimes \cK, \cG_i) \simeq \Hom_{\IndCoh(Z)}(\cF \otimes \cK, \cG).$$

Finally, we claim that $\intHom_{\IndCoh(Z)}(\cK, -)$ sends $\Perf(Z)$ to $\Coh(Z)$.  Assuming this claim, then the composite of the sequence of right adjoints above preserves compact objects, which finishes the proof: since $X$ is smooth, the image of $\Coh(X) = \Perf(X)$ under $p^! = p^* \otimes \cL$ takes values in $\Perf(Z)$, and since $p$ is proper (and again since $X$ is smooth) the image of $\Coh(Z)$ under $p_*$ takes values in $\Coh(X) = \Perf(X)$.  To prove the claim, note that the Grothendieck dual $\mathbb{D}(\cK) = \intHom_Z(\cK, \omega_Z)$ is coherent, and since $Z$ is quasi-smooth, $\omega_Z$ is a line bundle, so we have for $\cE \in \Perf(Z)$:
$$\intHom_{\IndCoh(Z)}(\cK, \cE) = \intHom_{\IndCoh(Z)}(\cK, \omega_Z) \otimes_{\cO_Z} \omega_Z^{-1} \otimes_{\cO_Z} \cE \simeq \bD(\cK) \otimes_{\cO_Z} \omega_Z^{-1} \otimes_{\cO_Z} \cE$$
which is coherent.
\end{proof}

\medskip

\subsection{Trace of the standard categorical representation}\label{standard rep}

In Lemma \ref{traceSpringerSheaf}, we have computed the trace of the regular representation $\IndCoh(Z)$ of $\IndCoh(Z)$ to be the universal trace sheaf, i.e. $[\IndCoh(Z), \phi_*] \simeq \cS_{X/Y, \phi} := \cL_\phi f_* \cO_{\cL_\phi X}$.  Our convolution set-up comes equipped with another natural module category: the \emph{standard representation}, i.e. the module category $\IndCoh(X)$.  In this section we compute the trace of this categorical representation, and relate it to the trace of the regular representation in a special case.  We first note a degenerate example.

\begin{exmp}\label{degenerate trace}
Consider the case when $X = Y = Z$ is smooth.  In this case, $\IndCoh(Y) = \QCoh(Y)$, and the trace correspondence of Definition \ref{trace correspondence} is simply given by pullback along the evaluation $\mathrm{ev}: \cL_\phi Y \rightarrow Y$:
$$\begin{tikzcd}
Y & \arrow[l, "{\mathrm{ev}}"'] \cL_\phi Y \arrow[r, equals] & \cL_\phi Y
\end{tikzcd}$$
and the corresponding singular support condition $\Lambda_{Y/Y,\phi} = \{0\}_{\cL_\phi Y}$ is the zero section, i.e. we have $\Tr(\IndCoh(Y),\phi_*) = \QCoh(\cL_\phi Y)$ (see Corollary 5.2 of \cite{BFN}).  The standard representation is the regular representation, and by Theorem 3.3.1 of \cite{BNP} (and Proposition \ref{calabi yau}), the trace of the regular representation is the structure sheaf
$$[\IndCoh(Y), \phi_*] = [\omega_Y] = \omega_{\cL_\phi Y} \simeq \cO_{\cL_\phi Y}.$$
\end{exmp}

\medskip

We recall a few notions from Section 2.3 of \cite{BNP}.  The following functors allow us to pass between categories with different singular supports.
\begin{defn}
For a pair $(X, \Lambda_X)$, there is an adjoint pair of functors (see Definition 2.3.2 of \cite{BNP}):
$$\iota_\Lambda: \begin{tikzcd} \IndCoh_\Lambda(X) \arrow[r, shift left] & \IndCoh(X) \arrow[l, shift left] \end{tikzcd} : \Gamma_\Lambda$$
where $\iota_\Lambda$ is the natural inclusion, and $\Gamma_\Lambda$ is the corresponding colocalization.\footnote{I.e. a ``projection'' functor to the subcategory $\IndCoh_{\Lambda}(X)$, which we view as a singular support analogue of local cohomology.  Note the abusive notation, i.e. the local cohomology functor usually refers to the functor $\iota_\Lambda \circ \Gamma_\Lambda$.}  
\end{defn}

\medskip

We need an identification of the relative tensor product of convolution categories, with specified support. We work in the set-up of Definition \ref{conv setup}: let $X_i$ be smooth QCA stacks over $k$, proper over $Y$, and let $Z_{ij} = X_i \times_Y X_j$.
\begin{defn}
Let $\Lambda_{12} \subset \TT_{Z_{12}}$ and $\Lambda_{23} \subset \TT_{Z_{23}}$.  Consider the diagram
$$\begin{tikzcd}
Z_{12} \times Z_{23} & \arrow[l, "\delta"'] \arrow[r, "\pi"] X_1 \times_Y X_2 \times_Y X_3 & Z_{13}.
\end{tikzcd}$$
We define the \emph{convolution of singular supports}
$$\Lambda_{12} \ast \Lambda_{23} = \pi_* \delta^!(\Lambda_{12} \boxtimes \Lambda_{23}).$$
We say that $\Lambda_{ij}$ is \emph{$Z_{ii}$-stable} if $\TT_{Z_{ii}} \ast \Lambda_{ij} \subset \Lambda_{ij}$.
\end{defn}

\begin{rmk}
The trace singular support condition $\Lambda_{X/Y}$ of Definition \ref{trace correspondence} can be viewed as the convolution of $\TT_Z$ with itself ``in a circle.''
\end{rmk}

We immediately observe that the convolution action restricts to an action of $\IndCoh_{\Lambda_{ii}}(Z_{ii})$ on $\IndCoh_{\Lambda_{ij}}(Z_{ij})$ if and only if $\Lambda_{ij}$ is $\Lambda_{ii}$-stable.  In particular, we have the following identification, which we prove in Appendix \ref{app horiz}; a proof will also appear in \cite{CD}.
\begin{prop}\label{conv comp}
In the set-up above, let $\Lambda_{12} \subset \TT_{Z_{12}}$ and $\Lambda_{23} \subset \TT_{Z_{23}}$ be $Z_{22}$-stable.  Define $\Lambda_{13} := \Lambda_{12} \ast \Lambda_{23}$.
Then convolution defines an equivalence of categories:
$$\begin{tikzcd}
\IndCoh_{\Lambda_{12}}(Z_{12}) \otimes_{\IndCoh(Z_{22})} \IndCoh_{\Lambda_{23}}(Z_{23}) \arrow[r, "\simeq"] & \IndCoh_{\Lambda_{13}}(Z_{13}).
\end{tikzcd}$$
Furthermore, we have the following functoriality of supports: let $\Lambda_{i,i+1} \subset \Lambda_{i,i+1}'$ be another singular support condition on $Z_{i,i+1}$ (for $i=1,2$) with $\Lambda_{13}' := \Lambda_{12}' \ast \Lambda_{23}'$.  Then, $\Lambda_{13} \subset \Lambda_{13}'$, and the following squares commute:
$$\begin{tikzcd}
\IndCoh_{\Lambda_{12}}(Z_{12}) \tens{\IndCoh(Z_{22})} \IndCoh_{\Lambda_{23}}(Z_{23}) \arrow[r, "\simeq"] \arrow[d, "\iota_{\Lambda_{12}} \otimes \iota_{\Lambda_{23}}"', shift right] & \IndCoh_{\Lambda_{13}}(Z_{13}) \arrow[d, "\iota_{\Lambda_{13}}"', shift right] \\
\IndCoh_{\Lambda_{12}'}(Z_{12}) \tens{\IndCoh(Z_{22})} \IndCoh_{\Lambda_{23}'}(Z_{23}') \arrow[r, "\simeq"] \arrow[u, "\Gamma_{\Lambda_{13}} \otimes \Gamma_{\Lambda_{23}}"', shift right] & \IndCoh_{\Lambda_{13}'}(Z_{13}). \arrow[u, shift right, "\Gamma_{\Lambda_{13}}"'] 
\end{tikzcd}$$
\end{prop}

\medskip

These actions are canonically $\phi_*$-semilinear.   We now compute the trace of the categorical representation, which arises via functoriality of horizontal traces (see Section 3.5 of \cite{BN:NT} for details).  Namely, consider the functor
$$T(-) := \IndCoh(X) \otimes_{\QCoh(Y)} -: \QCoh(Y)\mh\cat{mod} = \IndCoh(Y)\mh\cat{mod} \longrightarrow \IndCoh(Z)\mh\cat{mod}.$$
Note that the $\QCoh(Y)$-action on $\QCoh(X) = \IndCoh(X)$ via pullback commutes with the $\IndCoh(Z)$-action by convolution.  
This functor defines a functor on horizontal traces:\footnote{Note that, as discussed in Example \ref{ssupp example}, $\IndCoh_{\{0\}_{\cL_\phi Y}}(\cL_\phi Y) = \QCoh(\cL_\phi Y)$.}
$$\Tr(T, \phi_*): \Tr(\QCoh(Y), \phi_*) = \IndCoh_{\{0\}_{\cL_\phi Y}}(\cL_\phi Y) \longrightarrow \Tr(\IndCoh(Z), \phi_*) = \IndCoh_{\Lambda_{X/Y,\phi}}(\cL_\phi Y).$$
By definition,
$$[\IndCoh(X), \phi_*] = \Tr(T, \phi_*)([\QCoh(Y), \phi_*]) = \Tr(T, \phi_*)(\cO_{\cL_\phi Y}).$$

\begin{rmk}
A variant of the functor $T$ for quasi-coherent sheaves, and in the setting where $f: X \rightarrow Y$ is surjective, was studied in \cite{BFN appendix}.  Note that unlike in their setting, this functor $T$ is not an equivalence since we are considering ind-coherent sheaves $\IndCoh(Z)$ rather than quasi-coherent sheaves $\QCoh(Z)$.  Furthermore, the failure of $f$ to be surjective in our setting requires the application of local cohomology in the calculation of its trace. 
\end{rmk}

\medskip

 We now identify the trace of the standard representation.
\begin{prop}\label{trace of standard} 
Define the singular support condition $\{0\}_{f(X)} := \{0\}_{\cL_\phi Y} \cap \Lambda_{X/Y, \phi}$.  There is a canonical identification of functors
$$\Tr(T, \phi_*) \simeq \iota_{{\{0\}_{f(X)}}} \circ \Gamma_{{\{0\}_{f(X)}}}: \IndCoh_{\{0\}_{\cL_\phi Y}}(\cL_\phi Y) \rightarrow \IndCoh_{\Lambda_{X/Y,\phi}}(\cL_\phi Y).$$
Furthermore, letting $\mathrm{ev}^{-1} f(X) \subset \cL_\phi Y$ corresponding to $\{0\}_{f(X)}$, we have
$$[\IndCoh(X), \phi_*] \simeq \Gamma_{\mathrm{ev}^{-1} f(X) }(\omega_{\cL_\phi Y}).$$
\end{prop}
\begin{proof}
We claim that the right dual to $T$ is
$$T^R(-) := \IndCoh(X) \otimes_{\IndCoh(Z)} -: \IndCoh(Z)\mh\cat{mod} \rightarrow \IndCoh(Y)\mh\cat{mod}$$
where $\IndCoh(X)$ here is considered as a \emph{right} $\IndCoh(Z)$-module, so that we have
$$T^R \circ T(-) = (\IndCoh(X) \otimes_{\IndCoh(Z)} \IndCoh(X)) \otimes_{\IndCoh(Y)} - \simeq \IndCoh_{f(X)}(Y) \otimes_{\QCoh(Y)} -,$$
$$T \circ T^R(-) = (\IndCoh(X) \otimes_{\IndCoh(Y)} \IndCoh(X)) \otimes_{\IndCoh(Z)} - \simeq \IndCoh_{\{0\}_Z}(Z) \otimes_{\IndCoh(Z)} -.$$
The convolution $\QCoh(Y)$-action can be re-interpreted as the usual pullback and tensor product, while the $\IndCoh(Z)$-action is by convolution.  
The first isomorphism is due to Proposition \ref{conv comp}, whereby
$$\IndCoh(X) \otimes_{\IndCoh(Z)} \IndCoh(X) \simeq \IndCoh_{f(X)}(Y)$$
i.e. the full subcategory of $\IndCoh(Y) = \QCoh(Y)$ with classical support on the closed subset $f(X)$ (since $Y$ is smooth there are no possible singular codirections).  The second isomorphism is due to Theorem 4.7 of \cite{BFN}, i.e.  we have $\IndCoh(X) \otimes_{\IndCoh(Y)} \IndCoh(X) = \QCoh(Z) = \IndCoh_{\{0\}_Z}(Z)$.

To establish duality, we need to write down unit and counit maps
$$\eta: \IndCoh(Y) \longrightarrow \IndCoh(X) \otimes_{\IndCoh(Z)} \IndCoh(X) \simeq \IndCoh_{f(X)}(Y),$$
$$\epsilon: \IndCoh_{\{0\}_Z}(Z) \simeq \IndCoh(X) \otimes_{\IndCoh(Y)} \IndCoh(X) \rightarrow \IndCoh(Z)$$
satisfying the usual ``Zorro's identities''.    We define $\eta := \Gamma_{f(X)}$ to be the local cohomology functor, and $\epsilon = \iota_{\{0\}_Z}$ to be the fully faithful inclusion.  The verification of Zorro's identities is immediate from the observation that tensoring $\eta$ or $\epsilon$ with $\mathrm{id}_{\IndCoh(X)}$ (on either side) gives rise to the identity functor, i.e. that the following diagrams commute:
$$\begin{tikzcd}
\IndCoh(X) \tens{\IndCoh(Y)} \IndCoh(Y) \arrow[r, "\simeq"] \arrow[d, "{\mathrm{id}_{\IndCoh(X)} \otimes \eta}"'] & \IndCoh(X) \arrow[d, "\mathrm{id}_{\IndCoh(X)}"] & \IndCoh_{\{0\}_Z}(Z) \tens{\IndCoh(Z)} \IndCoh(X) \arrow[d, "{\epsilon \otimes \mathrm{id}_{\IndCoh(X)}}"'] \arrow[r, "\simeq"] & \IndCoh(X) \arrow[d, "{\mathrm{id}_{\IndCoh(X)}}"] \\
\IndCoh(X) \tens{\IndCoh(Y)} \IndCoh_{f(X)}(Y) \arrow[r, "\simeq"] & \IndCoh(X) &\IndCoh(Z) \otimes_{\IndCoh(Z)} \IndCoh(X) \arrow[r, "\simeq"]& \IndCoh(X) 
\end{tikzcd}$$
This follows by Proposition \ref{conv comp} and the singular support calculations (note that $X$ is smooth and thus $\TT_X$ has no singular codirections):
$$\{0\}_X \ast f(X) = \{0\}_X, \;\;\;\;\;\; \{0\}_Z \ast \{0\}_X = \{0\}_X.$$
This establishes the duality of $(T, T^R)$.

Now, we compute the map on traces, using the functoriality described in Section 3.5 of \cite{BN:NT}.  There is a canonical commuting structure $\psi: T \circ \phi_{Y*} \rightarrow \phi_{Z*} \circ T$, which for us is an equivalence (thus induces an equivalence on traces).  We let $f(X) \subset \TT_Y = Y$ denote the (necessarily, since $Y$ is smooth) classical support condition, and define $\Lambda := \mathrm{ev}^!(f(X))$, i.e. the loops with base points classically supported over $f(X) \subset Y$ and no singular codirections.  We have $\{0\}_{\cL Y} \supset \Lambda \subset \Lambda_{X/Y}$.  
$$\begin{tikzcd}
\Tr(\IndCoh(Y), \phi_{Y*}) \arrow[d, "{\Tr(\IndCoh(Y), \eta \circ \mathrm{id}_{\phi_*})}"'] \arrow[r, "\simeq"] & \IndCoh_{\{0\}_{\cL_\phi Y}}(\cL_\phi Y) \arrow[d, "\Gamma_\Lambda \circ \iota_{\{0\}} = \Gamma_\Lambda"]\\
\Tr(\IndCoh(Y), T^R \circ T \circ \phi_{Y*}) \arrow[d, "{\Tr(\IndCoh(Y), \mathrm{id}_{T^R} \circ \psi)}"', "\simeq"] \arrow[r, "\simeq"] & \IndCoh_{\mathrm{ev}^!f(X)}(\cL_\phi Y) \arrow[d, equals] \\
\Tr(\IndCoh(Y), T^R \circ \phi_{Z*} \circ T) \arrow[d, "\simeq"] \arrow[r, "\simeq"]& \IndCoh_{\mathrm{ev}^!f(X)}(\cL_\phi Y)  \arrow[d, equals] \\
\Tr(\IndCoh(Z), \phi_{Z*} \circ T \circ T^R) \arrow[d, "{\Tr(\IndCoh(Z),\mathrm{id}_{\phi_*}\circ \epsilon)}"'] \arrow[r, "\simeq"]& \IndCoh_{\delta_*\pi^! \{0\}_Z}(\cL_\phi Y)  \arrow[d, "\Gamma_{\Lambda_{X/Y}} \circ \iota_\Lambda = \iota_\Lambda"] \\
\Tr(\IndCoh(Z), \phi_{Z*}) \arrow[r, "\simeq"]& \IndCoh_{\Lambda_{X/Y}}(\cL_\phi Y)
\end{tikzcd}$$

The top and bottom isomorphisms are given by Theorem 3.3.1 in \cite{BNP}.  We argue the middle isomorphisms.  A combination of the arguments of Propositions \ref{convolution category trace} and \ref{conv comp} gives rise to identifications
$$\Tr(\IndCoh(Z), T \circ T^R \circ \phi_{Y*}) = \IndCoh(Y) \otimes_{\IndCoh(Y \times Y)} \IndCoh_{f(X)}(Y) \simeq \IndCoh_{\mathrm{ev}^!(f(X))}(\cL_\phi Y),$$
$$\Tr(\QCoh(Y), \phi_{Z*} \circ T^R \circ T) = \IndCoh(Z) \otimes_{\IndCoh(Z \times Z)} \IndCoh_{\{0\}_Z}(Z)  \simeq \IndCoh_{\delta_*\pi^!\{0\}_Z}(\cL_\phi Y),$$
where $\delta_*\pi^!\{0\}_Z$ is the pull-push of $\{0\}_Z$ along the correspondence in Theorem \ref{convolution category trace}.  We note that $\delta_*\pi^!\{0\}_Z = \delta_* \{0\}_{\cL_\phi Y_X} =  \mathrm{ev}^!f(X) = \{0\}_{f(X)}$.  The identification of the vertical functors follows via the functoriality of supports in Proposition \ref{conv comp} applied to the setting of Proposition \ref{convolution category trace}, and the observation that $\{0\}_{\cL_\phi Y} \supset \Lambda \subset \Lambda_{X/Y,\phi}$.  This establishes the first statement of the theorem.

For the second statement, note that $\omega_{\cL_\phi Y}$ is perfect (since $\cL Y$ is quasi-smooth), i.e. has no singular codirections.   In general, for singular support conditions $\Lambda_1, \Lambda_2 \subset \TT_X$, we have $\Gamma_{\Lambda_2} \circ \iota_{\Lambda_1} \circ \Gamma_{\Lambda_1} = \Gamma_{\Lambda_1 \cap \Lambda_2}.$  Now, take $\Lambda_1 = \{0\}_{\cL Y}$ (i.e.  no singular codirections with unrestricted classical support) and $\Lambda_2 = \mathrm{ev}^{-1}f(X) \times_{\cL_\phi Y} \TT_{\cL_\phi Y}$ (i.e. all singular codirections with restricted classical support).  The second statement follows, since $\Gamma_{\Lambda_1}(\omega_{\cL_\phi Y}) = \omega_{\cL_\phi Y}$ and $\Gamma_{\Lambda_2}$ is the classical local cohomology functor with support $\mathrm{ev}^{-1} f(X)$.
\end{proof}

\begin{cor}\label{standard module hh}
The functor
$$\Hom(\cS_{X/Y, \phi}, -):  \Tr(\IndCoh(Z), \phi_*) \simeq \IndCoh_{\mathrm{ev}^{-1}f(X)}(\cL_\phi Y)  \longrightarrow \End(\cS_{X/Y,\phi})\dmod  $$
takes $\Gamma_{\mathrm{ev}^{-1}f(X)}(\cO_{\cL_\phi Y})$ to the $\hh(\IndCoh(Z), \phi_*)$-module $\hh(\IndCoh(X), \phi_*)$.
\end{cor}
\begin{proof}
By Theorem \ref{trace theorem}, it suffices to identify the trace of the $\IndCoh(Z)$-module category $\IndCoh(X)$.  By the above theorem, $[\IndCoh(X), \phi_*] \simeq \Gamma_{\mathrm{ev}^{-1}f(X)}(\omega_{\cL_\phi Y}) \simeq \Gamma_{\mathrm{ev}^{-1}f(X)}(\cO_{\cL_\phi Y})$ (the latter isomorphism by Proposition \ref{calabi yau}).  
\end{proof}

\begin{rmk}
Note that it is immediate via adjunctions that
$$\Hom_{\cL_\phi Y}(\cS_{X/Y,\phi}, \Gamma_{\mathrm{ev}^{-1}f(X)}\cO_{\cL_\phi Y}) \simeq \Hom_{\cL_\phi X}(\cO_{\cL_\phi X}, \cO_{\cL_\phi X}) \simeq \cO(\cL_\phi X) \simeq HH(\IndCoh(X), \phi_*).$$
By working at the level of categorical traces, we automatically deduce that this is an identification as $HH(\IndCoh(Z), \phi_*)$-modules.
\end{rmk}

\medskip

\subsubsection{Splitting the universal trace sheaf}\label{app enhanced}

The coherent Springer sheaf $\cS_{X/Y, \phi}$ may be realized as the character of the regular $\IndCoh(Z)$-representation $[\IndCoh(Z), \phi_*]$, but also as the character of the $\QCoh(Y)$-representation $[\QCoh(X), \phi_*]$.  In this section, we will take the latter point of view.  This allows us to do something sneaky in the proof of Theorem \ref{thm endomorph}: we swap out $X = \wt{\cN}/G$ with $\wt{\mf{g}}/G$, and use the observation that their $q$-fixed points are canonical equivalent for $q$ not a root of unity.  There is a canonical map 
$$[\IndCoh(X), \phi_*] = \cS_{X/Y,\phi} = \cL_\phi f_* \omega_{\cL_\phi X} \longrightarrow \omega_{\cL_\phi Y} = [\IndCoh(Y), \phi_*]$$
arising via the pushforward of volume forms.  In this section we investigate when this map splits, realizing the trace of the standard representation as a summand of the trace of the regular representation.  To do so we require a discussion of enhanced vertical traces, i.e.  the realization of vertical traces of module categories for a monoidal category as characters in the horizontal trace of the monoidal category.

\begin{defn}
Let us fix a monoidal dg category $\cat{A}$, and a monoidal endofunctor $F$.  For any $\cat{A}$-module category $\cat{C}$ equipped with a commuting structure $F_{\cat{M}}$ for $F$ (see Definitions \ref{character def} and \ref{two trace maps}), we define the \emph{enhanced Hochschild homology} to be 
$$\underline{\hh}(\cat{C}, F_{\cat{M}}) := [\cat{C}, F_{\cat{M}}] \in \Tr(\cat{A}, F).$$  
By Theorem \ref{trace theorem}, the usual Hochschild homology can be recovered by applying the functor $\Hom_{\Tr(\cat{A}, F)}([\cat{A}, F], -)$.
\end{defn}

\begin{rmk}
We have seen examples of this enhanced Hochschild homology in Section \ref{HH geometric}, namely that in geometric settings Hochschild homology and maps induced by functoriality often sheafify, i.e. arise as global objects via local ones by taking global sections.    The category $\QCoh(Y)$ is monoidal, and for any module category $\cat{C}$ the Hochschild homology $HH(\cat{C}) := [\cat{C}] \in \cat{Vect}_k$ has an enhancement $\underline{\hh}(\cat{C}) \in \Tr(\QCoh(Y)) = \QCoh(\cL Y).$     Though we do not need or prove it, the enhanced Block-Getzler complex in Definition \ref{alg model} is also an example of this phenomenon, where we view the Hochschild homology of a $\Rep(G)$-module category as an object of $\Tr(\Rep(G)) = \QCoh(G/G)$.    
\end{rmk}

We now compute the enhanced trace in an example of interest; see Appendix \ref{app horiz} for a proof.
\begin{prop}\label{qcoh trace}
Let $f: X \rightarrow Y$ be a map of QCA (or more generally, perfect) stacks, and $\phi_X, \phi_Y$ compatible self-maps such that $\phi_{Y*}: \QCoh(Y) \rightarrow \QCoh(Y)$ is monoidal and $\phi_{X*}: \QCoh(X) \rightarrow \QCoh(X)$ is $\phi_{Y*}$-semilinear.  Consider $\QCoh(X)$ as a $\QCoh(Y)$-module category.  Then, we have
$$\underline{\hh}(\QCoh(X), \phi_{X*}) = [\QCoh(X), \phi_{X*}] \simeq \cL_\phi f_*\cO_{\cL_\phi X} \in \Tr(\QCoh(Y), \phi_{Y*}) = \QCoh(\cL_\phi Y).$$ 
\end{prop}

\medskip

We now establish the desired splitting.  We note that in the below, we take the category $\QCoh$ of quasi-coherent sheaves rather than ind-coherent sheaves.
\begin{prop}\label{splitting}
Let $f: X \rightarrow Y$ be a proper morphism of smooth QCA stacks, with compatible self-maps $\phi_X, \phi_Y$, such that $\phi_{Y*}$ is a monoidal endofunctor of $\QCoh(Y)$ and $\phi_{X*}$ is a $\phi_{Y*}$-semilinear endofunctor of $\QCoh(X)$.  Furthermore, assume that $f_*\cO_X \simeq \cO_{Y} \otimes_k V$ for a $\phi_*$-equivariant vector space $V$ such that $\mathrm{tr}(\phi_*, V) \ne 0$.  Then, $[\QCoh(X), \phi_{X*}] \simeq \cL_\phi f_* \cO_{\cL_\phi X}$ contains $[\QCoh(Y), \phi_{Y*}] \simeq \cO_{\cL_\phi Y}$ as a summand in $\Tr(\QCoh(Y), \phi_{Y*}) = \QCoh(\cL_\phi Y)$.
\end{prop}
\begin{proof}[Proof of Proposition \ref{splitting}]
To prove the claim, we need to produce a splitting. Since the character of the monoidal unit in $\QCoh(Y)$ is the monoidal unit in $\underline{HH}(\QCoh(Y), \phi_*)$, applying Proposition \ref{qcoh trace} we obtain a diagram (where $\ni$ means an element of global sections, i.e. non-enhanced Hochschild homology):
$$\begin{tikzcd}[column sep=small]
\underline{HH}(\QCoh(Y), \phi_*) \arrow[r, equals] \arrow[d, "{\underline{HH}(f^*, \phi_*)}"'] & \cO_{\cL_\phi Y} \arrow[d]  & \arrow[l, phantom, "\ni"'] 1_{\cL_\phi Y} \arrow[d, mapsto] & \arrow[l, equals] [\cO_Y, \phi_*]\arrow[d, mapsto]  \\
\underline{HH}(\QCoh(X), \phi_*) \arrow[r, equals] \arrow[d, "{\underline{HH}(f_*, \phi_*)}"'] &  \cL_\phi f_* \cO_{\cL_\phi X} \arrow[d] & \arrow[l, phantom, "\ni"'] 1_{\cL_\phi X} \arrow[d, mapsto] &  \arrow[l, equals] [f^*\cO_Y, \phi_*] = [\cO_X, \phi_*] \arrow[d, mapsto] \\
\underline{HH}(\QCoh(Y), \phi_*)\arrow[r, equals] & \cO_{\cL_\phi Y}& \arrow[l, phantom, "\ni"'] \mathrm{tr}_{\cO(Y)}(\phi_*, V) \cdot 1_{\cL_\phi Y}& \arrow[l, equals] [f_*\cO_X, \phi_*] = [\cO_Y \otimes V, \phi_*].
\end{tikzcd}$$
Note that $f^*$ always preserves perfect objects, and $f_*$ preserves perfect objects since $f$ is proper and $X$ and $Y$ are smooth, giving us the functoriality on the left following Proposition \ref{qcoh trace}.  To see that the composition is an isomorphism, note that a map $\cO_{\cL_\phi Y} \rightarrow \cO_{\cL_\phi Y}$ is determined by where the constant function maps; by the above, it maps to $[f_*, \cO_X, \phi_*] \in HH(\QCoh(Y), \phi_*)$.  We will show this is a unit, thus the composition of the arrows on the left is an isomorphism.

To this end, let $p: Y \rightarrow \Spec k$ be the ($\phi$-equivariant) projection to a point where $\phi$ acts trivially on $\Spec k$, and note that $f_* \cO_X \simeq p^* E$.  In particular,  $[p^* V, \phi_*]$ is the image of $[V, \phi_*] = \mathrm{tr}(V, \phi_*)$ under the pullback map $\cO(\cL_\phi p^*): k \rightarrow \cO(\cL_\phi Y)$, which is a non-zero multiple of the identity by assumption, thus a unit as required.
\end{proof}

\section{The affine Hecke algebra and the coherent Springer sheaf}\label{section coherent springer}

We now specialize the discussion of Section \ref{traces} to our Springer theory setting.   In this section, we will take $k = \bQl$ or $\C$.  We are interested in the following special cases.
\begin{defn}[Coherent Springer sheaves]\label{def q loop space}
Recall that $\GG = G \times \G_m$, and the set-up in Definition \ref{conv setup} and the universal trace sheaf of Definition \ref{univ trace sheaf}.
\begin{enumerate}
\item We take 
$$f = \mu: X = \wt{\cN}/\GG \longrightarrow \wh{\cN}/\GG \hookrightarrow Y = \mf{g}/\GG$$
to be the scaling-equivariant Springer resolution (with codomain in the Lie algebra rather than the nilpotent cone).  We call the resulting sheaf $\cS$ on $\cL(\wh{\cN}/\GG)$ (or equivalently, on $\cL(\mf{g}/\GG)$ supported over $\cN$) the \emph{coherent Springer sheaf}.
\item We take 
$$f = \mu: X = \wt{\cN}/G \longrightarrow \wh{\cN}/G \hookrightarrow Y = \mf{g}/G$$ to be the above Springer resolution without $\G_m$-equivariance, and $\phi := q$ to be multiplication by $q \in \G_m(k)$.  Then we have the derived $q$-fixed points:
$$\cL_q(\wh{\cN}/G) \simeq \cL(\wh{\cN}/\GG) \times_{\cL(B\G_m)} \{q\}.$$
This is the stack $\bL^u_{q,G}$ from the introduction.  We call the sheaf $\cS_q$ on $\cL_q(\wh{\cN}/G)$ the \emph{coherent $q$-Springer sheaf}.
\end{enumerate} 
\end{defn}

We note the following convenient presentation of the stacks $\LL(\NN/\GG)$ and $\LL(\wh{\cN}/\GG)$.
\begin{rmk}\label{Lq id}
We realize $\LL(\wh{\cN}/\GG)$ as the formal completion of $\LL(\mf{g}/\GG) \rightarrow \mf{g}/\GG$ over the nilpotent cone.  By Proposition 2.1.8 of \cite{Ch}, we can write $\LL(\mf{g}/\GG)$ as the pullback
$$\begin{tikzcd}
\LL(\mf{g}/\GG) \arrow[d] \arrow[r]  &\mf{g}/\GG \arrow[d, "\Delta"] \arrow[r] & \{0\}/\GG \arrow[d, hook] \\
 (\mf{g} \times \GG)/\GG \arrow[r, "a \times p"]  &  (\mf{g} \times \mf{g})/\GG \arrow[r, "-"] & \mf{g}/\GG
\end{tikzcd}$$
where the bottom right map is given by subtraction in $\mf{g}$, $a$ is the action map, $p$ the projection, and $\Delta$ the diagonal.  Explicitly, the map $\mf{g} \times \GG \rightarrow \mf{g}$ is given by $(x, g, q) \mapsto q^{-1} \mathrm{Ad}_g(x) - x$.  We also have a version for fixed $q$:
$$\begin{tikzcd}
\LL_q(\mf{g}/G) \arrow[d] \arrow[r]  &\mf{g}/G \arrow[d, "\Delta"] \arrow[r] & \{0\}/\GG \arrow[d, hook] \\
 (\mf{g} \times G)/G \arrow[r, "a_q \times p"]  &  (\mf{g} \times \mf{g})/\GG \arrow[r, "-"] & \mf{g}/\GG.
\end{tikzcd}$$
where $a_q$ is the $q$-twisted action map.  There is a similar description for $\LL(\NN/\GG) = \LL(\mf{n}/\wt{B})$:
$$\begin{tikzcd}
\LL(\NN/\GG) \arrow[d] \arrow[r]  & \NN/\GG \arrow[d, "\Delta"] \arrow[r] & (G/B)/\GG \arrow[d, hook] \\
 (\NN \times \GG)/\GG \arrow[r, "a \times p"]  &  (\NN \times \NN)/\GG \arrow[r, "-"] & \NN/\GG.
\end{tikzcd}$$
\end{rmk}

We record the following mild generalization and direct consequence of Proposition 4.2 in \cite{curtis} and Proposition 2.1 in \cite{hellmann} (also proven for $q$ a prime power in Proposition 3.1.5 of \cite{xinwen survey}).  In particular, when $q$ is not a root of unity (e.g. for arithmetic applications), we may replace $\wh{\cN}$ or $\mf{g}$ with $\cN$ in the $q$-twisted loop spaces.
\begin{prop}\label{no derived structure}
If $q$ is not a root of unity, then $\LL_q(\wh{\cN}/G)$ is a classical stack, i.e. has trivial derived structure and is supported at the nilpotent cone.  The maps 
$$\cL_q(\cN/G) \longrightarrow \cL_q(\wh{\cN}/G) \longrightarrow \cL_q(\mf{g}/G)$$
are isomorphisms of classical (but a priori derived) stacks.
\end{prop}
\begin{proof}
We first argue that $\LL_q(\mf{g}/G)$ is supported over the nilpotent cone, thus $\LL_q(\mf{g}/G) = \LL_q(\wh{\cN}/G)$.  The formation of (twisted) loop spaces commutes with products; note the Cartesian square
$$\begin{tikzcd}
\cN/G \arrow[r] \arrow[d]  & \mf{g}/G \arrow[d] \\
\{0\} \arrow[r] & \mf{h}//W.
\end{tikzcd}$$
The morphisms are $\G_m$-equivariant, where $\G_m$ acts on $\mf{h}$ by weight 1, and on $\mf{h}//W$ by weights $\geq 1$.  Thus if $q$ is not a root of unity, then the (derived and classical) $q$-fixed points of $\mf{h}//W$ is precisely $\{0\}$.  Thus the map on the bottom is an equivalence, and the claim follows.  The vanishing of derived structure follows by Proposition 4.2 in \cite{curtis} and in view of Remark 2.2(b) of \cite{hellmann}.
\end{proof}

\begin{rmk}
It is necessary to exclude roots of unity; when $G = \SL_2$, the weight of $\mf{h}//W$ is 2, so the argument fails for $q = \pm 1$.  When $G = \SL_3$, the weights of $\mf{h}//W$ are 2 and 3, so the argument fails for $q = \pm 1$ and any cubic root of unity.  A sharper statement is possible: for a fixed group $G$, the proposition is true if we avoid roots of unity with order dividing any fundamental invariant of $\mf{g}$.  The statements also hold for $G$ a parabolic subgroup, except that $\cL_q(\cN_P/P)$ may fail to be a classical stack (i.e. may have derived structure).
\end{rmk}

\medskip

We now give an alternative characterization of the coherent Springer sheaf (and likewise for the $q$-version) via coherent parabolic induction.
\begin{defn}
Consider the parabolic induction correspondence
$$\begin{tikzcd}
\wh{\cN}/\GG & \arrow[l, "\mu"'] \wh{\mf{n}}/\Bgr \arrow[r, "\nu"] & \wh{\{0\}}/\Hgr.
\end{tikzcd}$$
We define the \emph{coherent Springer sheaf} by applying the loop space of the above correspondence to the reduced structure sheaf of $\cL(\{0\}/\Hgr)$:
$$\cS := \cL\mu_* \cO_{\wt{\cN}/\GG} = \cL\mu_* \cL\nu^* \cO_{\cL(\{0\}/\Hgr)} \in \Coh(\cL(\wh{\cN}/\GG)).$$
We define the \emph{coherent $q$-Springer sheaf} analogously, or equivalently we can take $\cS_q := \iota_q^* \cS$, where $\iota_q: \cL_q(\wh{\cN}/G) \rightarrow \cL(\wh{\cN}/\GG)$.
\end{defn}
\begin{rmk}
Note that a priori, one could define $\cS_q$ via either the $*$ or $!$-pullback.  However, the map $\iota_q$ is base-changed from the map $i_q: \{q\} \rightarrow \G_m/\G_m$.  Since $\{q\} \subset \G_m$ has trivial normal bundle and $i_q$ has relative dimension zero, we have a canonical equivalence $\iota_q^! \simeq \iota_q^*$, i.e. it did not matter which definition we took.  Likewise, since derived loop spaces of smooth stacks (or smooth morphisms) are Calabi Yau by Proposition \ref{calabi yau}, we have an equivalence $\cL \nu^* \simeq \cL\nu^!$ and can use either.
\end{rmk}

\medskip

For number theory applications, we will be interested in specializing at $q$ a prime power.  These are the algebraic specializations of the affine Hecke algebra, which have no derived structure since $\Haff$ is flat over $k[z,z^{-1}]$.
\begin{defn}
We define the \emph{Iwahori-Hecke algebra} by 
$$\Haff_q := \Haff \otimes_{k[z,z^{-1}]} k[z,z^{-1}]/\langle z - q\rangle.$$  
\end{defn}

A potentially different algebra arises when specializing geometrically, i.e. taking endomorphisms of a $q$-specialized Springer sheaf.  We introduce the following unmixed version of the affine Hecke algebra, which is obtained by taking $G$-equivariant endomorphisms of the Springer sheaf \emph{without} taking $\G_m$-invariants, i.e. by passing to the base changed stack $\cL(\wh{\cN}/\GG) \times_{B\G_m} \pt$.

\begin{defn}
Let $\cL^{un}(\wh{\cN}/\GG) := \cL(\wh{\cN}/\GG) \times_{B\G_m} \pt$.  We define the \emph{unmixed affine Hecke algebra} and its specialization by
$$\Haff^{\un} := \End_{\LL^{un}(\wh{\cN}/\GG)}(\cS), \;\;\;\;\;\;\; \Haff^{un}_q := \Haff^{\un} \otimes^L_{k[z,z^{-1}]} k[z,z^{-1}]/\langle z - q\rangle.$$  
The algebra $\Haff^{\un}$ has the additional structure of a $\G_m$-representation, i.e. a weight grading.  
\end{defn}


The unmixed affine Hecke algebra arises naturally when considering the trace by pullback by various $q \in \G_m$ acting on the affine Hecke category $\Hcat = \Coh(\cZ/G)$ (as opposed to the mixed affine Hecke category $\Hcat^\mix = \Coh(\cZ/\GG)$).
\begin{prop}\label{unmixed ident}
There is a natural equivalence of algebras 
$$\cH_q^{\un} \simeq HH(\Hcat, q_*) \simeq \End_{\cL_q(\wh{\cN}/G)}(\cS_q).$$ 
That is,
$$\cH_q^{\un} \simeq \begin{cases} kW_a \otimes_k \Symp_k(\mf{h}^*[-1] \oplus \mf{h}^*[-2]) & \text{ when } q=1, \\ 
\cH_q & \text{ when } q\ne 1.\end{cases}$$
\end{prop}
\begin{proof}
We adopt the shorthand notation $\cS^{\un}$ for the corresponding coherent Springer sheaf on $\cL^{un}(\wh{\cN}/\GG)$.    Let $\iota_q: \cL_q(\wh{\cN}/G) \hookrightarrow \cL^{\un}(\wh{\cN}/\GG)$ be the base change along the closed immersion $\{q\} \hookrightarrow \G_m$.  Consider the forgetful functor for the natural map of algebras 
$$\cH^{\un} = \End_{\LL^{\un}(\wh{\cN}/\GG)}(\cS^{\un}) \rightarrow \Hom_{\LL_q(\wh{\cN}/G)}(\iota_q^* \cS^{\un}, \iota_q^* \cS^{\un}) = HH(\cat{H}, q_*).$$
obtained via functoriality (Proposition \ref{functoriality BG}).  Using the $(\iota_q^*, \iota_{q,*})$ adjunction, we have $\iota_{q,*} \iota_q^* \cF = \cone(q: \mathcal{F} \rightarrow \mathcal{F})$, and an equivalence of complexes
$$\begin{tikzcd}\Hom_{\LL^{un}(\wh{\cN}/\GG)}(\cS^{un}, \iota_{q,*} \iota_q^* \cS^{un}) & \arrow[l, "\simeq"'] \Hom_{\LL^{un}(\wh{\cN}/\GG)}(\cS^{un}, \cS^{un}) = \Haff^{un}  \\
& \Hom_{\LL^{un}(\wh{\cN}/\GG)}(\cS^{un},\cS^{un}) = \Haff^{un}. \arrow[u, "q"]
\end{tikzcd}$$
The equivalence is an equivalence of dg algebras, so $HH(\cat{H}, q_*) \simeq \cH^{un}_q$.  Finally, we have $HH(\cat{H}, q_*) \simeq \cH_q$  by Corollary \ref{trace q id} (and the identification for $q=1$ by Corollary \ref{ThmHH no Gm}), proving the claim.
\end{proof}

\begin{rmk}\label{unmixed id}
The algebra $\cH^{un}$ can be recovered as the $\G_m$-enhanced Hochschild homology of $\Hcat^\mix$ discussed in \cite{GKRV} and Section \ref{app enhanced}.  In particular, take coordinates $\cO(\G_m) = k[z,z^{-1}]$, let $\mf{h}[-n]$ denote the shifted Cartan algebra in cohomological-weight bidegree $(n, 1)$, and define the graded $k[z,z^{-1}]$ algebra
$$A^{[-n]} := \cO(\cL(\mf{h}^*[n]/\G_m) = \Sym_{\cO(\G_m)}(\mf{h}[-n] \otimes_k \cO(\G_m)) / \langle x(z - 1) \mid x \in \mf{h}[-n] \rangle.$$
One can compute (in a similar manner as Corollaries \ref{ThmHH no Gm} and \ref{trace q id}) that
$$\cH^{un} = \underline{\hh}^{\G_m}(\mathbf{H}^\mix) = \cH \otimes_{\cO(\G_m)} A^{[-2]}$$
recovering the above proposition on specialization at various $z = q$.  One can do the same for the variants in Remark \ref{hecke variants}, i.e.
$$\underline{\hh}^{\G_m}(\Coh(\cZ'/\GG)) = \cH, \;\;\;\;\;\;\; \underline{\hh}^{\G_m}(\Coh(\cZ^\wedge\!/\GG)) = \cH \otimes_{\cO(\G_m)} A^{[-1]}.$$
Note that Theorem 4.4.4 in \emph{op. cit.} establishes a relationship similar to this one.
\end{rmk}

\begin{rmk}
One can similarly argue that $\cH_q$ can be realized as the endomorphisms of the restriction of $\cS$ along the base change of the inclusion $\{q\}/\G_m \hookrightarrow \cL(B\G_m)$, i.e. where we retain $\G_m$-equivariance.
\end{rmk}

Our main result is the following theorem (see Proposition \ref{no derived structure}).
\begin{thm}\label{thm endomorph}
Assume that $q \ne 1$.  
\begin{enumerate}
\item The dg algebra of endomorphisms of the coherent Springer sheaf is concentrated in degree zero and is identified with the affine Hecke algebra, 
$$\End_{\LL(\wh{\cN}/\GG)}(\cS) \simeq \Haff, \;\;\;\;\;\;\;\;\;\;\;\;\;\; \End_{\cL_q(\wh{\cN}/G)}(\cS_q) \simeq \Haff_q.$$
In particular, $\cS$ generates full embeddings, the \emph{Deligne-Langlands functors}:
$$\mathrm{DL}: \Haff\module \hookrightarrow \IndCoh(\cL(\wh{\cN}/\GG)),\;\;\;\;\;\;\; \mathrm{DL}_q: \Haff_q\module \hookrightarrow \IndCoh(\cL_q(\wh{\cN}/G)).$$
\item On the anti-spherical modules $M^{\mathrm{asp}} := \Ind_{\cH^f}^{\cH} (\mathrm{sgn})$ and $M^{\mathrm{asp}}_q := \Ind_{\cH^f_q}^{\cH_q}(\mathrm{sgn})$, these functors take values
$$\mathrm{DL}(M^{\mathrm{asp}}) \simeq \mathrm{pr}_{\cS}(\omega_{\cL(\wh{\cN}/\GG)}), \;\;\;\;\;\;\; \mathrm{DL}_q(M^{\mathrm{asp}}_q) \simeq \mathrm{pr}_{\cS_q}(\omega_{\cL_q(\wh{\cN}/G)}),$$ 
where $\mathrm{pr}_{\cS} = \mathrm{DL} \circ \mathrm{DL}^R$ (resp. $\mathrm{pr}_{\cS_q} = \mathrm{DL}_q \circ \mathrm{DL}_q^R$), i.e. the composition of the Deligne-Langlands functor with its right adjoint.  Furthermore, when $q$ is not a root of unity,
$$\mathrm{DL}_q(M^{\mathrm{asp}}_q) \simeq \mathrm{pr}_{\cS_q}(\omega_{\cL_q(\wh{\cN}/G)}) = \omega_{\cL_q(\cN/G)} \simeq \cO_{\cL_q(\cN/G)}$$
and $\cO_{\cL_q(\cN/G)}$ is a summand of $\cS_q$.

\item These embeddings are compatible with parabolic induction, i.e. for a parabolic $P \supset B$ with quotient Levi $M$, we have commuting diagrams
$$\begin{tikzcd}
\Haff_{M}\dmod \arrow[d, "{\cH_G \otimes_{\cH_M} -}"']\arrow[r, hook] & \IndCoh(\cL(\wh{\cN}_M/M_{\ggr}))\arrow[d, "\cL\mu_* \circ \cL\nu^*"] & & \Haff_{q,M}^{un}\dmod \arrow[d, "{\cH^{}_{q,G} \otimes_{\cH^{}_{q,M}} -}"']\arrow[r, hook] & \IndCoh(\cL_q(\wh{\cN}_M/M))\arrow[d, "\cL_q \mu_* \circ \cL_q \nu^*"]\\
\Haff_{G}\dmod \arrow[r, hook] &  \IndCoh(\cL(\wh{\cN}_G/\GG)) & & \Haff^{}_{q,G}\dmod \arrow[r, hook] &  \IndCoh(\cL_q(\wh{\cN}_G/G)).
\end{tikzcd}$$
That is, the parabolic induction functor is the pull-push along the correspondence obtained by applying $\cL$ or $\cL_q$ to the usual correspondence
$$\begin{tikzcd}
\wh{\cN}_M/M_{\ggr} & \arrow[l, "\mu"'] \wh{\cN}_P/P_{\ggr} \arrow[r, "\nu"] & \wh{\cN}_G/\GG.
\end{tikzcd}$$
\end{enumerate}
\end{thm}
\begin{proof}
The first claim of the theorem is a combination of Theorems \ref{ThmHH} and Theorem \ref{traces convolution}, Corollaries \ref{trace q id} and \ref{ThmHH no Gm}, and Proposition \ref{unmixed ident}, for both general $q$ and specific $q$.  It remains to prove the claims regarding the anti-spherical module and compatibility with parabolic induction.

We first address the claim regarding anti-spherical modules.  
By Corollary \ref{standard module hh}, we have an equivalence as $\End(\cS) \simeq \hh(\Coh(\cZ/\GG))$-modules
$$\Hom(\cS, \omega_{\cL(\wh{\cN}/\GG)}) \simeq \hh(\Coh(\wt{\cN}/\GG)).$$
Thus,  it follows that $\mathrm{pr}_{\cS}(\omega_{\cL(\wh{\cN}/\GG)}) \simeq \hh(\Coh(\wt{\cN}/\GG)$ as $\hh(\Coh(\cZ/\GG))$-modules (and similarly for special $q$).  Thus, we need to compute the module $\hh(\Coh(\wt{\cN}/\GG))$ (and likewise for special $q$), and we need to identify the projection for $q$ not a root of unity.

We first produce an isomorphism $\hh(\Coh(\wt{\cN}/\GG)) \simeq M^{\mathrm{asp}}$ as $\hh(\Coh(\cZ/\GG))$-modules, and isomorphisms $\hh(\Coh(\wt{\cN}/G), q_*) \simeq M^{\mathrm{asp}}_q$ as $\hh(\Coh(\cZ/G), q_*)$-modules.  The first isomorphism follows via the identification of $K_0(\Coh(\wt{\cN}/\GG))$ as the anti-spherical module for $K_0(\Coh(\cZ/\GG))$ in Section 7.6 of \cite{CG}\footnote{In our convention, we identify $K_0(\Coh(\wt{\cN}/\GG))$ with the anti-spherical module, and $K_0(\Coh_{\cB/\GG}(\wt{\cN}/\GG))$ with the spherical module.} once we establish an equivalence $K_0(\Coh(\wt{\cN}/\GG)) \simeq \hh(\Coh(\wt{\cN}/\GG))$ as $K_0(\Coh(\wt{\cN}/\GG)) \simeq \hh(\Coh(\wt{\cN}/\GG))$-modules, and the second would follow from an equivalence $HH(\Coh(\wt{\cN}/G, q_*) \simeq HH(\Coh(\wt{\cN}/\GG)) \otimes_{k[\G_m]} k_q$ (similar to the identification in Proposition \ref{unmixed ident}).

To see this, note that $\Coh(\wt{\cN}/\GG)$ has a semiorthogonal decomposition indexed by $\lambda \in X^\bullet(H)$ characters of the quotient torus $H = B/[B,B]$, where each subcategory $\Coh(\wt{\cN}/\GG))_\lambda$ is generated over $\Rep(\G_m)$ by the line bundle $\cO_{\wt{\cN}/\GG}(\lambda)$.  Computing via the Block-Getzler complex of Definition \ref{alg model} (see also Corollary \ref{semiorthog barr beck}), and noting that $\End_{\wt{\cN}/\GG}(\cO_{\wt{\cN}/\GG}(\lambda)) = k$ we have that the specialization at $q$ map is:
$$\begin{tikzcd}
\hh(\Coh(\wt{\cN}/\GG))_\lambda) \arrow[r] \arrow[d, "\simeq"] & \underline{\hh}^{\G_m}(\Coh(\wt{\cN}/\GG))_\lambda)\arrow[r] \arrow[d, "\simeq"] & \hh(\Coh(\wt{\cN}/G))_\lambda, q_*) \arrow[d, "\simeq"] \\
\cO(\G_m) \arrow[r, equals] & \cO(\G_m) \arrow[r] & k_q.
\end{tikzcd}$$
The equivalence on the left induces an equivalence $K_0(\Coh(\wt{\cN}/\GG))_\lambda) \simeq \hh(\Coh(\wt{\cN}/\GG))_\lambda)$.  Summing over each subcategory in the semiorthogonal deomposition, this establishes both claims.

It remains to compute the projection $\mathrm{pr}_{\cS_q}(\omega_{\cL_q(\wh{\cN}/G)})$ for $q$ not a root of unity.  By Proposition \ref{no derived structure}, $\cL_q(\wh{\cN}/G) \simeq \cL_q(\cN/G)$; it suffices to show that $\omega_{\cL_q(\cN/G)} \simeq \cO_{\cL_q(\cN/G)}$ is a summand of $\cS_q$.  Since derived fixed points commutes with fiber products, the diagrams
$$\begin{tikzcd}
\cL_q(\cN/G) \arrow[r] \arrow[d] & \cL_q(\mf{g}/G) \arrow[d] &  \cL_q(\wt{\cN}/G) \arrow[r] \arrow[d] & \cL_q(\wt{\mf{g}}/G) \arrow[d]\\
\cL_q(\{0\}) \arrow[r] & \cL_q(\mf{h}//W)& \cL_q(\{0\}) \arrow[r] & \cL_q(\mf{h})
\end{tikzcd}$$
are Cartesian.  When $q$ is not a root of unity, we have $\cL_q(\{0\}) = \cL_q(\mf{h}//W) = \cL_q(\mf{h})$, so that $\cL_q(\cN/G) = \cL_q(\mf{g}/G)$ and $\cL_q(\wt{\cN}/G) = \cL_q(\wt{\mf{g}}/G)$.  We then apply Proposition \ref{splitting} to the Grothendieck-Springer resolution $\mu': \wt{\mf{g}}/G \rightarrow \mf{g}/G$ and $\phi = q_*$ to obtain the splitting, observing that $\mu'_*\cO_{\wt{\mf{g}}} \simeq \cO_{\mf{g}} \otimes_{\cO(\mf{h})^W} \cO(\mf{h})$, and that by the main theorem of \cite{demazure} $\cO(\mf{h})$ is a free graded $\cO(\mf{h})^W$-module of rank $|W|$ with homogeneous basis of degrees $-\ell(w)$ for $w \in W$ (in our sign convention), so we may further write 
$$\mu'_*\cO_{\wt{\mf{g}}} \simeq \cO_{\mf{g}} \otimes_k V$$ 
where $V \in \Rep(\bG_m)$ is the $k$-linear span of these $\bG_m$-eigenvector basis elements.  In particular the trace of the action of $q_*$ on the free $\cO(\mf{h})^W$-module $\cO(\mf{h})$ is the Poincar\'{e} polynomial of $W$ evaluated at $q$, which is non-zero when $q$ is not a root of unity \cite[Cor. 2.5]{macdonald}.

\medskip

We now address compatibility with parabolic induction.  First, note that by Proposition \ref{calabi yau} we have $\cL \nu^* = \cL \nu^*$, since $\nu$ is smooth.  Let $H = B/U$, fix a parabolic $P \supset B$ with quotient Levi $M$, and let $B_M \subset B$ denote the Borel subgroup defined to be the image of $B \subset P$ under the quotient.  Consider the correspondence
$$\begin{tikzcd}[column sep=-3em]
& \arrow[dl, "i"'] \cZ_P/P_{\ggr} := \mf{n}/\Bgr \times_{\mf{p}/P_{\ggr}} \mf{n}/\Bgr \arrow[dr, "p"] & \\
\cZ_G/\GG := \mf{n}/\Bgr \times_{\mf{g}/\GG} \mf{n}/\Bgr & & \cZ_M/M_{\ggr} := \mf{n}_M/B_{M,\ggr} \times_{\mf{m}/M_{\ggr}} \mf{n}_M/B_{M, \ggr}.
\end{tikzcd}$$
Note that the correspondence satisfies the conditions of Proposition \ref{springer functorial}, i.e. since $\mf{n}/B = \mf{b}/B \times_{\mf{h}/H} \{0\}/H$ (and similarly for $B_M$), and the formation of loop spaces commutes with fiber products, we have via base change that $\cS_G = \cL\mu_* \cO_{\cL(\mf{n}/B)} \simeq \cL\mu_*\cL \nu^* \cO_{\cL(\{0\}/H)}$, and similar formulas hold for $\cS_M$.  That is, the coherent Springer sheaf is the parabolic induction of the structure sheaf of $\cL(\{0\}/H)$.  Thus, we have a Cartesian diagram
$$\begin{tikzcd}
& & \cL(\mf{b}/B) \arrow[dl] \arrow[dr] & & \\
& \cL(\mf{b}_M/B_M) \arrow[dl] \arrow[dr] & & \cL(\mf{p}/P) \arrow[dl, "\nu"'] \arrow[dr, "\mu"] & \\
\cL(\mf{h}/H) & & \cL(\mf{m}/M) & & \cL(\mf{g}/G)
\end{tikzcd}$$
thus $\cL\mu_* \cL\nu^* \cS_M \simeq \cS_G$ by base change.  By the commuting diagram
$$\begin{tikzcd}[column sep=large]
HH(\Coh(\cZ_M/M_{\ggr})) \arrow[r, "\simeq"', "{\text{Prop. } \ref{HH prop}}"] \arrow[d, "HH(i_*p^*)"] & \omega(\cL(\cZ_M/M_{\ggr})) \arrow[r, "\simeq"', "{\text{Prop. } \ref{conv volume}}"] \arrow[d, "{\text{Def. } \ref{HH decat functoriality}}"] & \End(\cS_M) \arrow[d, "{\text{Def. } \ref{springer functorial}}"] \\
HH(\Coh(\cZ_G/\GG)) \arrow[r, "\simeq"', "{\text{Prop. } \ref{HH prop}}"] & \omega(\cL(\cZ_G/\GG)) \arrow[r, "\simeq"', "{\text{Prop. } \ref{conv volume}}"] & \End(\cS_G),
\end{tikzcd}$$
it remains to check that the map $HH(\Coh(\cZ_M/M_{\ggr})) \rightarrow HH(\Coh(\cZ_G/\GG))$ induces the parabolic induction map on affine Hecke algebras.  By Corollary \ref{actual theorem} we can argue for $K_0$ instead, i.e. we show that the map
$$\cH_M \simeq K_0(\Coh(\cZ_M/M_{\ggr})) \longrightarrow K_0(\Coh(\cZ_G/\GG)) \simeq \cH_G$$
agrees with the natural parabolic induction map of affine Hecke algebras $\cH_M \rightarrow \cH_G$ which takes $T_{M,w} \mapsto T_{G,w}$ where $w \in W_{a,M}$ (in the notation of Section 7.1 of \cite{CG}).   We will assume $G$ has simply connected derived subgroup, but the general case follows by passing to invariants of finite central subgroups (i.e. as in Section \ref{not sc}).  
It suffices to show that they agree for finite simple reflections and on the lattice.  Via the proof of Theorem 7.2.5 in \cite{CG}, it is clear that the map is as claimed on the lattice; we argue that parabolic induction on $K_0$ sends $[\mathcal{Q}_{M,s}] \mapsto [\mathcal{Q}_{G,s}]$ where $s$ is a finite simple reflection of $M$.

Let us recall the definition of $\cQ_{M, s}$.  The underlying closed, reduced scheme of $\cZ_M$ is a disjoint union of conormal bundles to closures of $M$-orbits $\overline{Y}_{M,s} \subset M/B_M \times M/B_M$; we denote these subschemes and the projection by $\pi_{M,s}: \cZ_{M,s} \rightarrow \overline{Y}_{M,s}$ and the inclusion $\iota_{M,s}: \cZ_{M,s} \hookrightarrow \cZ_M$.  We define $\cQ_{M,s} := \iota_{M,s,*} \pi^*_{M,s} \Omega^1_{\overline{Y}_{M,s}/(M/B_M)^2}$.

We have a similar description of $\cZ_{P,s} \subset \cZ_P$.  The map $p: \cZ_P \rightarrow \cZ_M$ is a $\mf{u}/U$-fibration, base changed from the quotient the quotient map $\mf{p}/P \rightarrow \mf{m}/M$.  In particular, $\cZ_{P,s}$ and $\cZ_{M,s} \times_{\cZ_M} \cZ_P$ are closed reduced underived subschemes of $\cZ_P$ with the same points, and thus agree.  On the other hand, we have
$\overline{Y}_{P, s} = (B \bs P / B) \times_{\overline{Y}_{M,s}} (B_M \bs M/B_M)$, so that denoting the projection $p: \overline{Y}_{P,s} \rightarrow \overline{Y}_{M,s}$ we have $\Omega^1_{\overline{Y}_{P,s}/(P/B)^2} \simeq p^*\Omega^1_{\overline{Y}_{M,s}/(M/B_M)^2}$ and thus $p^* \cQ_{M,s} \simeq \cQ_{P,s}$ by base change.  We have $\cQ_{G,s} = i_* \cQ_{P,s}$ by definition, and the claim follows.  Finally, the statements for specialized $q$ follow by Proposition \ref{unmixed id}, completing the proof.
\end{proof}

\begin{rmk}\label{DL not equivalence}
A few remarks on the theorem.
\begin{enumerate}
\item Analogous statements hold when $q=1$, where Hochschild homology of the Steinberg stack does not agree with the Grothendieck group, i.e. we have 
$$HH(\Coh(\cZ/G)) \simeq \End_{\cL(\cN/G)}(\cS_1)  \simeq kW^a \otimes \Symp(\mf{h}^*[-1] \oplus \mf{h}^*[-2]) \simeq \cH_1^{un}$$
while $K_0(\Coh(\cZ/G))_k \simeq kW^a \simeq \cH_1$.    However, the anti-spherical module arising via Hochschild homology agrees with $K_0$, i.e. 
$$HH(\Coh(\wt{\cN}/G)) \simeq K_0(\Coh(\wt{\cN}/G))_k \simeq k W^a \otimes_{k W^f} k_{\mathrm{sgn}},$$
 where $\mf{h}^*[-1] \oplus \mf{h}^*[-2] \subset \cH_1^{un}$ acts by zero.
\item Compatibility with parabolic induction implies that the action of the lattice subalgebra $k[X^\bullet(H)] \subset \cH$ on the coherent Springer sheaf comes from the $\cO(\cL(\wt{\cN}/\GG))$-action on $\cS = \cL \mu_* \cO_{\cL(\wt{\cN}/\GG)}$ via the natural $\cO(\cL(\{0\}/H)) = \cO(H)$-algebra structure on $\cO(\cL(\NN/\GG))$.
\item If $G = T$ is a torus, then $\cL(\cN_T/T) \simeq \wh{\{e\}} \times T \times BT$ and $\cS = \cO_{\{e\} \times T \times BT}$, and we see immediately that $\End(\cS) \simeq k[T] = k[X^\bullet(T)]$.
\item Let $\wt{\mf{g}}_P = G \times^P \mf{p}$; applying our methods in Sections \ref{standard rep} and \ref{app enhanced} to the $\QCoh(\mf{g}/G)$-module category $\QCoh(\wt{\mf{g}}_P/G)$, one can show that for $q$ not a root of unity, the coherent ``partial Whittaker'' sheaves, obtained by applying the parabolic induction correspondence 
$$\bL^u_{q,M} = \cL_q(\cN_M/M) \longleftarrow \bL^u_{q,P} = \cL_q(\cN_P/P) \longrightarrow \bL^u_{q,G} = \cL_q(\cN_G/G)$$ 
to the structure sheaf $\cO_{\bL^u_{q,M}}$ are also summands of the coherent Springer sheaf.  For example, at the extremes taking $P = G$ we obtain the statement for the anti-spherical sheaf, and taking $P = B$ we obtain the coherent Springer sheaf itself.
\end{enumerate}
\end{rmk}

\begin{rmk}\label{no sing supp}
We explain the absence of a singular support condition.  There are two Koszul dual versions of the Steinberg variety leading to two versions of the unipotent affine Hecke algebra: the ``Springer'' version $\cZ = \wt{\cN} \times_{\mf{g}} \wt{\cN}$ we consider and a ``Grothendieck-Springer'' version $\cZ_{\mf{g}} := \wt{\mf{g}} \times_{\mf{g}} \wt{\mf{g}}$.  Theorem 4.4.1 of \cite{BNP} shows the singular support condition appearing for trace sheaves in $\Tr(\Coh(\cZ_{\mf{g}}/\GG))$ in the ``Grothendieck-Springer'' version can be characterized by a nilpotence condition.

We now argue that the singular support condition for the ``Springer'' version $\Tr(\Coh(\cZ/\GG))$ is vacuous, i.e. that the singular support locus $\Lambda_{\wt{\cN}/\mf{g}}$ is the entire scheme of singularities $\Sing(\cL(\wh{\cN}/\GG))$.   The singular locus of $\LL(\wh{\cN}/\GG)$ at a $k$-point $\eta = (n, z = (g, q))$ where $gng^{-1} = qn$ is the set (after identifying $\mf{g} \simeq \mf{g}^*$ via a non-degenerate form $\langle-,-\rangle$):
$$\Sing(\LL(\wh{\cN}/\GG))_\eta = \{v \in \mf{g} \mid gvg^{-1} = q^{-1}v, [n, v] = 0, \langle n, v \rangle = 0\}.$$
A calculation\footnote{In contrast to the singular support calculation for $\Coh(\cZ_{\mf{g}}/\GG)$, it is the Lie algebra of the Borel $\mf{b}$ that appears in the above condition rather than its nilradical $\mf{n}$ since
$$\Sing(\wt{\cN} \times_{\mf{g}} \wt{\cN})_{(n, B, B')} \subset \mf{b} \cap \mf{b}', \;\;\;\;\;\;\; \Sing(\wt{\mf{g}} \times_G \wt{\mf{g}})_{(x, B, B')} \subset \mf{n} \cap \mf{n}'.$$} shows that the singular support locus is given by:
$$(\Lambda_{\wt{\cN}/\mf{g}})_\eta = \{v \in \Sing(\LL(\wh{\cN}/\GG))_\eta \mid \exists \text{ Borel } B \subset G \text{ such that } n, v \in \mf{b} = \mathrm{Lie}(B)\}.$$
Note that $n, v$ generate a two-dimensional solvable Lie algebra, thus are contained in a Borel, so $\Sing(\cL(\wh{\cN}/\GG))_\eta = \Lambda_{\wt{\cN}/\mf{g}}$.  In particular, the singular codirection $v$ need not be nilpotent.

The analogous claim at specific $q \in \G_m$ follows by a similar argument  and a calculation of the singular support locus at a point $\eta \in \{(n, g) \in \cN \times G \mid gng^{-1} = qn\}$ as
$$\Sing(\LL_q(\wh{\cN}/G))_\eta = \{v \in \mf{g} \mid gvg^{-1} = q^{-1}v, [n,v] = 0\}.$$
In the case of $q$ not a root of unity, the argument in Proposition \ref{no derived structure} shows that the singular codirection $v$ must be nilpotent, i.e. $\Sing(\cL_q(\wh{\cN}/G))$ can only contain nilpotent singular codirections.  This condition is not imposed by the singular support condition itself, i.e. in this case all singular codirections are nilpotent to begin with.
\end{rmk}

\medskip

It is natural to conjecture that the coherent Springer sheaf is in fact a sheaf -- i.e., lives in the heart of the dg category $\Coh(\LL(\mf{g}/\GG))$. We prove this in the case $G = \GL_2, \SL_2$ in Proposition \ref{sl2 springer heart}.

\begin{conj}\label{springer heart}
The Springer sheaf $\cS$ lives in the abelian category $\Coh(\LL(\wh{\cN}/\GG))^\heartsuit$. 
\end{conj}

\begin{rmk}
One consequence of the conjecture would be an explicit description of the endomorphisms of the cohrent Springer sheaf.  Namely, it is easy to see that the underived parabolic induction from $\cL(\{0\}/H)$ is generated as a module by the lattice $X^\bullet(H)$, and via the identification with $K$-theory and Theorem 7.2.16 of \cite{CG} we would obtain a description of the action of finite simple reflections in terms of Demazure operators.
\end{rmk}

\begin{rmk}
A variant of Conjecture \ref{springer heart} was answered in the affirmative in Corollary 4.4.6 of \cite{ginzburg isospectral}.  Namely, in \emph{loc. cit.} it is proven that the Lie algebra version of our coherent Springer sheaf at $q=1$ has vanishing higher cohomology.
\end{rmk}

\begin{rmk}\label{springer heart remark}
When $\GG$ acts on $\NN$ by finitely many orbits, then $\LL(\NN/\GG)$ has trivial derived structure, and the conjecture is implied by the vanishing of higher cohomology of a classical scheme $H^i(\LL(\NN/\GG) \times_{B\GG} \pt, \pi_0(\OO_{\LL(\NN/\GG) \times_{B\GG} \pt}))$ for $i > 0$.  The $G$-orbits in the Springer resolution are known to be finite exactly in types $A_1, A_2, A_3, A_4, B_2$ by \cite{kashin}.  
\end{rmk}

\medskip
\subsection{Conjectures and examples for $G = \SL_2, \GL_2, \PGL_2$}\label{sec conjectures}

In this case, $\GG$ acts on both $\wh{\cN}$ and $\NN$ by finitely many orbits, the derived loop spaces $\LL(\wh{\cN}/\GG)$ and $\LL(\NN/\GG)$ are classical stacks.  Recall that $\wh{\cN}$ is a formal completion; if the reader would rather do so, they may replace $\wh{\cN}$ with $\mf{g}$, which is also acted on by finitely many orbits.  We prove Conjecture \ref{springer heart} in these cases.
\begin{prop}\label{sl2 springer heart}
Conjecture \ref{springer heart} holds for $G = \SL_2, \GL_2, \PGL_2$.
\end{prop}
\begin{proof}
We give a proof for $G = \SL_2$; the case of $G = \GL_2$ is the same.  In view of Remark \ref{springer heart remark}, it suffices to forget equivariance and show vanishing of higher cohomology.  Since $X := \LL(\NN/\GG) \times_{B\GG} \pt$ is a closed subscheme of $\mf{g} \times G/B \times G$, and $\dim(G/B) = 1$, we know that $R\Gamma^i(X, -) = 0$ for $i > 1$.  To verify vanishing for $i=1$, let $i: X \hookrightarrow \NN \times \GG$ be the closed immersion.  We have a short exact sequence of sheaves:
$$0 \rightarrow \mathcal{I} \rightarrow \OO_{\NN \times \GG} \rightarrow i_* \OO_X \rightarrow 0$$
leading to a long exact sequence with vanishing $H^2$ terms (for the above reason).  Thus, it suffices to show that $H^1(\NN \times \GG, \OO_{\NN \times \GG})$.  By the projection formula, we have $H^1(\NN \times \GG, \OO_{\NN \times \GG}) \simeq H^1(\NN, \OO_{\NN}) \otimes_k \OO(\GG)$, but it is well-known that $H^i(\NN, \OO_{\NN}) = 0$ for $i > 0$.
\end{proof}

\begin{exmp}[Geometry of the loop space of the Springer resolution]
We describe the geometry of the looped Springer resolution $\cL(\NN/\GG) \rightarrow \cL(\wh{\cN}/\GG)$ for $G = \SL_2$.  Though this example is well-known, we reproduce it for the reader's convenience.  Let $A(s, n)$ denote the component group of the double stabilizer group, i.e. the component group of $\{g \in G \mid gng^{-1} = n, gs = sg\}$.  Let $\mathbb{A}^1_{\mathrm{node}} = \Spec k[x,y]/xy$ denote the affine nodal curve, and $(-)^\nu$ the normalization. \\
\begin{center}
\begin{tabular}{|c|c|c|c|c|c|}
$q$ &$n$ &   $s = \begin{pmatrix} \lambda &0\\0&\lambda^{-1} \end{pmatrix}$ & $\cN^{(s,q)} \rightarrow \NN^{(s,q)}$ &$A(s, n)$ & $ G^s$ \\ \hhline{|=|=|=|=|=|=|}
\multirow{2}*{$q=1$} & $n=0$&  \multirow{2}*{$\lambda=\pm 1$} & \multirow{2}*{$\NN \rightarrow \cN$} & $1$ & \multirow{2}*{$G$} \\
&   $n \ne 0$ & & & $\Z/2$& \\ 
\hline
$q=1$ & $n=0$ & $\lambda\ne \pm 1$ & $\pt \cup \pt \rightarrow \pt$ & 1 & $T$ \\ \hhline{|=|=|=|=|=|=|}

\multirow{3}*{$q=-1$} &  $n=0$ &\multirow{3}*{$\lambda = i$} & \multirow{3}*{$\mathbb{A}^{1, \nu}_{\mathrm{node}} \rightarrow \mathbb{A}^1_{\mathrm{node}}$} & $1$& \multirow{3}*{$T$} \\
& $n\ne 0$, upper triangular  & & & $\Z/2$&  \\
& $n \ne 0$, lower triangular  & & & $\Z/2$&  \\ \hline
$q = -1$ & $n=0$ &  $\lambda = \pm 1$ & $\mathbb{P}^1 \rightarrow \pt$ & $1$ & $G$ \\ \hline
$q=-1$ & $n=0$ & $\lambda\ne \pm 1$ & $\pt \cup \pt \rightarrow \pt$ & 1 & $T$ \\ \hhline{|=|=|=|=|=|=|}

\multirow{2}*{$q \ne \pm 1$} & $n=0$ & \multirow{2}*{$\lambda = \pm \sqrt{q}$} & \multirow{2}*{$\mathbb{A}^1 \cup \pt \rightarrow \mathbb{A}^1$} & $1$ & \multirow{2}*{$T$} \\
& $n \ne 0$& &  & $\Z/2$ & \\ \hline
$q\ne \pm 1$ & $n=0$ &  $\lambda = \pm 1$ & $\mathbb{P}^1 \rightarrow \pt$ & $1$ & $G$ \\ \hline
$q \ne \pm 1$ &$n=0$ &  $\lambda\ne \pm 1, \pm \sqrt{q}$ & $\pt \cup \pt \rightarrow \pt$ & $1$ & $T$ \\ \hhline{|=|=|=|=|=|=|}
\end{tabular} \\
\end{center}
\end{exmp}

\begin{exmp}[Generators and relations]\label{sl2 demazure}
For $G = \SL_2$, with some work, one can write down generators and relations for the (underived) scheme $\cL(\mf{g}/\GG)$ and the coherent Springer sheaf $\cS$.  Let us fix coordinates
$$g = \begin{pmatrix} a&b\\c&d\end{pmatrix} \in \SL_2, \;\;\;\;\;\;\; N = \begin{pmatrix} x&y\\z&-x\end{pmatrix} \in \cN_{\mf{sl}_2},\;\;\;\;\;\;\; q \in \G_m.$$
We implicitly impose the equations $ad-bc=1$ and $x^2+yz=0$, and by convention we take the commuting relation $gxg^{-1} = qx$; note that this is the relation that arises when $\G_m$ acts on fibers by weight -1 (i.e. inversely).  Then, we have that 
$\cS$ is the module with generators $\lambda^{n}$ for $n \in \mathbb{Z}$:
$$\displaystyle \frac{\cO(\SL_2 \times \cN_{\mf{sl}_2} \times \G_m)[\lambda, \lambda^{-1}]}{\substack{\displaystyle a+d=\lambda + \lambda^{-1}, (x,y,z) (q - \lambda^2) = 0, \\ \displaystyle z(\lambda - d) = ax, y(a - \lambda) = bx, x(d - \lambda) = cy, x(\lambda - a) = bz.}}$$
In particular, multiplication by $\lambda^{n}$ defines the action of the lattice, and one can verify that the Demazure operator for the anti-spherical module (see Theorem 7.2.16 of \cite{CG}) defines the endomorphism
$$T(\lambda^n) = \frac{\lambda^{n} - \lambda^{-n +2}}{\lambda^2 - 1} - q\displaystyle\frac{\lambda^n - \lambda^{-n}}{\lambda^2 - 1} $$
corresponding to the finite reflection.  In particular, it preserves the relations in the module, and the endomorphism satisfies $(T -q)(T+1)=0$.
For fixed $q$, and letting $k_{\mathrm{sgn}}$ denote the character of $\cH^f$ with $T \mapsto -1$, one can verify that $\cS \otimes_{\cH^f} k_{\mathrm{sgn}} \simeq \cO_{\cL_q(\wh{\cN}/G)}$, i.e. amounts to imposing the relation $\lambda^2 = q$, thus identifying the structure sheaf with the anti-spherical module.
\end{exmp}

\section{Moduli of Langlands parameters for \texorpdfstring{$\GL_n$}{GLn}}\label{moduli}

We now turn to arithmetic applications of our results, in particular the study of moduli spaces of Langlands parameters for $G = \GL_n$.  Let $F$ be a non-archimedian local field with residue field ${\mathbb F}_q$, and let $G^\vee$ denote a connected, split, reductive group over $F$ (i.e. on the automorphic side of Langlands).
\medskip

The derived category $D(G^\vee)$ of smooth complex representations of $G^\vee$ admits a decomposition into blocks, and the so-called {\em principal block} of $D(G^\vee)$ (that is, the block containing the trivial representation) is naturally
equivalent to the category of ${\mathcal H}_q$-modules, where ${\mathcal H}_q$ now denotes the affine Hecke algebra associated to $G$ with parameter $q$.  Theorem \ref{thm endomorph} then gives a fully faithful embedding
from this principal block into $\IndCoh(\cL_q(\wh{\cN}/G))$.

The space $\LL_q(\wh{\cN}/G)$ has a natural interpretation in terms of Langlands (or Weil-Deligne) parameters for $G^\vee(F)$.  Recall that a
Langlands parameter for $G^\vee$ is a pair $(\rho,N)$, where $\rho: W_F \rightarrow G(\C)$ is a homomorphism with open
kernel, and $N$ is a nilpotent element of $\Lie G$ such that, for all $\sigma$ in the inertia group $I_F$ of $W_F$, one has
$\operatorname{Ad}(\rho(\Fr^n \sigma))(N)  = q^n N,$ where $\Fr$ denotes a Frobenius element of $W_F$.

On the other hand, the underlying stack
of $\cL_q(\wh{\cN}/G)$ can be regarded as the moduli stack of pairs $(s,N)$, where $s \in G(\C)$, $N \in \Lie G$, and $\operatorname{Ad}(s) (N) = qN$, up
to $G$-conjugacy (i.e. the map $\rho$ above vanishes on inertia).  To such a pair we can attach the Langlands parameter $(\rho,N)$, where $\rho$ is the unramified representation
of $W_F$ taking $\Fr$ to $s$.  Such a Langlands parameter is called {\em unipotent}, and this construction identifies $\LL_q(\wh{\cN}/G)$ with
the moduli stack of unipotent Langlands parameters, modulo $G$-conjugacy.\footnote{Strictly speaking, a Langlands parameter is a pair $(\rho,N)$ as above in which {\em $\rho$ is semisimple}.  When building a
moduli space of Langlands parameters we must drop this condition, however, as the space of semisimple parameters is not a well-behaved geometric object.
In particular the locus in $\LL_q$ consisting of pairs $(s,N)$ in which $s$ is semisimple is neither closed nor open in $\LL_q$.}  We thus obtain a fully faithful embedding from the principal block of $D(G^\vee)$ into the category of ind-coherent sheaves
on the moduli stack of unipotent Langlands parameters.  

It is natural to ask if this extends to an embedding
of \emph{all} of $D(G^\vee)$ into a category of sheaves on the moduli stack of {\em all} Langlands parameters.  We will show
that, at least when $G = \GL_n$ over $F$, this is indeed the case.   For the remainder of the section, we will take $G = G^\vee = \GL_n$.

\medskip

\subsection{Blocks, semisimple types, and affine Hecke algebras}\label{types}

Our argument proceeds by reducing to the principal block.  On the representation theory side, this reduction is a consequence
of the Bushnell-Kutzko theory of types and covers \cite{BK structure, BK}, which we now recall.  \emph{For this subsection only, we will reverse our conventions} to avoid cumbersome notation; that is, we let $G$ be a connected reductive split group over $F$ on the automorphic side of Langlands duality.

\medskip

\subsubsection{Supercuspidal support}

Let $P \subset G$ be a parabolic subgroup with Levi $M$ and unipotent radical $U$, and let $\pi$ be a smooth complex representation of $M$.  Recall
that the parabolic induction $i_{P}^{G}(\pi)$ is obtained by inflating $\pi$
to a representation of $P$, twisting by the square root of the modulus character
of $P$, and inducing to $G$.  The parabolic induction functor $i_{P}^{G}$ has
a natural left adjoint, the parabolic restriction $r_{G}^{P}$ (restriction to $P$, untwist, and $U$-coinvariants).

\begin{defn} A complex representation $\pi$ of $G$ is \emph{supercuspidal}
if, for all proper parabolic subgroups $P$ of $G$, the parabolic restriction
$r_G^P(\pi)$ vanishes. 
Let $\pi$ be an irreducible supercuspidal representation of $M$; an irreducible complex representation $\Pi$ has \emph{supercuspidal support} $(M,\pi)$
if $\Pi$ is isomorphic to a subquotient of $i_P^G(\pi)$ (this is well-defined up to conjugacy).

A character $\chi$ of $M$ is {\em unramified} if it is trivial on
every compact open subgroup of $M$, and the Levi-supercuspidal pairs $(M,\pi)$ and $(L,\pi')$ are {\em inertially equivalent}
if there exists an unramified character $\chi$ of $L$ such that $(M, \pi)$ and $(L, \pi' \otimes \chi)$ are $G$-conjugate.
\end{defn}

For such a pair $(M, \pi)$ up to inertial equivalence, following Bernstein-Deligne~\cite{bernstein-deligne}, we
define  $D(G)_{[M,\pi]} \subset D(G)$ to be the full subcategory of objects such that every subquotient of $\Pi$ has supercuspidal support inertially equivalent to $(M,\pi)$.  Then Bernstein-Deligne show:

\begin{thm} The full subcategory $D(G)_{[M,\pi]}$ is a block of $D(G)$, i.e. summing over supercuspidals up to inertial equivalence,
$$D(G) = \bigoplus D(G)_{[M, \pi]}.$$
\end{thm}

\medskip

\subsubsection{Types and Hecke algebras}

We recall the notion of a type.
\begin{defn}
A {\em type} for $G$ is a pair $(K,\tau)$, where $K \subset G$ is a compact
open subgroup and $\tau$ is an irreducible complex representation of $K$, such that\footnote{See pp. 594 of \cite{BK structure} for why this is necessary.} the full subcategory $\Rep(G,K, \tau)^\heartsuit \subset \Rep(G)^\heartsuit$ of representations $V$ which are generated by the image of the evaluation $\Hom_K(\tau, V) \otimes \tau \rightarrow V$ is closed under taking subquotients.   Attached to a type we have its Hecke algebra 
$$\cH(G,K,\tau) := \End_{G}(\cind_K^G(\tau))$$
and an equivalence of abelian categories $\Rep(G, K, \tau)^\heartsuit \simeq D(\cH(G, K, \tau))^\heartsuit$.
\end{defn}

The main result of~\cite{BK} describes an arbitrary block of
$D(G)$ as a category of modules for a certain tensor product of Hecke algebras, via the theory of $G$-covers, providing a connection between parabolic induction methods (which involve subgroups which are not compact open) and Hecke algebra methods (which only make sense for compact open subgroups).

\medskip

We first consider the block $D(L)_{[L, \pi]}$, where $L^{}$ be a Levi subgroup of $G$ and $\pi$ a supercuspidal representation of $L$.  We denote by $L_0 \subset L$ the smallest subgroup containing every compact open; then $L/L_0$ is free abelian of rank equal to $\dim(Z(L))$.  Furthermore, the unramified characters of $L$ are in bijection with the characters of $L/L_0$.  There is a bijection
$$X^\bullet(L/L_0)/H \longleftrightarrow \mathrm{Irr}(L)_{[L^{},\pi]}, \;\;\;\;\;\;\; \chi \mapsto \pi \otimes \chi$$ 
where we denote $X^\bullet(L/L_0) = \Hom(L/L_0, \C^\times)$ and $H \subset X^\bullet(L/L_0)$ is the subgroup of unramified characters $\chi$ such that $\pi \otimes \chi \simeq \pi$.  Moreover, there is an equivalence of
categories:
$$D(L^{})_{[L^{},\pi]} \simeq D(\C[X^\bullet(L/L_0)]^H), \;\;\;\;\;\;\; \pi \otimes \chi \mapsto \C_\chi.$$

\medskip

We may rephrase this equivalence in terms of types and Hecke algebras as follows: first, we may (by Section 1.2 in \cite{BK}) choose a maximal simple cuspidal
type $(K_{L^{}},\tau_{L^{}})$ occurring in $\pi$.  One then has a natural support-preserving
isomorphism of $\cH(L^{},K_{L^{}},\tau_{L^{}}) \simeq \C[X^\bullet(L/L_0)]^H$, and thus an (inverse) equivalence
$$D(\C[X^\bullet(L/L_0)]^H) \simeq D(G)_{[L,\pi]}, \;\;\;\;\;\;\; V \mapsto V \otimes_{\cH(L, K_{L}, \tau_{L})} \cind_{K_{L}}^{L} \tau_{L}.$$

\medskip

We are interested in understanding the induction of $(L, \pi)$ to $G$.  This is achieved by the following composite of results of~\cite{BK}; we refer the reader to \emph{op. cit.} for the definitions of simple type and $G$-cover.

\begin{thm}[\cite{BK}] \label{thm:cover}
Let $[L,\pi]$ and the cuspidal type $(K_{L^{}},\tau_{L^{}})$ be as above, and
let $P \subset G^{}$ be a parabolic subgroup with Levi factor $L^{}$.
There exists an intermediate\footnote{Defined to be the smallest Levi containing the $G$-normalizer of the type $(K_L, \tau_L)$.} Levi subgroup $L \subset L^{\dagger} \subset G^{}$,
and types $(K^{\dagger},\tau^{\dagger})$ of ${L^{\dagger}}$ and $(K,\tau)$ of $G^{}$
with the following properties:

\begin{enumerate}
\item The type $(K^{\dagger},\tau^{\dagger})$ is a simple type of $L^\dagger$.
\item $(K,\tau)$ is a $G^{}$-cover of $(K^{\dagger},\tau^{\dagger})$, and $(K^{\dagger},\tau^{\dagger})$ is an  $L^\dagger$-cover of $(K_{L^{}},\tau_{L^{}})$.
In particular we have natural injections:
$$\begin{tikzcd}[row sep = tiny] 
T_{P^{} \cap L^{\dagger}}:  \, \cH(L^{},K_{L^{}},\tau_{L^{}}) \arrow[r, hook] & \cH(L^{\dagger},K^{\dagger},\tau^{\dagger}) \\
T_{L^{\dagger} P^{}}: \, \cH(L^{\dagger},K^{\dagger},\tau^{\dagger}) \arrow[r, hook, "\simeq"] & \cH(G^{},K,\tau)\end{tikzcd}$$
with $T_{L^{\dagger} P^{}}$ an isomorphism.
\item The functors 
$$\begin{tikzcd}[row sep=tiny] \Hom_K(\tau, -):\, D(G^{})_{[L^{},\pi]} \arrow[r, "\simeq"] & D(\cH(G^{},K,\tau)) \\
\Hom_{K^{\dagger}}(\tau^{\dagger}, -):\, D(L^{\dagger})_{[L^{},\pi]} \arrow[r, "\simeq"] & D(\cH(L^{\dagger},K^{\dagger},\tau^{\dagger}))\\
\Hom_{K_{L^{}}}(\tau_{L^{}},-): \,D(L^{})_{[L^{},\pi]} \arrow[r,"\simeq"] &  D(\cH(L^{},K_{L^{}},\tau_{L^{}}))
\end{tikzcd}$$
are equivalences of categories. Moreover, for any representation $V$ in $D(L^{})$, one has an isomorphism of $\cH(G^{},K,\tau)$-modules:
$$\Hom_K(\tau, i_{P'}^{G^{}} V) \cong 
\Hom_{K_{L^{}}}(\tau_{L^{}}, V) \otimes_{\cH(L^{},K_{L^{}},\tau_{L^{}})} \cH(G^{},K,\tau),$$
where $P'$ denotes the opposite parabolic to $P^{}$,
and where $\cH(G^{},K,\tau)$ is regarded as an $\cH(L^{},K_{L^{}},\tau_{L^{}})$-module via the map $T_{P^{}} := T_{L^{\dagger} P^{}} \circ T_{P^{} \cap L^{\dagger}}$.
\item Suppose $L^\dagger \simeq \prod_i L^\dagger_i$, with each $L^\dagger_i \simeq \GL_{n_i}$ for some $n_i$.  
Let $L_i$ be the projection of $L$ to $L^\dagger_i$, and
let $\pi_i$ be the projection of $\pi$ to $L^{}_i$.  Let $H_i$ denote the group of unramified characters $\chi$ of $L^{\dagger}_i$ such that $\pi \otimes \chi \simeq \pi$,
and let $r_i$ denote the order of $H_i$.
Then $n_i = r_im_i$ for some positive integer $m_i$, and there is a natural isomorphism (depending on $\pi$):
$$\cH(L^{\dagger},K^{\dagger},\tau^{\dagger}) \cong \bigotimes_i \Haff_{q^{r_i}}(m_i),$$
where $\Haff_{q^{r_i}}(m_i)$ denotes the affine Hecke algebra associated to $\GL_{m_i}$ with parameter $q^{r_i}$.
\end{enumerate}
\end{thm}

These constructions are naturally compatible with parabolic induction, in the following sense: let $M$ be a Levi with $L \subset M \subset G$, and with parabolic $Q = M^{} P^{}$.  Then Theorem~\ref{thm:cover}
gives us an $M^{}$-cover $(K_{M^{}},\tau_{M^{}})$ of $(K_{L^{}},\tau_{L^{}})$ and a $G^{}$-cover $(K,\tau)$ of $(K_{L^{}},\tau_{L^{}})$,
as well as maps:
$$T_{P^{} \cap M^{}}: \cH(L^{}, K_{L^{}}, \tau_{L^{}}) \rightarrow \cH(M^{}, K_{M^{}}, \tau_{M^{}}), \;\;\;\;\;\;\; T_{P^{}}: \cH(L^{}, K_{L^{}}, \tau_{L^{}}) \rightarrow \cH(G^{}, K, \tau).$$
We then have:
\begin{thm}[\cite{BK}] \label{thm:cover induction}
There exists a unique map:
$$T_{Q^{}}: \cH(M^{}, K_{M^{}},\tau_{M^{}}) \rightarrow \cH(G^{}, K, \tau)$$
such that $T_{P^{}} = T_{Q^{}} \circ T_{P^{} \cap M^{}}$.  Moreover, for any $V \in D(M^{})$, we have an isomorphism of $\cH(G^{},K,\tau)$-modules:
$$\Hom_K(\tau, i_{Q'}^{G^{}} V) \cong 
\Hom_{K_{M^{}}}(\tau_{M^{}}, V) \otimes_{\cH(M^{},K_{M^{}},\tau_{M^{}})} \cH(G^{},K,\tau).$$
\end{thm}

\medskip

\begin{exmp}
The fundamental (and motivating) example for this is when $L=T$ is the standard maximal torus with parabolic $P=B$ the standard Borel,
and $\tau = 1$ is the trivial character of $T^{}$.  In this setting $K_{L^{}}$ is the maximal compact subgroup
$T^{}_0 = T(O) \subset T^{}$, and $\tau_{L^{}}$ is the trivial character.  Moreover $L^\dagger = G^{}$, the subgroup $K = I \subset G$ is the
Iwahori subgroup, and $\tau$ is the trivial representation of $I$.  We then have natural identifications of the Hecke algebra:
$$\cH(L^{},K_{L^{}}, 1) \simeq \C[T/T_0] \simeq \C[X_\bullet(T)].$$
and a commutative diagram:
$$
\begin{tikzcd}
\C[X_{\bullet}] \arrow[r, "\simeq"]  \arrow[d, hook] & \cH(T,T(O), 1) \arrow[d, "T_{P}", hook] \\
\Haff_q \arrow[r, "\simeq"] &  \cH(G,I, 1).
\end{tikzcd}
$$

\medskip

More generally, if $M \subset G$ is a Levi subgroup and $Q$ is its standard parabolic, then $K_M$
is the Iwahori subgroup $I \cap M$ of $M$, and the map
$$T_Q: \cH(M, I \cap M, 1) \rightarrow \cH(G,I,1)$$ 
is uniquely determined by the following properties:
\begin{enumerate}
\item $T_Q \circ T_{B \cap M} = T_B$,
\item If $w \in W(M)$ is an element of the Iwahori-Weyl group of $M$, then $T_Q(I_M w I_M) = I w I$. 
\end{enumerate}
\end{exmp}

\medskip

This picture is compatible with the general situation in the following sense.  Suppose for simplicity that $L^\dagger = G^{}$.  Then $L^{}$ is a product
of $m$ copies of $\GL_{\frac{n}{m}}$ for some divisor $m$ of $n$, and (after an unramified twist) we may assume that $\pi$ has the form $\pi_0^{\otimes m}$.
There is an extension $E/F$ of degree $\frac{n}{m}$ and ramification index $r$, and an embedding $\GL_m(E) \subset G = \GL_n(F)$, such that the intersection $L \cap GL_m(E)$ is the standard maximal torus of $GL_m(E)$.  

We denote the subgroup $GL_m(E)$ by $G_E$, its standard maximal torus by $T_E$ and its standard Iwahori by $I_E$.  Let $M$ be a Levi such that $L \subset M \subset G$, define $M_E = M \cap G_E$ and take $(K_M, \tau_M)$ to be a cover of $(K_L, \tau_L)$ via Theorem \ref{thm:cover}.  The choice of $\pi$ then gives rise to an isomorphism $\C[X_{\bullet}(T)] \simeq \cH(L^{},K_{L^{}},\tau_{L^{}}),$ such that for each
coharacter $\lambda \in X_\bullet(T)$ the image of $\lambda$ is supported on the double coset $K_{L^{}} \lambda(\varpi_E) K_{L^{}}$, and such that the induced action
of $X_{\bullet}(T)$ on the Hecke module attached to $\pi$ is trivial.  We then have:

\begin{thm}[Theorem 6.4 \cite{BK gln}] \label{thm:cover compatibility}
Assume that $L^\dagger=G$.  There is an isomorphism  $\Haff_{q^r}(m) \simeq \cH(G^{},K,\tau)$ fitting into a commutative diagram:
$$
\begin{array}{ccccc}
\cH(T^{}_E,(T^{}_E)_0, 1) & \cong & \C[X_{\bullet}(T)] & \cong & \cH(L^{},K_{L^{}},\tau_{L^{}})\\
\downarrow & & \downarrow & & \downarrow\\
\cH(M^{}_E, I_E \cap M^{}_E, 1) & \cong & \bigotimes_{m_i} \Haff_{q^r}(m_i) & \cong & \cH(M^{}, K_{M^{}},\tau_{M^{}})\\
\downarrow & & \downarrow & & \downarrow\\
\cH(G^{}_E, I_E, 1) & \cong & \Haff_{q^r}(m) & \cong & \cH(G^{},K,\tau).
\end{array}
$$
\end{thm}

Thus when $[L,\pi]$ is ``simple'' (that is, when $L^\dagger = G$), we have a natural reduction of $D(G^{})_{[L,\pi]}$ to the principal block of
$D(G_E^{})$, in a manner compatible with parabolic induction.  In general we obtain a reduction of $D(G^{})_{[L,\pi]}$ to a tensor product of
such principal blocks.

\medskip

\subsection{The moduli spaces \texorpdfstring{$X_{F,G}^{\nu}$}{Xv}}

We now turn to our study of moduli stacks of Langlands parameters for $G = \GL_n$.  Henceforth we revert to our default notation, where $G$ denotes a group on the spectral side of Langlands duality. 

Moduli stacks of Langlands parameters for $\GL_n$ have been studied extensively in
mixed characteristic, for instance in \cite{curtis} in the case of $\GL_n$, or more recently in
\cite{bellovin-gee,booher-patrikis}, and \cite{DHKM} for more general groups.  Since in our present context we work over $\C$, the results we need are in general simpler than the
results of the above papers, and have not appeared explicitly in the literature in the form we need.

We first consider these moduli spaces as underived stacks; it will follow by Proposition \ref{no derived structure} that they have trivial derived structure.  As in the previous section, we take $G = \GL_n$, considered as the Langlands dual of $G^{\vee} = \GL_n(F)$.  We use $X_{F,G}$ to denote the moduli scheme whose quotient stack is the moduli stack $\bL_{F, G}$ in the introduction.

\begin{defn}
Let $I$ be an open normal subgroup of the inertia subgroup $I_F \subset W_F$.  Then there is a scheme $X_{F,G}^I$ parameterizing pairs
$(\rho,N)$, where $\rho: W_F/I \rightarrow \GL_n$ is a homomorphism, and $N$ is a nilpotent $n$ by $n$ matrix
such that for all $\sigma \in I_F$, $\operatorname{Ad} \rho(\Fr^n \sigma) (N) = q^n N$.  For any $\nu: I_F/I \rightarrow \GL_n(\C)$,
we may consider the subscheme $X_{F,G}^{\nu} \subset X_{F,G}^I$ corresponding to pairs $(\rho,N)$ such that the restriction of $\rho$
to $I_F$ is conjugate to $\nu$; it is easy to see that $X_{F,G}^{\nu}$ is both open and closed in $X_{F,G}^I$.  We will say that
a Langlands parameter is of ``type $\nu$'' if it lies in $X_{F,G}^{\nu}$.
\end{defn}

\begin{exmp}
When $\nu = 1$ is the trivial representation, the quotient stack $X_{F,G}^1/G$ is isomorphic to
the underlying underived stack of $\LL_q(\wh{\cN}/G)$, as we remarked in the previous section.
\end{exmp}

We will show that
in fact, for $\nu$ arbitrary, the stack $X_{F,G}^{\nu}/G$ is isomorphic to
a product of stacks of the form $\LL_{q^{r_i}}(\wh{\cN}_i/G_i)$, in a manner that exactly parallels the
type-theoretic reductions of the previous section.  This will allow us to transfer the structures we have built up
on $\LL_{q^{r_i}}(\wh{\cN}_i/G_i)$ to stacks of the form $X_{F,G}^{\nu}/G$ for arbitrary $\nu$.  Our approach very closely parallels the construction of Sections 7 and 8 of \cite{curtis}
with the exception that we are able to work with the full inertia group $I_F$, whereas the integral
$\ell$-adic setting of \cite{curtis} requires one to work with the prime-to-$\ell$ inertia instead.

\medskip

Our strategy will be to rigidify the moduli space $X_{F.G}^\nu$.  For any $\C$-algebra $R$, let us fix a representative $\rho: W_F/I \rightarrow \GL_n(R)$ of type $\nu$, i.e. of the conjugacy class.

For any irreducible complex representation $\eta$ of $I_F$, let $W_{\eta}$ be the finite index subgroup
of $W_F$ consisting of all $w \in W_F$ such that $\eta^w$ is isomorphic to $\eta$.  Then $\eta$
extends to a representation of $W_{\eta}$, although not uniquely; let $\tilde \eta$ be a choice of such an extension.  This choice defines a natural  $W_{\eta}/I_F$-action on the space $\Hom_{I_F}(\eta,\rho)$, and an injection of $W_\eta$-representations
$${\tilde \eta} \otimes \Hom_{I_F}(\eta,\rho) \hookrightarrow \rho.$$
Frobenius reciprocity then gives an injection:
$$\Ind_{W_{\eta}}^{W_F} \left({\tilde \eta} \otimes \Hom_{I_F}(\eta,\rho)\right) \hookrightarrow \rho.$$
The image of this injection is the sum of the $I_F$-subrepresentations of $\rho$ isomorphic to a $W_F$-conjugate of
$\eta$.  We thus have a direct sum decomposition of $W_F$-representations:
$$\rho \cong \bigoplus_{\eta} \Ind_{W_{\eta}}^{W_F} \left({\tilde \eta} \otimes \Hom_{I_F}(\eta,\rho)\right),$$
where $\eta$ runs over a set of representatives for the $W_F$-orbits of irreducible representations of $I_F/I$.
Moreover, the map\footnote{I.e. viewed as a map $N: I_F \twoheadrightarrow I_F/P_F \simeq \prod_{\ell'} \overline{\mathbb{Q}}_{\ell'} \twoheadrightarrow \bQl \simeq \C \rightarrow \GL_n(R)$.} $N$ is $I_F$-equivariant, and thus induces, for each $\eta$, a nilpotent endomorphism
$N_{\eta}$ of $\Hom_{I_F}(\eta,\rho)$.  If $\Fr_{\eta}$ is a Frobenius element of $W_{\eta}$,
we have $\Fr_{\eta} N_{\eta} \Fr_{\eta}^{-1} = q^{r_{\eta}} N_{\eta}$.

\medskip

Let $n_{\eta}(\rho)$ be the dimension of the space $\Hom_{I_F}(\eta,\rho)$; since $n_{\eta}(\rho)$ only depends on the type $\nu$ of $\rho$, we may also write this as $n_{\eta}(\nu)$.
A choice of $R$-basis for $\Hom_{I_F}(\eta,\rho)$ then gives a homomorphism:
$$\rho_{\eta}: W_F/I_F \rightarrow \GL_{n_i}(R)$$
and realizes $N_{\eta}$ as a nilpotent element of $M_{n_i}(R)$ such that
$(\rho_{\eta},N_{\eta})$ is an $R$-point of $X^1_{E_{\eta}, \GL_{n_{\eta}(\rho)}}$.  We thus define:
\begin{defn}
A {\em pseudo-framing} of a Langlands parameter $(\rho,N)$ over $R$ is a choice, for all $\eta$ such that $n_{\eta}(\rho)$ is nonzero,
of an $R$-basis for $\Hom_{I_F}(\eta,\rho)$.  Let ${\wt X}_{F,G}^{\nu}$ be the moduli scheme parameterizing parameters $(\rho,N)$ of type $\nu$ together with a pseudo-framing, and define
$$G_\nu := \prod_{\{\eta \mid n_\eta(\nu) \ne 0\}} \GL_{n_\eta}.$$
The scheme $\wt{X}_{F,G}^\nu$ is equipped with a $G \times G_\nu$-action.
\end{defn}

We denote by $E_{\eta}$ the fixed field
of $W_{\eta}$, by $r_{\eta}$ the degree of $E_{\eta}$ over $F$, and by $d_{\eta}$  the dimension of $\eta$.  We see that $G_{\nu}$
acts on ${\wt X}_{F,G}^{\nu}$ via ``change of pseudo-framing'', and this action makes ${\wt X}_{F,G}^{\nu}$ into a
$G_{\nu}$-torsor over $X^{\nu}_{F,G}$.  On the other hand, given an $R$-point $(\rho,N)$ of ${\wt X}_{F,G}^{\nu}$, the pseudo-framing gives, for each $\eta$,
an $R$-point $(\rho_{\eta},N_{\eta})$ of $X^1_{E_{\eta},\GL_{n_{\eta}(\nu)}}$.  We thus obtain a natural map:
$${\wt X}_{F,G}^{\nu} \rightarrow \prod_{\eta} X^1_{E_{\eta},\GL_{n_{\eta}(\nu)}}$$
which is a torsor for the conjugation action of $G$ on $\wt{X}_{F,G}^\nu$.  We thus obtain natural isomorphisms of quotient stacks:
$$X_{F,G}^{\nu}/G \cong {\wt X}_{F,G}^{\nu}/(G \times G^{\nu}) \cong \left(\prod_{\eta} X^1_{E_{\eta},\GL_{n_{\eta}(\nu)}} \right)/G_{\nu} \simeq \prod_\eta \LL_{q^{r_{\eta}}}(\wh{\cN}_{n_{\eta}(\nu)}/\GL_{n_{\eta}(\nu)}).$$
Note that the composite isomorphism depends on the choice, for each $\eta$, of an extension ${\tilde \eta}$ of $\eta$ to $W_F$.

\medskip
\subsection{The \texorpdfstring{$\nu$}{v}-Springer sheaves}
%
We define a Springer sheaf by transporting across the above isomorphism.
\begin{defn}
We define the \emph{$\nu$-Springer sheaf} $S_\nu \in \Coh(X^\nu_{F,G}/G)$ to be the product, over $\eta$, of the sheaves $\cS_{q^{r_{\eta}}}$ on the moduli stack $X^1_{E_{\eta},\GL_{n_{\eta}(\nu)}}/\GL_{n_{\eta}(\nu)}$.
\end{defn}

By Theorem \ref{thm endomorph}, the endomorphisms of the $\nu$-Springer sheaf are a tensor product of affine Hecke algebras, and we introduce the notation
$$\Haff_\nu := \bigotimes_{\eta} \Haff_{q^{r_{\eta}}}(n_{\eta}(\nu)).$$
We thus obtain a fully faithful embedding $D(\Haff_\nu) \hookrightarrow \IndCoh(X_{F,G}^\nu/G)$.  However, since our identifications depend, ultimately, on our choices of $\tilde \eta$, this embedding will also depend on these choices.  (By contrast, the sheaf $\cS_{\nu}$
itself is, at least up to isomorphism, independent of the choices of $\tilde \eta$.)  We can remove this dependence
by rephrasing this embedding in terms of smooth representations of $G^{\vee}$, via the type theory of the previous section.

\begin{prop}
There is a $G$-type $(K_\nu, \tau_\nu)$ such that $\cH(G^\vee, K_\nu, \tau_\nu) \simeq \Haff_\nu$ (depending on choices), and an identification of dg algebras
$$\End^\bullet(\cS_{\nu}) \simeq \cH(G^{\vee},K_{\nu},\tau_{\nu})$$
which is is independent of the choices of ${\tilde \eta}$.
\end{prop}
\begin{proof}
Let $L_{\nu}^{\vee}$ be the standard Levi of $G^{\vee}$ corresponding to block diagonal matrices whose blocks consist, for each $\eta$, of $n_{\eta}(\nu)$ blocks
of size $r_{\eta} d_{\eta}$.  Let $\pi^0_{\eta}$ be the cuspidal representation of $\GL_{r_{\eta} d_{\eta}}$ corresponding to $\Ind_{W_{\eta}}^{W_F} \tilde \eta$
under the local Langlands correspondence, and let $\pi_{\nu}$ be the cuspidal representation:
$$\pi_{\nu} := \bigotimes_{\eta} (\pi^0_{\eta})^{\otimes n_{\eta}(\nu)}$$
of $L_{\nu}^{\vee}$.  Then representations in the block $D(G^{\vee})_{[L^{\vee}_{\nu},\pi_{\nu}]}$ correspond, via local Langlands, to Langlands parameters for $G$
of type $\nu$.

For each $\eta$, we can find a cuspidal type $(K_{\eta},\tau_{\eta})$ in $\GL_{r_{\eta}d_{\eta}}$ 
for $\pi^0_{\eta}$.  From this we can form the type $(K_{L_{\nu}},\tau_{L_{\nu}})$ in $L_{\nu}^{\vee}$, by setting
$K_{L_{\nu}} = \prod_{\eta} K_{\eta}^{n_{\eta}(\nu)}$ and $\tau_{L_{\nu}} = \bigotimes_{\eta} \tau_{\eta}^{\otimes n_{\eta}(\nu)}$.  This type is associated to
the block $[L^{\vee}_{\nu},\pi_{\nu}]$ in $D(L_{\nu}^{\vee})$.  Let $P^{\vee}$ be the standard parabolic of $G^{\vee}$ with Levi $L^{\vee}$, and let $(P')^{\vee}$ denote
the opposite parabolic.  The theory of section~\ref{types} then gives us a Levi subgroup $(L^{\dagger})^{\vee}$ of $G^{\vee}$ containing $L^{\vee}_{\nu}$,
an $(L^{\dagger})^{\vee}$-cover $(K^{\dagger}_{\nu},\tau^{\dagger}_{\nu})$ of $(K_{L_{\nu}},\tau_{L_{\nu}})$,
and a $G^{\vee}$-cover $(K_{\nu},\tau_{\nu})$ of $(K^{\dagger}_{\nu},\tau^{\dagger}_{\nu})$.
These covers depend on a choice of parabolic with Levi $L^{\vee}$; we choose our covers to be the ones associated to {\em the opposite parabolic} $(P')^{\vee}$.
In particular we obtain a map 
$$T_{(P')^{\vee}}: \cH(L_{\nu}^{\vee}, K_{L_{\nu}^\vee}, \tau_{L_{\nu}^\vee}) \rightarrow \cH(G^{\vee}, K_{\nu}, \tau_{\nu})$$ 
that is compatible with the parabolic induction functor $i_{P^{\vee}}^{G^{\vee}}$ on $D(L_{\nu}^{\vee})$ in the sense of Theorem~\ref{thm:cover}.

One verifies, by compatibility of local Langlands with unramified twists, that for each $\eta$ the group of unramified characters $\chi$
of $\GL_{r_{\eta}d_{\eta}}$ such that $\pi^0_{\eta} \otimes \chi$ is isomorphic to $\pi^0_{\eta}$ is $r_{\eta}$.  Thus there is an isomorphism
of Hecke algebras $\cH(G^{\vee},K_{\nu},\tau_{\nu}) \simeq \Haff_{\nu}$.  Moreover, the composition:
$$\cH(G^{\vee},K_{\nu},\tau_{\nu}) \cong \Haff_{\nu} \cong \End(\cS_{\nu})$$
is independent of the choices of ${\tilde \eta}$.  This essentially boils down to the compatibility of the local Langlands correspondence with unramified twists
and parabolic induction.
\end{proof}

Since $D(G^{\vee})_{[L_{\nu}^\vee,\pi_{\nu}]}$ is canonically equivalent to the category of $\cH(G^{\vee},K_{\nu},\tau_{\nu})$-modules, and this equivalence
associates the representations $\cind_{K_{\nu}}^{G^{\vee}} \tau_{\nu}$ to the free $\cH(G^{\vee},K_{\nu},\tau_{\nu})$-module of rank one, we have shown:

\begin{thm} \label{derived local Langlands}
For each $\nu$ there is a natural fully faithful functor:
$$\operatorname{LL}_{G,\nu}: D(G^{\vee})_{[L_{\nu}^\vee,\pi_{\nu}]} \hookrightarrow \IndCoh(X_{F,G}^{\nu})$$
that takes the generator $\cind_{K_{\nu}}^{G^{\vee}} \tau_{\nu}$ to $\cS_{\nu}$.
\end{thm}

\begin{rmk} \label{cuspidal skyscrapers}
We will say that an inertial type $\nu$ is {\em cuspidal} if the representations of $W_F$ corresponding to points of $X^{\nu}_{F,G}$ are irreducible.  For $G = \GL_n$ this happens precisely when $n_{\eta} = 1$
for a single $\eta$ and is zero for all other $\eta$.  In such cases $X^{\nu}_{F,G}$ is simply a copy of ${\mathbb G}_m$, the sheaf $\cS_{\nu}$ is the structure sheaf, and the corresponding
affine Hecke algebra is simply $\C[T,T^{-1}]$, which our choices above identify with the global functions on $X^{\nu}_{F,G} \cong {\mathbb G}_m$.  In particular for such $\nu$ the functor
$\operatorname{LL}_{G,\nu}$ is an abelian equivalence, that takes an irreducible $\C[T,T^{-1}]$-module to a skyscraper sheaf on the corresponding point of $X^{\nu}_{F,G}$.

By taking products of the above picture we see that a similar statement holds for Levi subgroups $M$ of $G$ (with a suitable torus in place of ${\mathbb G}_m$.)
\end{rmk}

\medskip

\subsubsection{A direct construction of $\cS_{\nu}$}
In this section we give a more intrinsic construction of $\cS_{\nu}$.  Fix a particular $\nu$, and let $L_\nu$ denote the Langlands dual of $L_{\nu}^{\vee}$; we identify $L_\nu$
with the standard block diagonal Levi of $G$ containing $n_{\eta}(\nu)$ blocks of size $r_{\eta} d_{\eta}$.  Let $\nu': I_F \rightarrow L_\nu$ be the representation of $I_F$ on $L$ whose projection to each block of $L_\nu$ of type $\eta$ is the sum of the $W_F$-conjugates of $\eta$.
We then have a moduli space $X_{F,L_\nu}^{\nu'}$ parameterizing Langlands parameters for $L_\nu$ that are of type $\nu'$.

Let $P$ be the standard (block upper triangular) parabolic of $G$ containing $L_{\nu}$.  We then also have a moduli space $X_{F,P}^{\nu'}$ parameterizing Langlands parameters
for $G$ that factor through $P$, and whose projection to $L_{\nu}$ is of type $\nu'$.  The inclusion of $P \hookrightarrow G$, and the projection of $P \twoheadrightarrow L$ induce parabolic induction maps
$$\begin{tikzcd}
X^{\nu'}_{F, L_\nu} & \arrow[l, "\pi_P"'] X^{\nu'}_{F,P} \arrow[r, "\iota_P"] & X^{\nu'}_{F,G}
\end{tikzcd}$$
We then have:
\begin{thm} \label{Springer induction}
There are natural isomorphisms:
$$\cS_{\nu} \cong (\iota_P)_* {\mathcal O}_{F,P}^{\nu'} \cong (\iota_P)_* \pi_P^* {\mathcal O}_{F,L_\nu}^{\nu'},$$
where ${\mathcal O}_{F,P}^{\nu'}$ and ${\mathcal O}_{F,L_\nu}^{\nu'}$ denote the structure sheaves on $X_{F,P}^{\nu'}/P$ and $X_{F,L_\nu}^{\nu'}/L_\nu$, respectively.
\end{thm}
\begin{proof}
Let $L^{\dagger}$ be the standard Levi of $G$ that is block diagonal of block sizes $n_{\eta}(\nu) r_{\eta} d_{\eta}$. 
Let $Q$ be the standard block upper triangular parabolic of $G$ with Levi $L^{\dagger}$, and
let $\nu''$ be the composition of $\nu'$ with the inclusion of $L_\nu$ in $L^{\dagger}$. 
We then have spaces $X_{F,L^{\dagger}}^{\nu''}$ and $X_{F,Q}^{\nu''}$, where the former parameterizes pairs $(\rho,N)$ for $L^{\dagger}$ that are of type $\nu''$, and the latter parameterizes pairs $(\rho,N)$ for $G$ that
factor through $Q$ and whose projection to $L^{\dagger}$ is of type $\nu''$.  We may also consider the space $X_{F,P \cap L^{\dagger}}^{\nu'}$, which parameterizes pairs $(\rho,N)$ for $L^{\dagger}$ that factor
through $P \cap  L^{\dagger}$ and whose projection to $L$ is of type $\nu'$.  We then have a natural Cartesian diagram:

$$\begin{tikzcd}
X_{F,P}^{\nu'} / P \arrow[r] \arrow[d] & X_{F,P \cap L^{\dagger}}^{\nu'} / P \cap L^{\dagger}\arrow[d, "\iota_{P \cap L^\dagger}"] \\
X_{F,Q}^{\nu''} / Q \arrow[r, "\iota_Q"] & X_{F,L^{\dagger}}^{\nu''} / L^{\dagger}
\end{tikzcd}
$$
from which we conclude that $(\iota_P)_* \pi_P^* {\mathcal O}_{F,L_\nu}^{\nu'}$ is isomorphic to
$ (\pi_Q)_* \iota_Q^* (\iota_{P \cap L^{\dagger}})_* \pi_{P \cap  L^{\dagger}}^* {\mathcal O}_{F,L_\nu}^{\nu'}$,
where $\pi_Q: X_{F,Q}^{\nu''}/Q \rightarrow X_{F,G}^{\nu}/G$, and
$\pi_{P \cap L^{\dagger}}: X_{F, P \cap L^{\dagger}}^{\nu'}/ (P \cap L^{\dagger}) \rightarrow X_{F,L_\nu}^{\nu'}/L_\nu$.

On the other hand, let $B_{\eta}$ and $T_{\eta}$ denote the standard Borel subgroup and maximal torus of $\GL_{n_{\eta}(\nu)}$, for each $\eta$.  We then have
a commutative diagram (note that we transport derived structures across the isomorphisms by definition):


$$
\begin{array}{ccc}
\prod_{\eta} \LL_{q^{r_{\eta}}}(\wh{\cN}_{T_{\eta}}/{ T}) & \cong & X_{F,L}^{\nu'}/L_\nu\\
\uparrow & & \uparrow\\
\prod_{\eta} \LL_{q^{r_{\eta}}}(\wh{\cN}_{B_{\eta}}/{ B}) & \cong & X_{F,P \cap L^{\dagger}}^{\nu'} / (P \cap L^{\dagger})\\
\downarrow & & \downarrow\\
\prod_{\eta} \LL_{q^{r_{\eta}}}(\wh{\cN}_{n_{\eta}(\nu)}/{ G_{n_{\eta}(\nu)}}) & \cong & X_{F, L^{\dagger}}^{\nu''} / L^{\dagger}\\
\uparrow & & \uparrow\\
\prod_{\eta} \LL_{q^{r_{\eta}}}(\wh{\cN}_{n_{\eta}(\nu)}/{ G_{n_{\eta}(\nu)}}) & \cong & X_{F,Q}^{\nu''} / Q\\
\downarrow & & \downarrow\\
\prod_{\eta} \LL_{q^{r_{\eta}}}(\wh{\cN}_{n_{\eta}(\nu)}/{ G_{n_{\eta}(\nu)}}) & \cong & X_{F,G}^{\nu} / G\\
\end{array}
$$
where the bottom two vertical maps on the left are the identity.  It follows that the iterated pull-push
$(\iota_Q)_* \pi_Q^* (\iota_{P \cap L^{\dagger}})_* \pi_{P \cap L^{\dagger}}^* {\mathcal O}_{F,L_\nu}^{\nu'}$ corresponds, under the bottom isomorphism, to $\cS_{\nu}$, as the latter is simply the pushforward to
$\prod_{\eta} \LL_{q^{r_{\eta}}}(\wh{\cN}_{n_{\eta}(\nu)}/{ G_{n_{\eta}(\nu)}})$ of the structure sheaf on
$\prod_{\eta} \LL_{q^{r_{\eta}}}(\wh{\cN}_{B_{\eta}}/{B})$.
\end{proof}


\medskip

\subsubsection{Compatibility with parabolic induction}

As in the previous subsection, we fix a particular $\nu$ and let $L_{\nu}^{\vee}$, $L_\nu$ and $P$ be as above.  Let $Q$ be a standard Levi subgroup of $G$ whose standard Levi subgroup $M$
contains $L_\nu$, and let $M^{\vee}$ and $Q^{\vee}$ be the corresponding dual subgroups of $G^{\vee}$.  Let $\nu'$ be the inertial type $I_F \rightarrow L_{\nu}$ constructed in the
previous subsection, and let $\nu''$ be the composition of $\nu'$ with the inclusion of $L_{\nu}$ in $M$.  We have a diagram with the square Cartesian:
$$
\begin{tikzcd}[column sep=large]
X_{F,L_{\nu}}^{\nu'}/L_{\nu} & \arrow[l, "\pi_{P \cap M}"']   \arrow[d, "\iota_{P \cap M}"] X_{F,P \cap M}^{\nu'} / P \cap M   & X_{F,P}^{\nu'} /P \arrow[d, "\iota_{P,Q}"] \arrow[l, "\pi_{P, P \cap M}"'] \\
& X_{F,M}^{\nu''}/M & \arrow[l, "\pi_Q"']  \arrow[d, "\iota_Q"] X_{F,Q}^{\nu''}/Q \\
& & X_{F,G}^{\nu}/G.
\end{tikzcd}
$$

Theorem \ref{Springer induction} shows that $\cS_{\nu}$ is isomorphic to the pushforward to $X_{F,G}^{\nu}/G$ of the structure sheaf on $X_{F,P}^{\nu'} / P$, and the corresponding sheaf
$\cS_{\nu,M}$ on $X_{F,M}^{\nu''}$ is the pushforward to $X_{F,M}^{\nu''} / M$ of the structure sheaf on $X_{F,P \cap M}^{\nu'}/ (P \cap M)$.  The above diagram then gives us a natural isomorphism:
$$\cS_{\nu} \cong (\iota_Q)_* \pi_Q^* \cS_{\nu,M}.$$
Via functoriality and this isomorphism one obtains an embedding of $\End(\cS_{\nu,M})$ in $\End(\cS_{\nu})$.

Recall that we have identified these endomorphism rings with certain Hecke algebras via type theory.  In particular, we have the type $(K_{L_{\nu}},\tau_{L_{\nu}})$ of $L_{\nu}^{\vee}$,
an $M^{\vee}$-cover $(K_{M^{\vee}},\tau_{M^{\vee}})$ coming from the parabolic $(P')^{\vee} \cap M^{\vee}$ opposite $P^{\vee} \cap M^{\vee}$, and a $G^{\vee}$-cover $(K,\tau)$ coming
from the parabolic $(P')^{\vee}$ opposite $P^{\vee}$.  Theorem~\ref{thm:cover induction} then gives us a map:
$$T_{(Q')^{\vee}}: \cH(M^{\vee},K_{M^{\vee}}, \tau_{M^{\vee}}) \rightarrow \cH(G^{\vee},K,\tau).$$

\begin{lemma} 
We have a commutative diagram:
$$
\begin{tikzcd}
\cH(M^{\vee},K_{M^{\vee}},\tau_{M^{\vee}}) \arrow[r, "\simeq"] \arrow[d,"T_{(Q')^{\vee}}"'] &  \End(\cS_{\nu,M}) \arrow[d] \\
\cH(G^{\vee},K,\tau) \arrow[r, "\simeq"] & \End(\cS_{\nu})
\end{tikzcd}
$$
where the right hand map is induced by the isomorphism of $\cS_{\nu} \simeq (\iota_Q)_* \pi_Q^* \cS_{\nu,M}$.
\end{lemma}
\begin{proof}
The machinery of the previous subsection, together with the compatibility of the general case with the Iwahori case in section~\ref{types} allow us to reduce to the case where $\nu = 1$.
In this case the claim reduces to the compatibility of the Ginsburg-Kazhdan-Lusztig interpretation of the affine Hecke algebra as $K_0$ of the Steinberg variety with parabolic induction, checked in the proof of Theorem \ref{thm endomorph}.
\end{proof}

As a consequence, we deduce:

\begin{thm} \label{thm:LL parabolic induction}
We have a commutative diagram of functors:
$$
\begin{tikzcd}
D(M^{\vee})_{[L_{\nu},\tau_{\nu}]} \arrow[rr, "{\operatorname{LL}_{M, \nu}}", hook]  \arrow[d, "i_{Q^{\vee}}^{G^{\vee}}"'] & & \IndCoh(X_{F,M}^{\nu}) \arrow[d, "(\iota_Q)_* \pi_Q^*"] \\
D(G^{\vee})_{[L_{\nu},\tau_{\nu}]}  \arrow[rr, hook, "{\operatorname{LL}_{G,\nu}}"] & & \IndCoh(X_{F,G}^{\nu}).
\end{tikzcd}$$
\end{thm}

\begin{proof}
We have isomorphisms:
\begin{eqnarray*}
\operatorname{LL}_{G,\nu} (i_{Q^{\vee}}^{G^{\vee}} V) & \cong & \Hom(\cind_K^{G^{\vee}}\tau, i_{Q^{\vee}}^{G^{\vee}} V) \otimes_{\cH(G^{\vee},K,\tau)} \cS_{\nu}\\
& \cong & \Hom_{M^{\vee}}(\cind_{K_{M^{\vee}}}^{M^{\vee}} \tau_{M^{\vee}}, V) \otimes_{\cH(M^{\vee}, K_{M^{\vee}}, \tau_{M^{\vee}})} (\iota_Q)_* \pi_Q^* \cS_{M^\vee,\nu}\\
& \cong & (\iota_Q)_* \pi_Q^* (\operatorname{LL}_{M,\nu} V)
\end{eqnarray*}
from which the result follows.
\end{proof}

\appendix
\section{Proofs}\label{app proofs}

This appendix contains proofs of technical results used in the body of the paper.


\subsection{Functoriality of Hochschild homology in geometric settings}\label{app functoriality}

\begin{proof}[Proof of Proposition \ref{HH prop}]
The first and second statements are Theorem 2.21 (or Proposition 5.5) in \cite{BN:NT}.  We give a direct argument for the third statement (which can also be adapted toward the second).  
We let $Z := X \times_Y X$, and denote the diagonals by $\Delta_X: X \hookrightarrow X \times X$ (and likewise for $Y$), the relative diagonal by $\Delta: X \hookrightarrow Z = X \times_Y X$, and its inclusion by $i: Z =  X \times_Y X \hookrightarrow X \times X$.  

Note that we use $!$-integral transforms in our convention; thus to describe the integral transforms it is convenient to pass between $*$-pullbacks and $!$-pullbacks.  For any quasi-smooth map $g: E \rightarrow B$ we denote by $\beta_g^*: f^*(-) \simeq f^!(-) \otimes_{\cO_X} \omega_{E/B}^{-1}$ and $\beta_g^!: f^!(-) \simeq f^*(-) \otimes_{\cO_X} \omega_{E/B}$ the canonical equivalences.

The integral transform corresponding to $f_* f^*: \Coh(Y) \rightarrow \Coh(Y)$ is given by the kernel 
$$\cK_{f_*f^*} := \Delta_{Y*} f_*(\omega_X \tens{\cO_X} \omega_{X/Y}^{-1}).$$
Letting $\eta_f$ denote the unit for the adjunction $(f^*, f_*)$, the unit  $\eta \in \Hom_{Y \times Y}(\Delta_{Y*} \omega_Y,\cK_{f_*f^*})$ is defined:
$$\eta := \Delta_{Y*}(\beta_f^* \circ \eta_{f}): \Delta_{Y*}\omega_Y \longrightarrow \Delta_{Y*}(f_* f^* \omega_Y)  \simeq \Delta_{Y*}(f_*(f^!\omega_Y \tens{\cO_X} \omega_{X/Y}^{-1})).$$

The integral transform corresponding to $f^* f_*: \Coh(X) \rightarrow \Coh(X)$ is given by the kernel:
$$\cK_{f^*f_*} := i_*(\omega_{Z} \tens{\cO_{Z}} \omega_{Z/X}^{-1}).$$
Letting $\eta_{\Delta}$ denote the unit for the adjunction $(\Delta^*, \Delta_{*})$, the counit $\epsilon \in \Hom_{X \times X}(\cK_{f^*f_*}, \Delta_{X*} \omega_X)$ is defined:
$$\epsilon := i_*(\beta_{\Delta}^{!-1} \circ \eta_{\Delta}): i_*(\omega_{Z} \tens{\cO_{Z}} \omega^{-1}_{Z/X}) \rightarrow i_* \Delta_{*}( \Delta^* \omega_{Z}\tens{\cO_{Z}} \omega_{X/Z}) \simeq i_*\Delta_{*}\omega_{X}$$
where we implicitly use the canonical identification $\Delta^* \omega_{Z/X}^{-1} \simeq \omega_{X/Z}$ (i.e. since $\omega_{X/X}$ is canonically trivial).  We leave verification of the adjunction identites to the reader.

The functoriality $\omega(\cL_\phi Y) \rightarrow \omega(\cL_\phi X)$ is given by composing the unit and counit after applying $\Gamma \circ \Gamma_{\phi_Y}^!$ and $\Gamma \circ \Gamma_{\phi_X}^!$ (where, somewhat confusingly, $\Gamma$ denotes the global sections functor, and $\Gamma_{\phi}$ denotes the graph).  Recall the factorization and notation of Lemma \ref{lemma rel loops}, 
 let $p_X: \cL_\phi Y_X  \rightarrow X$ and $p_Y: \cL_\phi Y_X \rightarrow Y$ denote the natural maps, and $\mathrm{ev}_X: \cL_\phi X \rightarrow X$ the evaluation (and likewise for $Y$).  
For the unit map $\eta$, we have
$$\Gamma_{\phi_Y}^! \eta: \Gamma_{\phi_Y}^! \Delta_{Y*} \omega_Y \longrightarrow \Gamma_{\phi_Y}^! \Delta_{Y*} f_*(\omega_X \otimes_{\cO_X} \omega_{X/Y}^{-1}).$$
We perform a base change along the diagram:
$$\begin{tikzcd}
\cL_\phi Y_X \arrow[r, "\pi"] \arrow[d, "p_X"] & \cL_\phi Y \arrow[r, "{\mathrm{ev}_Y}"] \arrow[d,"{\mathrm{ev}_Y}"] & Y \arrow[d, "\Gamma_{\phi_Y}"] \\
X \arrow[r, "f"] & Y \arrow[r, "\Delta_Y"] & Y \times Y.
\end{tikzcd}$$
to find 
$$\Gamma_{\phi_Y}^! \Delta_{Y*} f_*(\omega_X \otimes_{\cO_X} \omega_{X/Y}^{-1}) \simeq p_{Y*} p_X^! (\omega_X \otimes_{\cO_X} \omega_{X/Y}^{-1}) \simeq p_{Y*}(\omega_{\cL_\phi Y_X} \otimes_{\cO_{\cL_\phi Y_X}} p_X^*\omega_{X/Y}^{-1})$$
$$ \simeq p_{Y*}(\omega_{\cL_\phi Y_X} \otimes_{\cO_{\cL_\phi Y_X}} \omega_{\cL_\phi Y_X/\cL_\phi Y}^{-1}) \simeq \mathrm{ev}_{Y*} \pi_*\pi^* \omega_{\cL_\phi Y}$$
and an identification of $\eta$ with the unit $\eta_\pi$ for the adjunction $(\pi^*, \pi_*)$:
$$\eta \simeq \mathrm{ev}_{Y*}(\eta_\pi(\omega_{\cL_\phi Y})): \mathrm{ev}_{Y*} \omega_{\cL_\phi Y} \longrightarrow \mathrm{ev}_{Y*} \pi_* \pi^* \omega_{\cL_\phi Y}.$$
For the counit map $\epsilon$, we have
$$\Gamma_{\phi_X}^! \epsilon: \Gamma_{\phi_X}^! i_* (\omega_{Z} \otimes_{\cO_{Z}} \omega_{Z/X}^{-1}) \longrightarrow \Gamma_{\phi_X}^!\Delta_{X*} \omega_X.$$
We perform a base change along the diagram:
$$\begin{tikzcd}
\cL_\phi X \arrow[d, "{\mathrm{ev}_X}"] \arrow[r, "\delta"] & \cL_\phi Y_X \arrow[d, "s"] \arrow[r, "p_X"] & X \arrow[d, "\Gamma_{\phi_X}"] \\
X \arrow[r, "\Delta"] &Z\arrow[r, "i"] & X \times X.
\end{tikzcd}$$
to find that
$$\Gamma_{\phi_X}^! i_*(\omega_Z \otimes_{\cO_Z} \omega_{Z/X}^{-1}) \simeq p_{X*} s^!(\omega_Z \otimes_{\cO_Z} \omega_{Z/X}^{-1}) \simeq p_{X*}(\omega_{\cL_\phi Y_X} \otimes_{\cO_{\cL_\phi Y_X}} s^* \omega^{-1}_{Z/X})$$
$$\simeq p_{X*}(\omega_{\cL_\phi Y_X} \otimes_{\cO_{\cL_\phi Y_X}} \omega_{\cL_\phi Y_X/\cL_\phi Y}^{-1}) \simeq p_{X*} \delta^* \omega_{\cL_\phi Y}.$$
Due to the canonical Calabi-Yau equivalence $\omega_{\cL_\phi X/\cL_\phi Y} \simeq \cO_{\cL_\phi X}$ of Proposition \ref{calabi yau}, we have a canonical equivalence $\omega_{\cL_\phi X/ \cL_\phi Y_X} \simeq \delta^*\omega_{\cL_\phi Y_X/\cL_\phi Y}^{-1}$.  Passing through this equivalence, we have
$$\Gamma_{\phi_X}^! i_* \Delta_* \omega_X \simeq p_{X*} s^! \Delta_* \omega_X \simeq p_{X*} \delta_* \delta^! \omega_{\cL_\phi Y_X} \simeq p_{X*} \delta_*( \delta^* \omega_{\cL_\phi Y_X} \otimes_{\cO_{\cL_\phi X}} \omega_{\cL_\phi X/\cL_\phi Y_X})$$
$$\simeq p_{X*} \delta_*\delta^*(\omega_{\cL_\phi Y_X} \otimes_{\cO_{\cL_\phi Y_X}} \omega_{\cL_\phi Y_X/\cL_\phi Y}^{-1}).$$
Thus, $\epsilon$ is identified with the unit $\eta_\delta$ for the adjunction $(\delta^*, \delta_*)$:
$$\epsilon \simeq p_{X*}(\eta_\delta(\omega_{\cL_\phi X/\cL_\phi Y_X} \otimes_{\cO_{\cL_\phi Y_X}} \omega_{\cL_\phi Y_X/\cL_\phi Y}^{-1})): p_{X*} \pi^* \omega_{\cL_\phi Y} \rightarrow \mathrm{ev}_{X*} \omega_{\cL_\phi X}.$$
Taking global sections and composing, we see that the map 
$$\omega(\cL_\phi Y) \rightarrow \Gamma(\cL_\phi Y_X, \omega_{\cL_\phi Y_X} \otimes \omega^{-1}_{\cL_\phi Y_X/\cL_\phi Y}) \simeq \Gamma(\cL_\phi Y_X, \omega_{\cL_\phi Y_X} \otimes \delta_*\omega_{\cL_\phi X/\cL_\phi Y_X}) \rightarrow \omega(\cL_\phi X)$$ is induced by the unit of the adjunction $(\cL_\phi f^*, \cL_\phi f_*)$, twisted by the Calabi-Yau equivalence.
\end{proof}

The following is a generalization of Proposition \ref{conv volume}.  While Proposition \ref{conv volume} is stated in the setting of derived loop spaces, the arguments hold in the following more general setting.
\begin{prop}\label{app conv volume}
Let $f: X \rightarrow Y$ be a proper map of derived stacks, and let $Z = X \times_Y X$ with projections $p_1, p_2: Z \rightarrow X$ and $p: Z \rightarrow Y$.  There is a canonical equivalence:
$$\zeta_f: p_*\intHom_Z(\cO_Z, \omega_Z) \simeq \intHom_Y(f_*\cO_X, f_*\omega_X).$$
In particular, if $X$ is Calabi-Yau, then we have a natural equivalence $\omega(Z) \simeq \End_Y(f_*\omega_X)$.  This equivalence is functorial in the following sense.  Let $f': X' \rightarrow Y'$ (and $p': Z' \rightarrow Y'$) be as above.
\begin{itemize} 
\item Suppose that $\alpha_Y: Y \rightarrow Y'$ is proper, and that $X = X'$.  We let $f: X \rightarrow Y$ be as above, $f' = \alpha_Y \circ f: X \rightarrow Y \rightarrow Y'$.  We have commuting squares
$$\begin{tikzcd}
\alpha_{Y*} p_*\intHom_Z(\cO_Z, \omega_Z) \arrow[r, "\simeq"', "\alpha_{Y*} (\zeta_f)"] \arrow[d, "{\mathrm{Def. \ref{HH decat functoriality}}}"'] & \alpha_{Y*} \intHom_Y(f_*\cO_X, f_*\omega_X) \arrow[d, "{\mathrm{Def. \ref{springer functorial}}}"] \\
p'_*\intHom_{Z'}(\cO_{Z'}, \omega_{Z'}) \arrow[r, "\simeq"', "\zeta_{f'}"] & \intHom_{Y'}(f_*\cO_{X'}, f_*\omega_{X'}).
\end{tikzcd}$$
\item Suppose that $\alpha_Y: Y \rightarrow Y'$ is Calabi-Yau, and that $X = X' \times_{Y'} Y$ (so $\alpha_X$ is also Calabi-Yau).  Then we have commuting squares
$$\begin{tikzcd}
p'_*\intHom_{Z'}(\cO_{Z'}, \omega_{Z'})\arrow[d, "{\mathrm{Def. \ref{HH decat functoriality}}}"'] \arrow[r, "\simeq"', "\zeta_{f'}"] & \intHom_{Y'}(f_*\cO_{X'}, f_*\omega_{X'}) \arrow[d, "{\mathrm{Def. \ref{springer functorial}}}"] \\
\alpha_{Y*} p_*\intHom_Z(\cO_Z, \omega_Z) \arrow[r, "\simeq"', "\alpha_{Y*} (\zeta_f)"]  & \alpha_{Y*} \intHom_Y(f_*\cO_X, f_*\omega_X).
\end{tikzcd}$$
\end{itemize}
\end{prop}
\begin{proof}
The first statement is a formal consequence of adjunctions and base change:
$$p_* \intHom_Z(\cO_Z, \omega_Z) \simeq f_*\intHom_X(\cO_X, p_{1*}\omega_Z) \simeq f_* \intHom_X(\cO_X, f^!f_*\omega_X) \simeq \intHom_Y(f_*\cO_X, f_*\omega_X).$$
Functoriality for proper morphisms follows by a diagram chase on:
$$\begin{tikzcd}
\alpha_{Y*} p_* \intHom_Z(\cO_Z, \omega_Z) \arrow[d, "\simeq"] \arrow[r] & p'_* \intHom_{Z'}(\cO_{Z'}, \omega_{Z'})  \arrow[d, "\simeq"] \\
f'_*\intHom_X(\cO_X, p_{1*}\omega_Z) \arrow[d, "\simeq"] \arrow[r] &  f'_*\intHom_{X}(\cO_{X}, p'_{1*}\omega_{Z'}) \arrow[d, "\simeq"]  \\
f'_* \intHom_X(\cO_X, f^!f_*\omega_X)  \arrow[d, "\simeq"] \arrow[r] &   f'_* \intHom_{X}(\cO_{X}, f'^!f'_*\omega_{X})\arrow[d, "\simeq"] \\
   \alpha_{Y*} \intHom_Y(f_*\cO_X, f_*\omega_X) \arrow[r] & \alpha_{Y*} \intHom_{Y'}(f'_*\cO_{X}, f'_*\omega_{X})
\end{tikzcd}$$
where we use the identification in the middle left terms $\alpha_{Y*} f_* \simeq f'_*\alpha_{X*} \simeq f'_*$ (i.e. since $X=X'$ and $\alpha_X = \mathrm{id}_X$), and the middle horizontal maps are given by functoriality of pushforwards of dualizing sheaves.  In the Calabi-Yau case, we pass to left adjoints, apply the base change $\alpha^*f'_* \simeq f_*\alpha^*$ and chase the diagram:
$$\begin{tikzcd}
p_* \alpha_{Z}^* \intHom_{Z'}(\cO_{Z'}, \omega_{Z'}) \arrow[r] \arrow[d, "\simeq"] &p_* \intHom_Z(\cO_Z, \omega_Z) \arrow[d, "\simeq"] \\
f_*\alpha_{X}^*\intHom_{X'}(\cO_{X'}, p'_{1*}\omega_{Z'})\arrow[d, "\simeq"]  \arrow[r] &f_*\intHom_X(\cO_X, p_{1*}\omega_Z) \arrow[d, "\simeq"] \\
f_*\alpha_{X}^* \intHom_{X'}(\cO_{X'}, f'^!f'_*\omega_{X'})\arrow[r] \arrow[d, "\simeq"] &f_* \intHom_{X'}(\cO_X, f^!f_*\omega_X) \arrow[d, "\simeq"] \\
\alpha_{Y}^*\intHom_{Y'}(f'_*\cO_{X'}, f'_*\omega_{X'}) \arrow[r] &\intHom_Y(f_*\cO_X, f_*\omega_X)
\end{tikzcd}$$
where the middle arrows arise by functoriality of Calabi-Yau pullback (as in Definition \ref{HH decat functoriality}) after passing to left adjoints.
\end{proof}

\subsection{Horizontal trace of convolution categories}\label{app horiz}

\begin{proof}[Proof of Theorem \ref{convolution category trace}]
We will employ the notation in Theorem 3.3.1 of \cite{BNP} to point out how its argument can be modified to acommodate this more general setting.  First, note that the surjectivity condition is not needed nor used in the proof of the theorem; it is subsumed by the singular support condition, so we omit it from the statement.  The quasi-smoothness of $q_n$ follows by quasi-smoothness of the graph $\Gamma_\phi$. 
We replace, in the definition of $\cC_\bullet$, the diagonal module $\Perf(X)$ with the module defined by the graph $\Gamma_\phi$.  In the definition of $Z_\bullet$, this amounts to replacing $\cL Y$ with $Y^\phi$ (informally, introducing a twist by $\phi$ as we ``come around the circle,'' i.e. in Lemma 3.3.2 the automorphism $\ell$ lives in $\Map_{Y(k)}(y, \phi(y))$).  In the definition of $W_\bullet$, this amounts to replacing the last factor of $X \times_Y X = X \times_{f,Y,f} X$ representing the ``segment containing the twist by $\phi$'' with $X \times_{f,Y,\phi_X \circ f} X$ (i.e. in Lemma 3.3.3, the final point $x_n$ should lie in the fiber $f^{-1}(\phi(y))$ rather than $f^{-1}(y)$).  The rest of the proof goes through without modification as the formulas still hold with the $\phi$-twist.
\end{proof}


\begin{proof}[Proof of Proposition \ref{conv comp}]
The argument in Theorem 3.3.1 of \cite{BNP} may be adapted in the following way. Let $\cat{M} = \IndCoh(Z_{12})$ and $\cat{N} = \IndCoh(Z_{23})$, and following the notation of \emph{loc. cit.} we let $\cat{A} = \IndCoh(Z_{22})$ and $\cat{B} = \IndCoh(X_2)$.  Then, writing $\cat{M} \otimes_{\cat{A}} \cat{N} = \cat{M} \otimes_{\cat{A}} \cat{A} \otimes_{\cat{A}} \cat{N},$ and (following the argument of \emph{loc. cit.}) resolving $\cat{A}$ as a $\cat{A} \otimes_{\cat{B}} \cat{A}^{rv}$-module via the relative bar complex for $\cat{A}$ over $\cat{B}$, we find that $\cat{M} \otimes_{\cat{A}} \cat{N}$ can be realized as the geometric realization of the cosimplicial object:
$$\cat{M} \otimes_{\cat{A}} \cat{N} = \colim(\IndCoh_{\Lambda_n}(Z_n))$$
where we define
$$q_n: Z_n := X_1 \utimes{Y} \overbrace{X_2 \utimes{Y} \cdots \utimes{Y} X_2}^{n+1} \utimes{Y} X_3 \longrightarrow  W_n := Z_{12} \times Z_{22}^{n} \times Z_{23},$$
$$ \Lambda_n = q_n^!(\Lambda_{12} \boxtimes \overbrace{\TT_{Z_{22}} \boxtimes \cdots \boxtimes \TT_{Z_{22}}}^{n} \boxtimes \,\Lambda_{23}).$$
Explicitly, for $\eta = (x_1, x_2^{(0)}, \ldots, x_2^{(n)}, x_{3}) \in Z_n(k)$ with each coordinate living in the fiber over $y \in Y(k)$, we have
$$\TT_{Z_n} = \{(\omega_{12}, \omega_{22}^{(01)}, \ldots, \omega_{22}^{(n-1,n)}, \omega_{23}) \in \bT^*_{Y,y} \mid df_1^* \omega_{12} = 0, df_3^* \omega_{23} = 0, df_2^* \omega_{22}^{(i,j)} = df_2^*\omega_{22}^{(i',j')} \},$$
$$\Lambda_n =  \{(\omega_{12}, \omega_{22}^{(01)}, \ldots, \omega_{22}^{(n-1,n)}, \omega_{23}) \in \bT^*_{Y,y} \mid \omega_{12} \in \Lambda_{12, \eta}, \omega_{23} \in \Lambda_{23, \eta}, df_2^* \omega_{22}^{(i,j)} = 0 \}.$$
Here, we note that the fiber of the singular support condition $\Lambda_{ij}$ at the point $(x_i, x_j) \in Z_{ij}(k)$  in the fiber over $y$ is naturally a subset $\Lambda_{ij, (x_i, x_j)} \subset \bT^*_{Y,y}$.  The singular support stability condition implies that the face maps $(Z_m, \Lambda_m) \rightarrow (Z_n, \Lambda_n)$ are maps of pairs.  Pullback along the augmentation is conservative by definition of $\Lambda_{13}$.  Analogous formulas in Lemma 3.3.9 of \emph{op. cit.} hold in this situation (without the need to ``loop around''), and the strictness condition follows by an argument analogous to Proposition 3.3.8 of \emph{op. cit.}  Thus, we have an equivalence
$$\IndCoh_{\Lambda_{13}}(Z_{13}) \simeq \Tot(\IndCoh_{\Lambda_n}(Z_n)).$$

For functoriality, we note that the resulting maps $(Z_n, \Lambda_n) \rightarrow (Z_n, \Lambda'_n)$ are maps of pairs by our description above for $n \geq 0$, and the case $n=-1$ is a straightforward verification.  The claim then follows by functoriality of the descent with support discussed in Section 2.4 of \cite{BNP}.  We adopt the notation of \emph{loc. cit.}: 
let $(X_\bullet, \Lambda_\bullet) \rightarrow (X_{-1}, \Lambda_{-1})$ and $(Y_\bullet, \Theta_\bullet) \rightarrow (Y_{-1}, \Theta_{-1})$ be augmented simplicial diagrams of maps of pairs satisfying the descent conditions of Theorem 2.4.1 and Corollary 2.4.2 of \cite{BNP}, and let $g_\bullet: (X_\bullet, \Lambda_\bullet) \rightarrow (Y_\bullet, \Theta_\bullet)$ be a level-wise proper map of augmented simplicial diagrams of pairs.  We claim that we have a limit $\Tot(\mf{g}_\bullet^!) \simeq \mathfrak{g}_{-1}^!$ and a colimit $\mathrm{Real}(\mf{g}_{\bullet *}) \simeq \mathfrak{g}_{-1*}$, which proves the functoriality claims (i.e. since the maps $\mf{g}_\bullet$ are the identity, the functors $\mf{g}_{\bullet *}$ are the inclusion functors and $\mf{g}_{\bullet}^!$ are the local cohomology functors).  The first statement follows by commutativity of $!$-pullbacks with supports (see Remark 2.3.3 of \cite{BNP}) and by universal property of the limit.  The second statement follows by passing to left adjoints (as in Corollary 2.4.2 of \emph{op. cit.}).
\end{proof}

\begin{proof}[Proof of Proposition \ref{qcoh trace}]
Consider the functors
$$T(-) := - \otimes_{\QCoh(k)} \QCoh(X): \cat{dgCat}_k \rightarrow \QCoh(Y)\mh\cat{mod},$$
$$T^R(-) := - \otimes_{\QCoh(Y)} \QCoh(X): \QCoh(Y)\mh\cat{mod} \rightarrow \cat{dgCat}_k.$$
We claim that $(T, T^R)$ are adjoint.  Let $\Delta_X: X \rightarrow X \times X$ denote the diagonal, $p: X \rightarrow \pt$ denote the structure map, and $\Delta_{X/Y}: X \rightarrow X \times_Y X$ the relative diagonal.  We define the unit $\eta: \mathrm{id}_{\cat{dgCat}_k} \rightarrow T^R \circ T$ via the functor $\Delta_{X/Y*}p^*: \QCoh(\pt) \rightarrow \QCoh(X \times_Y X)$ and the counit $\epsilon: T \circ T^R \rightarrow \mathrm{id}_{\QCoh(Y)\mh\cat{mod}}$ by the functor $f_*\Delta_X^*: \QCoh(X \times X) \rightarrow \QCoh(Y).$    Verification of the adjunction axioms is a straightforward application of base change and Theorem 4.7 of \cite{BFN}.  To compute the trace, we apply base change and find that $[\QCoh(X), \phi_{X*}]$ is the pull-push of $k \in \QCoh(\pt)$ along the diagram (where $\Delta_Y: Y \rightarrow Y \times Y$ is the diagonal):
$$\begin{tikzcd}[column sep=0.5em]
& & X \times_Z \cL_\phi Y_X \simeq \cL_\phi  X \arrow[dl] \arrow[dr] & & \\
& X \arrow[dl, "p"'] \arrow[dr, "\Delta_{X/Y}"] & & \cL_\phi Y_X\simeq X \utimes{(f, f \circ \phi_X), Y \times Y, \Delta_Y} Y \arrow[dl, "\Gamma_\phi \times \mathrm{id}_Y"'] \arrow[dr, "f \times \mathrm{id}_Y"] & \\
\pt & & Z = X \times_Y X = (X \times X) \utimes{(f, f), Y \times Y, \Delta_Y} Y & & \cL_\phi  Y,
\end{tikzcd}$$
i.e. $[\QCoh(X), \phi_{X*}] \simeq \cL_\phi f_* \cO_{\cL_\phi X}$.
\end{proof}

\end{document}